\newif\ifpreprint
\ifdefined\buildjournal
  \preprintfalse
\else
  \preprinttrue
\fi

\ifpreprint
  \documentclass[11pt,reqno]{amsart}
\else
  \documentclass[aap]{imsart}
\fi

\RequirePackage{amsthm,amsmath,amsfonts,amssymb}
\RequirePackage[numbers]{natbib}
\RequirePackage[colorlinks,citecolor=red,urlcolor=black,hypertexnames=false]{hyperref}
\RequirePackage{graphicx}
\usepackage{mathrsfs}
\usepackage[textsize=tiny]{todonotes} 
\usepackage{algorithm}
\usepackage{comment}
\usepackage{algpseudocode} 
\usepackage{algorithmicx}     
\usepackage{booktabs, mathtools}
\usepackage{tikz}
\usetikzlibrary{arrows.meta, positioning, calc, fit, backgrounds}
\usepackage{csquotes}
\usepackage{thmtools}
\usepackage{afterpage}

\MakeOuterQuote{"}
\ifpreprint
\usepackage[T1]{fontenc}
\ifpreprint
\usepackage[left=1.2in, right=1.2in]{geometry}
\usepackage{lmodern}
\usepackage{fix-cm}
\fi
\fi

\ifpreprint
\ifdefined\startlocaldefs\else\let\startlocaldefs\relax\fi
\ifdefined\endlocaldefs\else\let\endlocaldefs\relax\fi

\theoremstyle{plain}
\newtheorem{theorem}{Theorem}[section]
\newtheorem{proposition}[theorem]{Proposition}
\newtheorem{lemma}[theorem]{Lemma}
\newtheorem{corollary}[theorem]{Corollary}

\theoremstyle{definition}
\newtheorem{assumption}[theorem]{Assumption}
\newtheorem{remark}[theorem]{Remark}
\newtheorem{definition}[theorem]{Definition}
\newtheorem{example}[theorem]{Example}
\else
\theoremstyle{definition}
\newtheorem{proposition}{Proposition}
\newtheorem{lemma}{Lemma}
\newtheorem{corollary}{Corollary}
\newtheorem{theorem}{Theorem}
\newtheorem{assumption}{Assumption}
\newtheorem{remark}{Remark}

\fi
\newtheorem*{proposition*}{Proposition} 
\newtheorem*{lemma*}{Lemma} 
\newtheorem*{corollary*}{Corollary} 
\newtheorem*{theorem*}{Theorem}

\newcommand{\E}{\mathbb{E}}   
\renewcommand{\P}{\mathbb{P}}
\newcommand{\Cov}{\mathrm{Cov}}
\newcommand{\Law}{\mathrm{Law}}
\newcommand{\KL}{D_{\mathrm{KL}}}
\newcommand{\Pp}{\mathcal{P}} 
\newcommand{\R}{\mathbb{R}}   
\newcommand{\N}{\mathbb{N}}   
\newcommand{\Tr}{\text{Tr}}   
\newcommand{\Trcov}{\Tr\Cov}   
\newcommand{\dd}{\mathrm{d}}
\newcommand{\opnorm}[1]{\left\|#1\right\|_{2}}
\newcommand{\pr}{\operatorname{pr}}

\newcommand{\map}{T}

\DeclareMathOperator*{\argmin}{arg\,min}

\startlocaldefs

\endlocaldefs

\ifpreprint
\newcommand{\papertitle}{{\Large Quantitative Wasserstein Propagation of Chaos for\\[3pt] Transport Ensemble Filters}}
\else
\newcommand{\papertitle}{Quantitative Wasserstein Propagation of Chaos for Transport Ensemble Filters}
\fi
\newcommand{\papershorttitle}{Propagation of Chaos for Transport Ensemble Filters}
\newcommand{\papershortauthors}{Jorgensen, Baptista, Hoffmann, and Marzouk}
\newcommand{\paperkeywords}{nonlinear filtering, ensemble Kalman filter, transport ensemble filters, mean-field, propagation of chaos, transport maps, Wasserstein distance, high-dimensional state estimation, interacting particle systems}

\ifpreprint
\title[\papershorttitle]{\papertitle}
\author[Jorgensen]{Frederic J.\ N.\ Jorgensen}
\address{Department of Mathematics, Massachusetts Institute of Technology, MA, USA}
\email{fjorgen@mit.edu}
\author[Baptista]{Ricardo Baptista}
\address{Department of Statistical Sciences, University of Toronto, Canada}
\email{r.baptista@utoronto.ca}
\author[Hoffmann]{Franca Hoffmann}
\address{Department of Computing and Mathematical Sciences, California Institute of Technology, CA, USA}
\email{franca.hoffmann@caltech.edu}
\author[Marzouk]{Youssef Marzouk}
\address{Laboratory for Information and Decision Systems, Massachusetts Institute of Technology, MA, USA}
\email{ymarz@mit.edu}
\subjclass[2010]{Primary 60G35, 60G25, 65C35; secondary 93E11}
\keywords{\paperkeywords}
\fi

\begin{document}

\ifpreprint
\else
\begin{frontmatter}
\title{\papertitle}
\runauthor{\papershortauthors}
\runtitle{\papershorttitle}

\begin{aug}

\author[A]{\fnms{Frederic J.\ N.}~\snm{Jorgensen}\ead[label=e1]{fjorgen@mit.edu}}
\author[D]{\fnms{Ricardo}~\snm{Baptista}\ead[label=e2]{r.baptista@utoronto.ca}}
\author[F]{\fnms{Franca}~\snm{Hoffmann}\ead[label=e3]{franca.hoffmann@caltech.edu}} 
\author[E]{\fnms{Youssef}~\snm{Marzouk}\ead[label=e4]{ymarz@mit.edu}}

\address[A]{Department of Mathematics, Massachusetts Institute of Technology, MA, USA\printead[presep={,\ }]{e1}}
\address[D]{Department of Statistical Sciences, University of Toronto, Canada\printead[presep={,\ }]{e2}}
\address[F]{Department of Computing and Mathematical Sciences, California Institute of Technology, CA, USA\printead[presep={,\ }]{e3}}
\address[E]{Laboratory for Information and Decision Systems, Massachusetts Institute of Technology, MA, USA\printead[presep={,\ }]{e4}}
\end{aug}
\fi

\begin{abstract} 
We develop a general probabilistic framework for analyzing propagation of chaos in \emph{transport ensemble filters} (TEFs), a broad class of interacting particle systems that are used to approximate the sequence of state distributions in hidden Markov models given a history of observations. This class of transport-based filtering algorithms includes the widely used ensemble Kalman filter (EnKF), based on affine
updates at each filtering step, as well as the ensemble stochastic map filter (EnSMF), which employs nonlinear updates. For this class, we identify the limiting mean-field dynamics. We then establish non-asymptotic, high-probability, pathwise Wasserstein convergence of the interacting particle system to an i.i.d.\ ensemble drawn from this mean-field limit at the Monte Carlo rate. Convergence to the mean-field law itself follows with the usual dimension-dependent empirical Wasserstein rate. The proof combines a synchronous coupling construction with stability of moments and tails under conditioning, together with quantitative estimates for the propagation of the underlying dynamics through the interacting particle system. Applying our theory to both the EnKF and the EnSMF yields the first non-asymptotic, high-probability convergence guarantees for TEFs.
\end{abstract}

\ifpreprint
\maketitle
\else
\begin{keyword}[class=MSC]
\kwdgroup[type=primary]{\kwd{60G35}\kwd{60G25}\kwd{65C35}}
\kwdgroup[type=secondary]{\kwd{93E11}}
\end{keyword}

\begin{keyword}
\kwd{nonlinear filtering}
\kwd{ensemble Kalman filter}
\kwd{transport ensemble filters}
\kwd{mean-field}
\kwd{propagation of chaos}
\kwd{transport maps}
\kwd{Wasserstein distance}
\kwd{high-dimensional state estimation}
\kwd{interacting particle systems}
\end{keyword}

\end{frontmatter}
\fi

\ifpreprint
\tableofcontents
\fi

\section{Introduction}
\label{subsection:filtering_problem}
\subsection{Problem Statement}
Given measurable maps $\Psi:\mathbb{R}^n \rightarrow\mathbb{R}^n$ and $h:\mathbb{R}^n\rightarrow \mathbb{R}^m$, we consider a sequence of random variables $\{(X^j, Y^{j+1})\}_{j\geq 0}$, with $(X^j, Y^{j+1})$ taking values in $\mathbb{R}^n\times \mathbb{R}^m$, defined by the dynamics and observation models: 
\begin{subequations} 
\label{eq:dynamics}
\begin{align}
    X^{j+1}&=\Psi\left(X^{j}\right)+\xi^{j},  \\ 
    Y^{j+1}& =h\bigl(X^{j+1}\bigr)+\eta^{j+1}, \label{eq:observation_model}
\end{align}
\end{subequations}
respectively, for all $j \in \mathbb{N}$, where $X^0 \sim \mu^0 \in \Pp(\R^n)$ is the initial condition, $\xi^j \sim \nu^j_\xi  \in \Pp(\R^n)$ for $j\ge 0$ are process noises, and $\eta^j \sim \nu_\eta^j \in\Pp(\R^m)$ for $j\ge 1$ are observation noises. The goal of the \emph{filtering problem} is to characterize the regular conditional probability distribution, 
\begin{equation}
\label{eq:filtered_distribution}
\mu_X^{a,j} \coloneqq  \Law(X^{j}|Y^{1:j}=y^{1:j}),
\end{equation}
for any $j \in \N_{\geq 1}$ and a given observation path $y^{1:j}\in(\mathbb{R}^{m})^j$. We refer to $\mu_X^{a,j} $ as the \emph{filtering} or \emph{analysis distribution}.
Throughout, we assume 
that $\{(X^j, Y^{j+1})\}_{j\geq 0}$ follow
a standard hidden Markov model (HMM), so that $X^0$, $\{\xi^j\}_j$, and $\{\eta^j\}_j$ are jointly independent random variables. With this HMM structure, $\mu_X^{a,j}$ admits a recursive dependence on the filtered law $\mu_X^{a,j-1}$ from the previous time step, a property that is exploited to develop efficient filtering algorithms~\cite{reich2015probabilistic, law2015data}.

The recursive update for the filtering distribution is defined using a two-step procedure. First, we define 
a joint measure on $(X,Y)$, called the \emph{joint forecast distribution}, obtained by propagating the analysis distribution $\mu^{a,j-1}_X$ through the dynamics and observation models in Equation~\eqref{eq:dynamics}:
\begin{equation}
      \mu_{XY}^{f, j} = \Law(X^j, Y^j|Y^{1:j-1} = y^{1:j-1}) \, . \label{eq:joint_forecast}
\end{equation}
(For \(j=1\), conditioning on \(Y^{1:0}\) denotes conditioning on the trivial \(\sigma\)-algebra.)
Formally, we define the dynamics operator $P^j: \Pp(\mathbb{R}^n)\rightarrow \Pp(\mathbb{R}^n)$ acting on measures
$$P^j(\pi) \coloneqq (\Psi_\sharp\pi) \star \nu^j_\xi,$$
where $\star$ denotes the convolution for Borel measures and $\Psi_\sharp \pi$ is the pushforward measure of $\pi$ through the map $\Psi$. Similarly, we define the observation operator $Q^j:\Pp(\mathbb{R}^n )\rightarrow \Pp(\mathbb{R}^n\times \mathbb{R}^m)$ acting on measures 
$$Q^j(\pi) \coloneqq  H_\sharp(\pi\otimes \nu_\eta^j),$$
where
\begin{equation*}
H: \R^n\times \R^m \rightarrow \R^n \times \R^m, \quad  H(x, z) = \left(x, h(x) + z\right),
\end{equation*} 
implements the observation step in Equation \eqref{eq:dynamics} by lifting the measure on $X$ to a joint measure on $(X,Y)$. The pushforward through $H$ implicitly carries out the convolution with the noise measure. Using these operators, the forecast distribution can be expressed in terms of the filtering distribution as 
\begin{equation}
\label{eq:dynamics_measure_perspective_joint}
    \mu_{XY}^{f,j} = Q^{j}P^{j-1}\mu_X^{a,j-1}.
\end{equation}

Second, we obtain the analysis distribution by conditioning the joint forecast distribution on a new observation $y^j$. That is,
\begin{equation}
\label{eq:analysis_step}
\mu_X^{a,j} = B_{y^j}\mu_{XY}^{f,j}.
\end{equation}
Here $B_y$ is the conditioning operator, defined by
\begin{equation}
\label{eq:cond_operator}
B: \R^m \times \Pp(\R^n\times\R^m)\to \Pp(\R^n), \quad     B(y,\pi)\coloneqq \pi_{X\mid Y=y},
\end{equation}
where \(\pi_{X\mid Y=y}\) denotes a version of the regular conditional distribution of \(X\) given a fixed \(Y=y\). We abbreviate \(B_y\pi\coloneqq B(y,\pi)\). Since regular conditional distributions are unique only \(\pi_Y\)-almost surely, statements involving \(B_Y\pi\) for a random observation \(Y\sim\pi_Y\) are understood \(\pi_Y\)-a.s., which is the setting in which the exact filtering recursion is used below. In summary, the analysis distribution evolves as
\begin{equation}
\label{eq:dynamics_measure_perspective}
    {\mu_X^{a,j} = B_{y^j} Q^{j}  P^{j-1} \mu_X^{a,j-1}},
\end{equation}
for all \(j \in \N_{\ge 1}\)~\cite{reich2015probabilistic, law2015data}.  

Implementing Equation \eqref{eq:dynamics_measure_perspective} for $\mu_X^{a,j}$ computationally poses two obstacles:
\begin{enumerate} \itemsep0pt 
\item Probability measures of the form $\mu^{a,j}_X$ or $\mu^{f,j}_{XY}$ cannot in general be represented exactly on a computer as members of a parametric class of distributions;
\item Computing $B_{y^j}\mu_{XY}^{f,j}$ requires high-dimensional integration (e.g., to compute normalizing constants), which is typically intractable for models as in Equation~\eqref{eq:dynamics} with high-dimensional variables.
\end{enumerate}
Two principal methods have been developed to address some of these challenges: particle filters and ensemble filters~\cite{gordon1993novel,doucet2001introduction,evensen2003ensemble}. While particle filters come with consistency guarantees in the large sample limit~\cite{crisan2002survey}, they suffer from the curse of dimensionality~\cite{bengtsson2008curse}. This work focuses on ensemble filters, which remain less well studied. 

In ensemble filters, the first problem is resolved by approximating $\mu^{a,j-1}_X$  with the empirical measure 
\begin{equation}
\label{eq:filtering_ensemble}
    \hat\mu_{X,N}^{a,j-1} \coloneqq  \frac{1}{N}\sum_{\ell=1}^N \delta_{x_{\ell}^{a,j-1}},
\end{equation}
given an ensemble of $N$ particles $x_{\ell}^{a,j-1} \in \R^n, \ell=1,\ldots,N$. 
(Throughout, hats on measures denote interacting empirical measures.)
The forecast step in Equation~\eqref{eq:dynamics_measure_perspective_joint}
is then implemented by evaluating the forecast model at each particle $x_{\ell}^{a,j-1}$ and adding independent realizations of the process noise and observation noise, i.e., a single draw of $Q^{j}P^{j-1}$, so that the ensemble size $N$ remains fixed. That is, the propagation of the ensemble is given by: 
\begin{subequations}
\label{eq:particles_dynamics}
\begin{align}
x_{\ell}^{f,j} &= \Psi(x_{\ell}^{a,j-1}) + \xi_\ell^{j-1}, 
& \xi_\ell^{j-1} \overset{\text{i.i.d.}}{\sim} \nu_\xi^{j-1}, 
& \quad \ell=1,\dots,N, \\
y_{\ell}^{f,j} &= h(x_{\ell}^{f,j}) + \eta_\ell^{j}, 
& \eta_\ell^{j} \overset{\text{i.i.d.}}{\sim} \nu_\eta^{j}, 
& \quad \ell=1,\dots,N.
\end{align}
\end{subequations}
This yields an approximate forecast distribution given by the empirical measure
\begin{equation}
\label{eq:forecast_ensemble}
\hat \mu_{XY, N}^{f,j} \coloneqq \frac{1}{N}\sum_{\ell=1}^N \delta_{(x_{\ell}^{f,j}, y_{\ell}^{f,j})}.
\end{equation}
Finally, the analysis step in Equation~\eqref{eq:analysis_step} is implemented in ensemble filtering by approximating the conditioning (``Bayes'') operator \(B_{y}\) using a transport map \(\map_{y}^{\pi}\) that acts on particles from the joint forecast distribution $\pi$.
The map implements an approximate conditioning operator $\tilde B_{y}$ that is close to the true Bayes operator $B_{y}$. 
Specifically, let $\kappa\in\Pp_2\left(\R^a\right)$ be a fixed latent distribution, naturally including deterministic maps. Then the map \(\map_{y}^{\pi}\) acts on $\pi$ and $\kappa$ as
$$
\tilde B_{y} \pi := \bigl(\map_{y}^\pi\bigr)_\sharp(\pi\otimes\kappa) \approx B_{y}\pi,
$$
for ``typical'' forecast joint distributions $\pi\in\Pp_2\left(\R^n\times\R^m\right)$.
A popular choice for $\map_y^\pi$ is the ensemble Kalman filter (EnKF)~\cite{evensen2003ensemble}, where the transport map is an \textit{affine} function of $x$ and $y$, defined as
\begin{equation}
\label{eq:enkf_beta_map}
\map_{y^j}^{\pi}(x,y) = x + \Cov(\pi)_{XY}
\Cov(\pi_Y)^{\dagger}(y^{j}-y) \, .
\end{equation}
Here $\Cov(\pi)_{XY}$ is the cross-covariance of $(X,Y)$ under the joint forecast $\pi$, $\Cov(\pi_Y)$
is the covariance of the $Y$-marginal $\pi_Y$ of $\pi$,\footnote{For the observation model in Equation~\eqref{eq:observation_model} with independent noise, the covariance of the $Y$-marginal is given by $\Cov(\pi_Y) = \Cov(h_{\sharp}\pi_X) + \Cov(\nu_\eta^j)$ where $\Cov(h_{\sharp}\pi_X)$ is the predictive covariance of $Y$ induced by $h$ and $\pi_X$ and $\Cov(\nu_\eta^j)$ is the observation-noise covariance at time $j$. We will analyze the EnKF map with this form of the marginal covariance $\Cov(\pi_Y)$ in Subsection~\ref{subsection:application_enkf}.}
and ${}^\dagger$ denotes the Moore--Penrose pseudoinverse. 
Other choices of $\map$---in particular, \textit{nonlinear} maps---are also possible; specific choices yield ensemble filtering algorithms such as the ensemble stochastic map filter (EnSMF)~\cite{spantini2022coupling}, the conditional mean filter (CMF)~\cite{hoang2021machine}, learning-enhanced ensemble filters~\cite{bach2025learning}, or other nonlinear filters~\cite{bao2025nonlinear}.
To include such possibilities in our theory, we keep the map $\map_{y^j}^{\pi}$ general. More formally, taking $\Pp_2(\R^{n}\times \R^m)$ endowed with the $2$-Wasserstein topology, we consider Borel measurable maps
\[
\map:\R^m\times \Pp_2(\R^{n} \times \R^{m})\times\R^{n}\times\R^{m} \times \R^a\to \R^n.
\]
For $y^\star\in\R^m$ and $\pi\in\Pp_2(\R^{n}\times\R^{m})$, we abbreviate the map as
\[
\map_{y^\star}^\pi(x,y,\omega) \coloneqq  \map\bigl(y^\star,\pi,x,y,\omega\bigr).
\]
Here $\omega \in \R^a$ is an auxiliary random variable drawn from $\kappa$ to allow for stochastic transport mechanisms as, for example, in \cite{bao2025nonlinear}.

In practice, the transport map is usually constructed from the empirical joint forecast $\hat\mu_{XY,N}^{f,j}$ and the new observation $y^j$. Then, the analysis particles are obtained by evaluating the map at each sample from the forecast distribution in Equation~\eqref{eq:forecast_ensemble}. That is,
\begin{equation}
\label{eq:particle_stochastic_updates}
x_{\ell}^{a,j}=  \map^{\hat \mu^{f, j}_{XY,N}}_{y^j} \bigl ( x_{\ell}^{f,j},  y_{\ell}^{f,j}, \omega_\ell^j \bigr), \quad \omega_\ell^{j} \overset{\text{i.i.d.}}{\sim} \kappa, \qquad \ell = 1,\dots,N.
\end{equation}
Consequently, the $x_{\ell}^{a,j}$ are generally \emph{not} i.i.d.; they interact through their common dependence on the forecast \(\hat \mu_{XY,N}^{f,j}\). Equation \eqref{eq:particle_stochastic_updates}  completes the recursion for \(\hat\mu_{X,N}^{a,j}\) in time and yields a fully specified filtering algorithm via Equations \eqref{eq:filtering_ensemble}, \eqref{eq:particles_dynamics}, \eqref{eq:forecast_ensemble}, and \eqref{eq:particle_stochastic_updates}. We summarize this scheme in Algorithm \ref{alg:TEF} and call this class of algorithms \textit{transport ensemble filters}. 
The term \textit{transport ensemble filters} should not be confused with ensemble \textit{transform} Kalman filters~\cite{bishop2001adaptive,takeda2024uniform}, or with more general ensemble transform methods and particle filters based on optimal transport~\cite{reich2013nonparametric,reich2015probabilistic,aljarrah2023optimal}. In our terminology, \textit{transport} refers broadly to the particle-wise analysis map in Algorithm~\ref{alg:TEF}; in this sense, ensemble transport filters~\cite{zeng2024ensemble} can be viewed as falling under our template, whereas the cited optimal-transport particle filters and ensemble transform Kalman filters are distinct notions and are not meant to be included here.

\begin{algorithm}[t]
\caption{Transport Ensemble Filter}\label{alg:TEF}
\begin{algorithmic}[1]
\State \textbf{Input:} ensemble size $N$ and sequentially acquired data $\{y^{j}\}_{j \geq 1}$.
\State \textbf{Initialization:} $x_{\ell}^{a,0} \overset{\text{i.i.d.}}{\sim} \mu^0,\qquad 1 \leq \ell \leq N.$
\For{$j = 1, 2, \ldots$}
    \State \textbf{Forecast: }
    \begin{align*}
    x_{\ell}^{f,j} & = \Psi(x_{\ell}^{a,j-1}) + \xi_\ell^{j-1}, \qquad \xi_\ell^{j-1} \overset{\text{i.i.d.}}{\sim} \nu_\xi^{j-1}, \qquad 1 \leq \ell \leq N, \\
    y_{\ell}^{f,j} &= h( x_{\ell}^{f,j}) + \eta_\ell^{j}, \qquad \eta_\ell^{j} \overset{\text{i.i.d.}}{\sim} \nu_\eta^{j}, \qquad 1 \leq \ell \leq N.
    \end{align*}
    \State \textbf{Analysis:}
    \[\hat \mu^{f, j}_{XY,N} =  \frac{1}{N}\sum\limits_{\ell = 1}^N\delta_{ (x_{\ell}^{f,j},  y_{\ell}^{f,j})} \]
    \[
x_{\ell}^{a,j}=  \map^{\hat \mu^{f, j}_{XY,N}}_{y^j} ( x_{\ell}^{f,j},  y_{\ell}^{f,j}, \omega_\ell^j), \qquad \omega_\ell^{j} \overset{\text{i.i.d.}}{\sim} \kappa, \qquad 1 \leq \ell \leq N.
\]
\EndFor
\State \textbf{Output:} Approximate filtering distribution $\hat\mu_{X,N}^{a,j} = \frac{1}{N}\sum\limits_{\ell = 1}^N\delta_{x_{\ell}^{a,j}} $ for $j = 1, 2, \ldots$
\end{algorithmic}
\end{algorithm}

This work addresses the central question of consistency of transport ensemble filtering algorithms: what is the limiting behavior of the empirical measure $\hat\mu_{X,N}^{a,j}$ generated by the interacting particle scheme in Equation~\eqref{eq:particle_stochastic_updates}, and how is it related to the true filtering distribution $\mu_X^{a,j}$ for each time $j$? Intuitively, under mild regularity of $\map$, increasing $N$ weakens the interaction among particles through the forecast $\hat\mu_{XY,N}^{f,j}$, suggesting convergence as $N\to\infty$. This phenomenon---particles becoming asymptotically independent---is known as propagation of chaos~\cite{kac1956foundations, sznitman2006topics}. The main objective of this paper is to make this intuition precise and to answer the following two questions:
\begin{enumerate}
    \item What is the mean-field law \(\tilde\mu_X^{a,j}\) to which \(\hat\mu_{X,N}^{a,j}\) converges in a suitable sense as the number of particles grows 
    $N\rightarrow \infty$?
    \item At what rate does \(\hat\mu_{X,N}^{a,j}\) converge to the mean-field law $\tilde\mu_X^{a,j}$?
\end{enumerate}

\begin{remark}
We work on \(\mathcal{P}_2(\R^{n}\times\R^{m})\) for two reasons. First, most linear and nonlinear transport ensemble filters (e.g., EnKF, EnSMF, CMF) are formulated at least in terms of first- and second-order moments (e.g., means, covariances, and/or covariances of nonlinear features).
Second, \(\mathcal{P}_2\) equips the space with the 2-Wasserstein distance \(W_2\), which provides a geometric and statistically meaningful metric for both (i) assessing prediction error of the empirical law \(\hat\mu_{X,N}^{a,j}\) and (ii) analyzing empirical measures. 
\end{remark}

\subsection{Related Work}
To the best of our knowledge, there are no results establishing mean-field convergence bounds for \emph{nonlinear} transport ensemble filters---i.e., when the analysis step is implemented via a nonlinear map \(\map^\pi_{y^j}\) in Equation \eqref{eq:particle_stochastic_updates}---as in, for example, the EnSMF or CMF~\cite{spantini2022coupling,hoang2021machine}. 

By contrast, for the specific case of EnKF-type algorithms in which $\map^\pi_{y^j}$ is an affine map, some results are available in the literature. The earliest mean-field convergence results for the EnKF are due to Mandel et al.~\cite{mandel2011convergence} and Le~Gland et al.~\cite{le2009large}. 
In the linear-Gaussian setting (linear $\Psi$, linear $h$, and Gaussian noise/initialization in Equation \eqref{eq:dynamics}), Mandel et al.\ prove convergence of the first two moments of the empirical analysis ensemble $\hat \mu_{X,N}^{a,j}$ as \(N\to\infty\). Thus, in expectation over the particle draws of $(\xi_\ell^{j-1}, \eta_\ell^{j})$, the ensemble statistics of $\hat \mu_{X,N}^{a,j}$ converge in \(L^{p}\), for all \(p\in[1,\infty)\), to those of the mean-field $\tilde\mu_X^{a,j}$ (which coincides with \(\mu_X^{a,j}\) in the linear-Gaussian setting). 
Le~Gland et al.\ extend this line of analysis: for Gaussian noise variables, initial condition $\mu^0$ with bounded moments of all orders, locally Lipschitz $\Psi$, and linear \(h\) in Equation~\eqref{eq:dynamics}, the authors identify the mean-field limit \(\tilde\mu_X^{a,j}\) and show that, for locally Lipschitz test functions with polynomial growth at infinity \(\varphi:\mathbb{R}^n\to\mathbb{R}\), we have 
\begin{equation}
\label{eq:previous_work_bound}
\left(\mathbb{E}\Bigl|\textstyle\int \varphi \, \dd\hat\mu_{X,N}^{a,j}-\int \varphi \, \dd \tilde\mu_{X}^{a,j}\Bigr|^p\right)^{1/p}
=O\left(N^{-1/2}\right)
\end{equation}
for any $p \geq 1$.  A more concise proof, with explicit dependence of the $O(N^{-1/2})$ constant on the relevant moments, is given in the recent work of Calvello et al.~\cite{calvello2024accuracy}. Law et al.~\cite{law2016deterministic} restate the result in Equation \eqref{eq:previous_work_bound} and broaden the class of admissible test functions $\varphi$ to include bounded functions as well. For the ensemble square root filter (ESRF), Kwiatkowski and Mandel~\cite{kwiatkowski2015convergence} obtain analogous mean-field convergence using similar techniques, and the recent work of Al-Ghattas and Sanz-Alonso~\cite{al2024non} proves the first dimension-independent and non-asymptotic $O\left(N^{-1/2}\right)$ convergence rates of the first two moments for the \emph{first} ESRF update (i.e., at time $j = 1$) in the linear-Gaussian setting. 
A related  but distinct thread studies \emph{state-estimation accuracy} of EnKF-type algorithms (rather than convergence of empirical measures), typically under dissipativity of the dynamics, linear and sufficiently informative observation operators, Gaussian noise variables, and Gaussian initial condition. 
Representative results include well-posedness and estimation accuracy for the perturbed-observation EnKF~\cite{kelly2014well}, accuracy and stability of  localized variants~(based on sparse regularization of the covariances in the EnKF map)~\cite{tong2015nonlinear,tong2023localized}, and long-time/stability analyses under structural assumptions on the dynamics and observations~\cite{sanz2024long}. 

There is also a parallel literature on the continuous-time limit of filtering, often called the ensemble Kalman--Bucy filter (EnKBF). These works use techniques different from the discrete-time setting to establish mean-field limits, stability, and propagation-of-chaos-type results for the associated McKean--Vlasov dynamics; see, e.g., Del~Moral and coauthors~\cite{del2017stability,del2018stability}, de~Wiljes et al.~\cite{de2018long}, and recent overviews such as Bishop and Del~Moral~\cite{bishop2023mathematical}. 


\subsection{Main Results}
Our results advance the theory of
ensemble filtering by developing a unified mean-field convergence framework for the broad class of transport ensemble filters defined in Algorithm~\ref{alg:TEF}. This framework includes the EnKF as a special case, thereby furthering EnKF theory, while allowing for genuinely nonlinear analysis maps, such as the EnSMF~\cite{spantini2022coupling}. For transports in this broad class, we identify the mean-field limit \(\tilde \mu^{a,j}_{X}\) in Equation \eqref{eq:mf_alg_measure_perspective}.
Our main technical result in Theorem~\ref{thm:ips_converges_to_iid_particles} proves convergence of the algorithmic ensemble $\hat \mu^{a,j}_{X,N}$ to this mean-field limit. More precisely,
we construct a coupling of the interacting analysis ensemble \(\hat\mu_{X,N}^{a,j}\) with an empirical measure \(\tilde\mu_{X,N}^{a,j}\) of i.i.d.\ mean-field particles such that for a finite time horizon $J\in\N$ we have the pathwise convergence
\begin{equation}
\label{eq:informal_1}
\sup_{1 \leq j \leq J}W_2\bigl(\hat\mu_{X,N}^{a,j},\tilde\mu_{X, N}^{a,j}\bigr)
  \lesssim N^{-1/2},
\end{equation}
with high probability (Theorem~\ref{thm:ips_converges_to_iid_particles}) and in expectation (Corollary~\ref{cor:ips_converges_to_iid_particles_exp}). In both results, the probability and expectation are taken over both the particle draws in Algorithm~\ref{alg:TEF} and the observation realizations \(y^{1:j}\). 

\begin{remark}
Our proofs also imply that with high probability we have the pathwise bound
$$
\sup_{1 \leq j \leq J}\sqrt{\frac{1}{N}\sum\limits_{\ell = 1}^N \left\|x_{\ell}^{a,j} - v_\ell^{a,j}\right\|_2^2}\lesssim N^{-1/2},
$$
where $v_\ell^{a,j}$ are mean-field i.i.d.\ particles. This is a tighter bound and closer to the classical way of stating propagation of chaos results in the literature~\cite{sznitman2006topics, chaintron2022propagation}. We will state our results by comparing measures, i.e., 
in the form of Equation~\eqref{eq:informal_1}, for ease of presentation.  
\end{remark}

Our results show that the interacting particle system is quantitatively close to an i.i.d.\ sample from the mean-field law at the \(N^{-1/2}\) rate.  For a confidence parameter \(k\ge2\), corresponding to events of probability at least \(1-1/k\), the component of the standard i.i.d.\ Wasserstein convergence rate~\cite{larsson2024concentration, fournier2015rate} is
\begin{equation}
\label{eq:standard_iid_wasserstein_rate}
\gamma_{k,N}^{p,n}
=
\begin{cases}
\left(\frac{\log k}{N}\right)^{1/n}
&\text{if } 1\le p < \frac n2,\\
\left(\frac{\log k}{N}\right)^{1/n}\bigl(\log(2+N)\bigr)^{2/n}
&\text{if } p = \frac n2,\\
\left(\frac{\log k}{N}\right)^{1/(2p)}
&\text{if } \frac n2 < p \le 2.
\end{cases}
\end{equation}
Combining Equation \eqref{eq:informal_1} with the rate in Equation~\eqref{eq:standard_iid_wasserstein_rate} yields a convergence result
\begin{equation}
\label{eq:informal_2}
\sup_{1 \leq j \leq J}W_p\bigl(\hat\mu_{X,N}^{a,j},\tilde\mu_{X}^{a,j}\bigr)
  \lesssim \gamma_{k,N}^{p,n}+\frac{(\log k)^{1+J/2}}{\sqrt N}.
\end{equation}
This bound holds with probability \(1-1/k\) over the particle realizations and observations for \(1 \le p \le 2\). In the high-dimensional regime \(n>2p\), the empirical-measure term reduces to \(N^{-1/n}\) and dominates the Monte Carlo interaction term for fixed \(k\) and \(J\). We state the result formally in Corollary~\ref{cor:finite_particle_converges_to_mean_field}. 
  
The main ideas and the full proofs of our results are developed in Sections~\ref{section:proof_intuitive} and~\ref{section:proof_formal}. Our proofs rely on a classical synchronous coupling argument, by driving $\hat\mu_{X,N}^{a,j}$ and $\tilde\mu_{X, N}^{a,j}$ with the same noise (Subsection \ref{subsection:idea_proof}), and the main technical challenge lies in handling the moments of the intermediate distributions (Subsections~\ref{subsection:moments_mean_field} and \ref{subsection:moments_ips}) as well as their dependence on the random observations \(y^{1:j}\) 
(Subsections~\ref{subsection:dynamics_control} and \ref{subsection:moments_mean_field}). In the process, we derive an interesting technical result in Proposition~\ref{prop:subG_stability_conditioning} on the high-probability stability of sub-Gaussianity under disintegration and conditioning. In order to prove this result, we introduce an equivalent characterization of sub-Gaussianity that extends the moment-generating-function definition to
non-centered 
random variables.
  
Section~\ref{sec:applications} presents applications of the results derived in the earlier sections to specific transport ensemble filters. Subsections~\ref{subsection:application_enkf} and~\ref{subsection:application_ensmf} illustrate how Equations \eqref{eq:informal_1} and \eqref{eq:informal_2} specialize to the EnKF and to the nonlinear transport analogue, the EnSMF, respectively. In the EnKF case, these results improve the convergence analyses in \cite{le2009large}, \cite{mandel2011convergence}, and \cite{kwiatkowski2015convergence} along several axes: first, our statements account for the randomness of the observation path, rather than conditioning on a fixed \(y^{1:J}\).
Second, we control (with high probability) the observational and algorithmic randomness, rather than only obtaining convergence in expectation. Third, our bounds are stated in Wasserstein metrics on \(\Pp_2( \R^n)\) (as opposed to single test-function averages), thereby directly quantifying distributional error. Fourth, we allow more general noise variables (i.e., sub-Gaussian, as opposed to Gaussian). Fifth, we replace the linear observation assumption by a Lipschitz condition on \(h\), accommodating nonlinear observation operators. Finally, our guarantees are non-asymptotic in \(N\), delivering the first explicit finite-sample rates, whereas all previous results for the EnKF held in the asymptotic regime.

\begin{remark}
Another way to relate these results to bounds of the form in Equation \eqref{eq:previous_work_bound} appearing in the literature is as follows. Since the $W_p$ distance for $p\ge 1$ dominates the $W_1$ distance, and $W_1$ admits a variational characterization in terms of Lipschitz functions $\varphi$ with constant $\mathrm{Lip}(\varphi)$, it follows from Equations \eqref{eq:informal_1} and \eqref{eq:informal_2} that the following \emph{uniform} convergence bounds hold:
\begin{align*}
\sup_{1 \le j \le J} \sup_{\mathrm{Lip}(\varphi)\le 1}
\Bigl|\textstyle\int \varphi \, \dd \hat\mu_{X,N}^{a,j}-\int \varphi \, \dd \tilde\mu_{X, N}^{a,j}\Bigr|
&\lesssim N^{-1/2}\\
\sup_{1 \le j \le J} \sup_{\mathrm{Lip}(\varphi)\le 1}
\Bigl|\textstyle\int \varphi \, \dd \hat\mu_{X,N}^{a,j}-\int \varphi \, \dd \tilde\mu_{X}^{a,j}\Bigr|
&\lesssim \gamma_{k,N}^{p,n}
   +\frac{(\log k)^{1+J/2}}{\sqrt N},
\end{align*}
where $\mathrm{Lip}(\varphi)$ denotes the Lipschitz constant of $\varphi$.
The first bound holds both with high probability and in expectation. The second one holds with high probability as displayed, and in expectation with \(\gamma_{k,N}^{p,n}\) replaced by a $k$-independent factor \(\bar\gamma_N^{p,n}\).
\end{remark}

\subsection{Notation and Basic Definitions}
\label{subsection:notation_definition}
Throughout the paper we use the conventions $\N=\{0,1,2,\ldots\}$ and $\N_{\ge 1}=\{1,2,\ldots\}$. Moreover, we work on finite-dimensional Euclidean spaces:  vectors are equipped with the Euclidean norm $\|x\|_{2}$ and matrices with the induced operator norm $\opnorm{A}=\sup_{\|x\|_{2}=1}\|Ax\|_{2}$.  All Lipschitz constants are taken with respect to these norms.   The {Moore--Penrose pseudoinverse} of a matrix $A$ is denoted by $A^{\dagger}$. For $u \in \N_{\ge 1}$, the identity map on $\R^{u}$ is $I_{u}$; more generally $I_{V}$ denotes the identity on a vector space $V$. For any Lipschitz function $f$ between Euclidean vector spaces, we define $\mathrm{Lip}(f)$ as the smallest Lipschitz constant of $f$.  For \(x \in \mathbb{R}^u\) and any integer \(1 \le k \le u\), we define the subvector $x_{1:k}  =  (x_1,x_2,\dots,x_k)  \in\mathbb{R}^k$.
For a sequence of  vectors $y^j \in \R^d$, $d\in \N_{\geq 1}$ we write $y^{1:J}=(y^{1},\dots,y^{J})\in(\R^{d})^{J}$.
All probability measures are defined on the Borel $\sigma$-algebra $\mathcal{B}(\R^{u})$ of the relevant space.  We use $\Pp(\R^{u})$ for all Borel probability measures on $\R^{u}$ and consider the subset of measures
\begin{align*}
  \Pp_{p}(\R^{u}) &=\Bigl\{\mu\in\Pp(\R^{u}) : \int\|x\|_{2}^{p} \, \dd \mu(x)<\infty\Bigr\}.
\end{align*}
A real random variable \(X\) is {sub-Gaussian} with parameter $K$ (cf.~Def.~2.5.5 of \cite{vershynin2018high}) if there exists \(K>0\) such that
\[
  \E\bigl(e^{X^2/K^2}\bigr)\le2.
\]
Equivalently, by introducing the {\(\psi_2\)-Orlicz norm}
\begin{equation} \label{eq:Orlicz}
  \|X\|_{\psi_2}
  :=
  \inf\Bigl\{t>0 : \E\bigl(e^{(X/t)^2}\bigr)\le2\Bigr\},
\end{equation}
we have that \(X\) is sub-Gaussian when \(\|X\|_{\psi_2}<\infty\). Note that the infimum in~\eqref{eq:Orlicz} is attained unless $X$ is almost surely $0$.  We extend the definition to probability measures through $\|\mu\|_{\psi_2} = \|X\|_{\psi_2}$ for $X \sim \mu$. The notion generalizes directly to vectors (cf.~\cite{vershynin2018high}): a random vector \(X\in\R^u\) is {sub-Gaussian} if
\[
  \|X\|_{\psi_2}
  \coloneqq \sup_{\|v\|_2=1}\|v^\top X\|_{\psi_2}
  <\infty.
\]
To take the sub-Gaussian norm of the Euclidean norm of a random vector we will write 
$$
\|X\|_{2,\psi_2} \coloneqq  \left\|\|X\|_2\right\|_{\psi_2},
$$
instead, and similarly for measures $\mu \in \Pp(\R^u)$, we write $\|\mu\|_{2,\psi_2} = \|X\|_{2,\psi_2}$ for $X\sim \mu$. 
We fix dimensions $n,m\in\N_{\ge1}$ corresponding to the state and observation spaces $\R^{n}$ and $\R^{m}$, respectively. For a joint law $\mu\in\mathcal{P}\!\bigl(\R^{n}\times\R^{m}\bigr)$, we define the coordinate projections
\[
  \pr_X:\R^{n}\times\R^{m}\to\R^{n},\quad\pr_X(x,y)=x,
  \qquad
  \pr_Y:\R^{n}\times\R^{m}\to\R^{m},\quad\pr_Y(x,y)=y.
\]
The {$X$-marginal} and {$Y$-marginal} of $\mu$ are given by the pushforward distributions
\[
  \mu_X \coloneqq  (\pr_X)_{\sharp}\mu,
  \qquad
  \mu_Y \coloneqq  (\pr_Y)_{\sharp}\mu.
\]
Moreover, we write $\mu_{XY}=\mu$ for the full joint distribution.  Because $\R^{n}$ and $\R^{m}$ are Polish spaces (hence standard Borel), the disintegration theorem  (e.g.,~\cite{kallenberg1997foundations}) ensures the existence of a {regular conditional probability} measure given by the map 
\[
  y \mapsto \mu_{X\mid Y=y}\in\mathcal{P}(\R^{n}),
\]
satisfying
\[
  \mu(A\times B)=\int_{B}\mu_{X\mid Y=y}(A) \, \dd \mu_{Y}(y),
  \qquad
  \forall A\in\mathcal{B}(\R^{n}),B\in\mathcal{B}(\R^{m}),
\]
where $y\mapsto\mu_{X\mid Y=y}(A)$ is Borel-measurable for every fixed $A$. We call $\mu_{X\mid Y=y}$ the {conditional distribution of $X$ given $Y=y$}; the map $y\mapsto\mu_{X\mid Y=y}$ is unique $\mu_{Y}$-a.s.
For random vectors \(X\in\R^{u_1}\) and \(Y\in\R^{u_2}\)  we write $\E(X)$ for \emph{expectation}, 
\[
  \Cov(X,Y)
  \coloneqq \E\,\!\bigl((X-\E X)(Y-\E Y)^{\mathsf T}\bigr)
  \in\R^{u_1\times u_2},
\]
for the cross-covariance matrix of \(X\) and \(Y\) (if it exists), and $\Cov(X) \coloneqq  \Cov(X,X)$ for the {covariance matrix}. In particular, note that $ \Cov((X,Y)) \in \R^{(u_1+u_2)\times(u_1+u_2)}$ is the covariance of the random variable $(X,Y)$. $\Law(X)$ is the distribution of $X$, the {trace} of a matrix is $\Tr(\cdot)$,  and we abbreviate
$\Trcov(X)\coloneqq \Tr\bigl(\Cov(X)\bigr).$ 
For \(\mu\in\Pp_1(\R^u)\), we write \(\E\mu\) for its mean and define the centered measure
\[
\mu-\E\mu \coloneqq (x\mapsto x-\E\mu)_\sharp\mu.
\]
Similarly, \(\|\mu-\E\mu\|_{2,\psi_2}\) means \(\|X-\E X\|_{2,\psi_2}\) and \(\|\mu\|_{2,\psi_2}\) means \(\|X\|_{2,\psi_2}\) for any \(X\sim\mu\), when finite. For \(\mu\in\Pp_2(\R^u)\), covariance-related quantities are defined by considering \(X\sim\mu\); in particular, \(\Cov(\mu)\coloneqq\Cov(X)\) and \(\Trcov(\mu)\coloneqq\Tr(\Cov(\mu))\). If the displayed Orlicz norm is finite, then $\Trcov(\mu)\le 2 \|\mu-\E\mu\|_{2,\psi_2}^2\le 8 \|\mu\|_{2,\psi_2}^2$.  For $\mu\in\Pp_{q}(\R^{u})$ and $q\ge1$ we use the {centered $q$-th moment}
\[
  \overline{M}_{q}(\mu)=
  \Bigl(\int\|x-\E\mu\|_{2}^{q}\,\dd \mu(x)\Bigr)^{1/q}.
\]
For $\mu,\nu\in\Pp_{p}(\R^{u})$ ($p\ge1$), the {$p$-Wasserstein distance} is
\[
  W_{p}(\mu,\nu)=
  \Bigl(\inf_{\pi\in\Pi(\mu,\nu)}\int\|x-y\|_{2}^{p} \, \dd \pi(x,y)\Bigr)^{1/p},
\]
where $\Pi(\mu,\nu)$ is the set of couplings with marginals $\mu,\nu$.  Tildes denote mean-field quantities, hats denote interacting empirical measures, and a subscript \(N\) denotes an empirical measure.
\section{Mean-Field Approximation and Proof Strategy}
\label{section:proof_intuitive}
\subsection{Deriving Mean-Field Dynamics}
\label{subsection:deriving_mean_field}
In the transport ensemble filter, the particles are coupled through the update map $\map$:  each particle's position after the analysis step depends not only on its own forecast position $(x_{\ell}^{f,j},y_{\ell}^{f,j})$ 
but also on the empirical joint forecast distribution $\hat\mu_{XY,N}^{f,j}$. This creates an $N$-particle interacting system.  The mean-field approximation is a standard way to handle such models:  instead of tracking all interactions explicitly, we approximate the random empirical laws  $\hat\mu_{XY,N}^{f,j}$ and  $\hat\mu_{X,N}^{a,j}$ with deterministic measures  $\tilde\mu_{XY}^{f,j}$ and  $\tilde\mu_{X}^{a,j}$ (given $y^{1:j}$) in the regime where $N$ becomes infinite.  In this infinite-ensemble limit, particles become independent with  common law, and their interaction through $\map$ is replaced by the response to a fixed mean-field. To make this idea more concrete, consider a representative mean-field particle $U^{j-1}\sim \tilde\mu_X^{a,j-1}$. By drawing fresh noise, 
\[
\xi^{j-1}\sim \nu_\xi^{j-1}, 
\qquad \eta^{j}\sim \nu_\eta^{j},
\]
the forecast step in the mean field stays the same, i.e.,
\[
(U^{f,j},Y^{f,j})
=\Bigl(\Psi(U^{j-1})+\xi^{j-1}, \,  h(\Psi(U^{j-1})+\xi^{j-1})+\eta^{j}\Bigr),
\]
in law or, equivalently,
$$
\tilde\mu_{XY}^{f,j} =Q^jP^{j-1} \tilde\mu_{X}^{a,j-1}.
$$
For the analysis step in Equation \eqref{eq:particle_stochastic_updates}, however, we replace the random empirical law by the deterministic forecast law $\tilde\mu_{XY}^{f,j}$.
This means that the mean-field particle is updated using the map
\[
U^{j}
=\map_{y^j}^{\tilde\mu_{XY}^{f,j}}\!\bigl(U^{f,j},Y^{f,j},\omega^{j}\bigr),
\]
where $\omega^{j}\sim \kappa$ is an independent auxiliary random sample and we assume that the expression above is well-defined (i.e., $\tilde\mu_{XY}^{f,j}\in\Pp_2\left(\R^n\times\R^m\right)$; we will verify this under sufficient regularity of Equation \eqref{eq:dynamics} in the next section). We define the mean-field approximate conditioning operator given $\nu \in \Pp_2\left(\R^n\times\R^m \right)$, whenever the following pushforward lies in \(\Pp_2(\R^n)\) (as will be ensured under the conditions established in the next section), as 
\begin{equation}
\label{eq:mean_field_operator}
\tilde B_{y^j}^\nu:\Pp_2(\R^{n}\times \R^{m})\to \Pp_2(\R^n),
\qquad
\tilde  B_{y^j}^\nu\pi\coloneqq  \bigl(\map_{y^j}^\nu\bigr)_\sharp\!\bigl(\pi\otimes \kappa\bigr),
\end{equation}
with the abbreviation \( \tilde B_{y^j} \pi\coloneqq  \tilde B_{y^j}^\pi\pi\). Using $\tilde B$, we rewrite the update on the law of $U^{j}$ as 
\[
\tilde\mu_X^{a,j}=\Law(U^{j}) = \tilde B_{y^j}\tilde\mu_{XY}^{f,j}.
\]
Then, collecting both steps from above, the derived limiting mean-field system is
\begin{subequations}
\label{eq:mf_alg_measure_perspective}
\begin{align}
\tilde\mu_X^{a,0} &= \mu^0,\\
\tilde\mu_X^{a,j} &= \tilde B_{y^j} Q^j P^{j-1}\tilde\mu_X^{a,j-1}.
\end{align}
\end{subequations}
This recursion mirrors the exact filtering recursion in Equation \eqref{eq:dynamics_measure_perspective}, 
with $B_{y^j}$ replaced by the approximate conditioning operator $\tilde B_{y^j}$ defined through $\map$.  The argument above is heuristic, and we will prove formally that $\hat\mu_{X,N}^{a,j}$ converges to the mean-field system $\tilde \mu^{a,j}_{X}$ in the next sections.
\subsection{Idea of the Proof}
\label{subsection:idea_proof}
Our goal is to show that the interacting particle system $\hat\mu_{X,N}^{a,j}$ converges to the mean-field system $\tilde \mu^{a,j}_{X}$ in $W_p$-distance.  To prove this, we apply the classic idea of a \emph{synchronous coupling}~\cite{mckean1967propagation, sznitman2006topics}: we construct an i.i.d.\ mean-field ensemble of size $N$, $(v_\ell^{a,j})_{\ell=1}^N$, whose empirical laws $(\tilde\mu^{f,j}_{XY,N},\tilde\mu^{a,j}_{X,N})$ are sampled from the mean-field recursion. When the observation path is random, this i.i.d.\ statement is understood conditional on that path; for readability, we omit this qualifier below. The key idea is to couple  $(\tilde\mu^{f,j}_{XY,N},\tilde\mu^{a,j}_{X,N})$ and $(\hat\mu^{f,j}_{XY,N},\hat\mu^{a,j}_{X,N})$ by driving them with the same randomness $\{(\xi_\ell^{j-1},\eta_\ell^j,\omega_\ell^j)\}_{\ell=1}^N$ at every step \(j\). This ensures that any discrepancy stems only from the interaction caused by the dependence of $\map$ on the empirical forecast distribution $\hat\mu_{XY,N}^{f,j}$, which we will refer to as an \emph{interaction term}. This allows us to make an inductive argument. For fixed $N$ and for each $j$, we draw i.i.d.\ initializations $\{x_{\ell}^{a,0}\}_{\ell=1}^N \sim \mu^0$ and i.i.d.\ noise samples $\{(\xi_\ell^{j-1},\eta_\ell^{j},\omega_\ell^{j})\}_{\ell=1}^N
\sim \nu_\xi^{j-1}\!\otimes\!\nu_\eta^{j}\!\otimes\!\kappa$,
independent across $j$. Then, we define the coupling as follows, letting $ v_\ell^{a,0} =  x_{\ell}^{a,0}$:
\begin{equation}
\label{eq:sync-coupling}
\begin{minipage}{.45\linewidth}
\center{interacting particle system}
\[
\begin{aligned}
x_{\ell}^{f,j}&=\Psi(x_{\ell}^{a,j-1})+\xi_\ell^{j-1}\\
y_{\ell}^{f,j}&=h(x_{\ell}^{f,j})+\eta_\ell^{j}\\
\hat\mu_{XY,N}^{f,j}&=\tfrac1N\sum_{\ell=1}^N\delta_{(x_{\ell}^{f,j},y_{\ell}^{f,j})}\\
x_{\ell}^{a,j}&=\map_{y^j}^{\hat\mu_{XY,N}^{f,j}}\!\big(x_{\ell}^{f,j},y_{\ell}^{f,j},\omega_\ell^{j}\big)\\
\hat\mu_{X,N}^{a,j}&=\tfrac1N\sum_{\ell=1}^N\delta_{x_{\ell}^{a,j}}
\end{aligned}
\]
\end{minipage}\hfill
\begin{minipage}{.45\linewidth}
\center{i.i.d.\ mean-field}
\[
\begin{aligned}
v_\ell^{f,j}&=\Psi(v_\ell^{a,j-1})+\xi_\ell^{j-1}\\
o_\ell^{f,j}&=h(v_\ell^{f,j})+\eta_\ell^{j}\\
\tilde\mu_{XY,N}^{f,j}&=\tfrac1N\sum_{\ell=1}^N\delta_{(v_\ell^{f,j},o_\ell^{f,j})}\\
v_\ell^{a,j}&=\map_{y^j}^{\tilde\mu_{XY}^{f,j}}\!\big(v_\ell^{f,j},o_\ell^{f,j},\omega_\ell^{j}\big)\\
\tilde\mu_{X,N}^{a,j}&=\tfrac1N\sum_{\ell=1}^N\delta_{v_\ell^{a,j}}.
\end{aligned}
\]
\end{minipage}
\end{equation}
To compare $\tilde\mu_{X,N}^{a,j}$ to the mean-field distribution $\tilde\mu_{X}^{a,j}$ defined in Equation~\eqref{eq:mf_alg_measure_perspective}, it is instructive to apply the triangle inequality
\begin{equation}
\label{eq:triangle_intuitive_1}
W_p(\hat\mu^{a,j}_{X,N},\tilde\mu^{a,j}_{X}) \leq W_p(\hat\mu^{a,j}_{X,N},\tilde\mu^{a,j}_{X, N}) +W_p(\tilde\mu^{a,j}_{X, N},\tilde\mu^{a,j}_{X}).
\end{equation}
Since $\tilde\mu^{a,j}_{X, N}$ consists of i.i.d.\ particles, the second term $W_p(\tilde\mu^{a,j}_{X, N},\tilde\mu^{a,j}_{X})$ can be explicitly controlled in expectation or with high probability, non-recursively, through a rate $\gamma(N)$ using modern Wasserstein convergence results~\cite{fournier2015rate}, as long as we can control higher moments of $\tilde\mu^{a,j}_{X}$.  $\gamma(N)$ typically suffers from the curse of dimensionality: for $p$-Wasserstein distances in $\R^n$, it scales as $\gamma(N)\asymp N^{-1/n}$, which is known to be sharp in general~\cite{fournier2015rate}. For this reason, simply controlling  $W_p(\hat\mu^{a,j}_{X,N},\tilde\mu^{a,j}_{X})$ conflates the interaction error with the slow i.i.d.\ sampling rate $W_p(\tilde\mu^{a,j}_{X,N},\tilde\mu^{a,j}_{X})$ in high dimensions. 

As a result, it is informative to separate the scale of the second term in the upper bound in Equation~\eqref{eq:triangle_intuitive_1} and to study the first term $W_p(\hat\mu^{a,j}_{X,N},\tilde\mu^{a,j}_{X,N})$ in isolation: it captures the algorithmic interaction error by quantifying how close the interacting ensemble is to behaving like an i.i.d.\ sample from the mean-field law. Given sufficient Lipschitz regularity of $\map$ (Assumption  \ref{assumption:algorithm_finite_particle}) and the dynamics model (Assumption \ref{assumption:dyn_assumptions}), we will show that it satisfies a self-recursion 
\begin{equation}
\label{eq:triangle_intuitive_2}
W_p(\hat\mu^{a,j}_{X,N},\tilde\mu^{a,j}_{X, N}) \leq C_1W_p(\hat\mu^{a,j-1}_{X,N},\tilde\mu^{a,j-1}_{X, N})   + C_2/\sqrt{N},
\end{equation}
where $C_1$ and $C_2$ are random variables. The second term enters this upper bound because of the discrepancy caused by the interaction measure in $\map$ between $\hat\mu^{a,j}_{X,N}$ and $\tilde\mu^{a,j}_{X, N}$. Then, Equation~\eqref{eq:triangle_intuitive_2} yields an inductive bound $W_p(\hat\mu^{a,j}_{X,N},\tilde\mu^{a,j}_{X,N})\lesssim N^{-1/2}$ with high probability and in expectation, and this is essentially the proof of the main result (Theorem \ref{thm:ips_converges_to_iid_particles}). As a direct corollary, we obtain the corresponding high-probability and expectation bounds for $W_p(\hat\mu^{a,j}_{X,N},\tilde\mu^{a,j}_{X})$ using the empirical-measure rates in Corollary \ref{cor:finite_particle_converges_to_mean_field}. The key ingredient in these proofs is to obtain   high-probability bounds on the random variables $C_1$ and $C_2$ in Equation \eqref{eq:triangle_intuitive_2}. To obtain such bounds, we need to control the moments and tails of the distributions $\mu^{f,j}_{XY}$, $\hat\mu^{f,j}_{XY,N}$, $\tilde\mu^{f,j}_{XY,N}$, and $\tilde\mu^{f,j}_{XY}$ as well as some of their differences (with high probability for the random empirical measures).

Before formulating these bounds, we make explicit one additional source of randomness: all the measures we defined (e.g., $\hat\mu_{X,N}^{a,j}$, $\tilde\mu_{X}^{a,j}$, $\mu^{a,j}_{X}$) depend implicitly on the observation path $y^{1:j}$. Under the model in Equation \eqref{eq:dynamics}, however, the observation sequence $Y^{1:j}$ itself is random, and this randomness should be taken into account when formulating the bounds with high probability or in expectation. Letting $J\in\N$ be a finite time horizon and using the stochastic process $Y^{1:j}$ given by Equation~\eqref{eq:dynamics}, we define
\begin{equation}
\label{eq:rho_def}
\rho^J \coloneqq  \Law(Y^{1:J}),
\end{equation}
and promote $y^{1:J}$ to a random draw \(Y^{1:J}\sim\rho^J\). This observation path is taken independent of the algorithmic particle, noise, and auxiliary variables used to construct the interacting and mean-field ensembles. Henceforth, all bounds will be stated in probability or in expectation with respect to this joint law. 

\section{Proving Convergence to the Mean-Field}
\label{section:proof_formal}
In this section, we will present and prove our main result, Theorem \ref{thm:ips_converges_to_iid_particles} (Subsection~\ref{subsection:convergence_intro}). The argument hinges on three ingredients: (i) high-probability control of moments of the forecast distribution under random observations (Subsection~\ref{subsection:dynamics_control}), (ii) analogous moment bounds for the mean-field dynamics (Subsection~\ref{subsection:moments_mean_field}), and (iii) concentration of empirical moments for the interacting system (Subsection~\ref{subsection:moments_ips}). After building up these technical preliminaries, we combine them to prove Theorem~\ref{thm:ips_converges_to_iid_particles} in Subsection \ref{subsection:convergence_proof}.

\subsection{Mean-Field Limit Statements}
\label{subsection:convergence_intro}
The main result of this section is the control of the distance between the interacting particles and the i.i.d.\ draws from the corresponding mean-field distribution as defined in Equation \eqref{eq:sync-coupling}. This result holds under Assumptions~\ref{assumption:dyn_assumptions} and~\ref{assumption:algorithm_finite_particle}, which are introduced in the following sections.

\begin{theorem}
\label{thm:ips_converges_to_iid_particles}
Fix $J\in\N_{\ge1}$, draw $Y^{1:J} \sim \rho^J$, and suppose Assumptions~\ref{assumption:dyn_assumptions} and ~\ref{assumption:algorithm_finite_particle} hold. Let
$\{\hat\mu_{X,N}^{a,j}\}_{j=1}^{J}$ be the analysis empirical measures generated by the interacting-particle system
and $\{\tilde\mu_{X, N}^{a,j}\}_{j=1}^{J}$ the corresponding i.i.d.\ mean-field sequence  as coupled in Equation \eqref{eq:sync-coupling}, both using observations $y^{1:J}=Y^{1:J}$. Then there exists a constant $C\coloneqq C(\mathcal{Q})>0$ depending only on the quantities
\[
\mathcal{Q} \coloneqq  (n,m,a,b,\sigma_{\max},\lambda_{\max},\sigma_X,
L_\Psi,L_h,C_{\mathrm{L}},L_y,e_{\mathrm{L}},
C_{\mathrm{est}},e_{\mathrm{est},1},e_{\mathrm{est},2},e_y,
\sigma_\kappa,J,L_g)
\]
such that for all $k\ge 2$, with probability at least $1-\frac1k$,
$$
W_2\bigl(\hat\mu_{X,N}^{a,j},\tilde\mu_{X, N}^{a,j}\bigr)\leq C \left(1 + \left( \frac{\log k}{\sqrt{N}}\right)^{p_j}\right)\frac{(\log k)^{1 + j/2}}{\sqrt{N}} \text{ for all } 1 \leq j \leq J,
$$
where $p_j \coloneqq(1+2e_\mathrm{L})^{J-1}\left(e_{\mathrm{est},1} + (j-1)\max\left(e_\mathrm{L}, e_{\mathrm{est},1} +\frac{1}{2}\right)\right)$.
In particular, if $N\geq (\log k)^2$, 
$$
\sup_{1 \leq j \leq J}W_2\bigl(\hat\mu_{X,N}^{a,j},\tilde\mu_{X, N}^{a,j}\bigr)\leq C\frac{(\log k)^{1 +J/2}}{\sqrt{N}},
$$
with the same probability.
The randomness in the expression $W_2\bigl(\hat\mu_{X,N}^{a,j},\tilde\mu_{X, N}^{a,j}\bigr)$  is over particle realizations and over the assumed values of the observations $Y^{1:J}$.
\end{theorem}

\begin{remark}
   The constant $C(\mathcal{Q})$ diverges to infinity as $J\to\infty$ and thus the bounds in Theorem~\ref{thm:ips_converges_to_iid_particles} only hold for finite time. Extending these bounds to infinite time horizons is of key interest for the analysis of filtering algorithms.
\end{remark}

As an immediate corollary, through the elementary equality for non-negative scalar random variables $X$ with finite expectation $\E(X) = \int_0^{\infty} \P(X > t)dt$, we can bound the expectation over particle realizations and observations. 

\begin{corollary}
\label{cor:ips_converges_to_iid_particles_exp}
Consider the same setting as in Theorem \ref{thm:ips_converges_to_iid_particles}. 
Then,
\[
\E \left(\sup_{1 \leq j \leq J}W_2 \bigl(\hat\mu_{X,N}^{a,j},\tilde\mu_{X, N}^{a,j}\bigr)\right)
\ \le\
\frac{C'(\mathcal{Q})}{\sqrt{N}}
\]
with a constant $C' \coloneqq C'(\mathcal{Q})>0$. 
\end{corollary}

\begin{remark}
Note that sub-Gaussian tail assumptions are only needed for the high--probability statements in Theorem \ref{thm:ips_converges_to_iid_particles}. For the bound
in Corollary~\ref{cor:ips_converges_to_iid_particles_exp} it would also suffice to assume
moment control of the driving noises and the initialization in Assumptions \ref{assumption:dyn_assumptions} and  \ref{assumption:algorithm_finite_particle}.
\end{remark}

Next, we would like to bound $W_p\bigl(\hat\mu_{X,N}^{a,j},\tilde\mu_X^{a,j}\bigr)$ for $1\le p \le 2$, for which we use the concentration result Proposition~\ref{cor:wp_concentration} in Appendix~\ref{subsection:sup_MFL} following Example~3.4 of \cite{larsson2024concentration} (we also refer to Theorem~2 of \cite{fournier2015rate}).
We use the high-probability empirical-measure rate \(\tilde\gamma_{k,N}^{p,d}\) defined in Proposition~\ref{cor:wp_concentration}. For the expectation bound, define the corresponding \(k\)-free empirical-measure rate
\[
\bar\gamma_N^{p,d} =
\begin{cases}
N^{-1/d}
&\text{if } 1\le p < \frac d2,\\
N^{-1/d}\bigl(\log(2+N)\bigr)^{2/d}
&\text{if } p = \frac{d}{2},\\
N^{-1/(2p)}
&\text{if } \frac d2 < p \le 2.
\end{cases}
\]

\begin{corollary}
\label{cor:finite_particle_converges_to_mean_field}
Consider the setting of Theorem \ref{thm:ips_converges_to_iid_particles} and the mean-field sequence $\{\tilde\mu_{X}^{a,j}\}_{j=1}^{J}$ as defined in Equation \eqref{eq:mf_alg_measure_perspective}  using the observation $y^{1:J}=Y^{1:J}$.  
Pick any $p\in[1,2]$. Then there exists a constant  $C''=C''(\mathcal{Q})>0$, depending only on $\mathcal{Q}$ and $p$, such that for all $k\geq 2$  and $N\geq (\log k)^2$, with probability at least $1-\frac1k$, 
\[
W_p\bigl(\hat\mu_{X,N}^{a,j},\tilde\mu_X^{a,j}\bigr)
\leq C''\left(\gamma_{k,N}^{p,n} +\frac{(\log k)^{1 +j/2}}{\sqrt{N}}\right) \text{ for all }1 \leq j \leq J.
\]
 Moreover,
\[
\E \left(\sup\limits_{1\le j\le J} W_p\bigl(\hat\mu_{X,N}^{a,j},\tilde\mu_X^{a,j}\bigr)\right)
 \le  C'' \bar\gamma_N^{p,n} .
\]
\end{corollary}

\begin{proof}
Proposition \ref{prop:mf_dynamics_higher_moments} states the sub-Gaussian bound for the forecast laws \(\tilde \mu_{XY}^{f,j}\), but its proof also controls the corresponding centered analysis variables \(\tilde X^{a,j}\) at each step of the same induction. Hence, for fixed observation paths \(y^{1:j}\), the analysis laws satisfy \(\left\|\tilde \mu^{a,j}_{X}-\E\tilde \mu^{a,j}_{X}\right\|_{2,\psi_2}\leq \tilde C_{\mathrm{subG}}\) for all $1\le j\le J$, after increasing the constant if necessary, with a constant uniform over observation paths.
Let us condition on \(Y^{1:J}\). Then \(\tilde\mu_{X,N}^{a,j}\) is an empirical measure of i.i.d.\ samples from the deterministic law \(\tilde\mu_X^{a,j}\), so Proposition~\ref{cor:wp_concentration} and a union bound over \(j=1,\ldots,J\) give
\[
\mathbb P\left(W_p(\tilde\mu_{X,N}^{a,j},\tilde\mu_X^{a,j})\le C\tilde\gamma_{k,N}^{p,n}\ \text{for all }1\le j\le J\,\middle|\,Y^{1:J}\right)\ge 1-\frac{1}{2k},
\]
after adjusting the constants. By the definition of \(\tilde\gamma_{k,N}^{p,n}\) and the elementary bound \(\max(a,b)\le a+b\),
\[
\tilde\gamma_{k,N}^{p,n}\le \gamma_{k,N}^{p,n}+\sqrt{\log k/N}.
\]
Since \(\sqrt{\log k/N}\lesssim_J (\log k)^{1+j/2}/\sqrt N\) for \(k\ge2\) and \(1\le j\le J\), the square-root term is absorbed by the interaction term. Integrating this conditional statement over \(Y^{1:J}\), intersecting it with the event from Theorem~\ref{thm:ips_converges_to_iid_particles}, and applying the triangle inequality yields the high-probability bound. The expectation bound follows similarly by conditioning on \(Y^{1:J}\), applying the \(k\)-free empirical-measure expectation rate \(\bar\gamma_N^{p,n}\) uniformly in the observation path, and combining with Corollary~\ref{cor:ips_converges_to_iid_particles_exp}. The latter term is dominated by \(\bar\gamma_N^{p,n}\) up to constants for \(p\in[1,2]\).
\end{proof}

\begin{remark}
It is well known that the empirical Wasserstein rate as in Corollary~\ref{cor:finite_particle_converges_to_mean_field} (i.e., \(N^{-1/d}\) in high dimensions) is sharp in general~\cite{dudley1969speed,dobric1995asymptotics}; see also Theorem~2 of \cite{fournier2015rate}. 
While there are regimes in which this curse of dimensionality can be alleviated~\cite{fournier2015rate, weed2019sharp}, identifying structural assumptions natural to filtering
that ensure improved rates lies beyond the scope of this work.
\end{remark}

As illustrated in Figure \ref{fig:decomposition}, our approach 
decomposes the total error into an {interaction} term---quantifying how far the finite-\(N\) interacting ensemble is from behaving as i.i.d.\ draws from the mean field (Theorem \ref{thm:ips_converges_to_iid_particles})---and an unavoidable {sampling} term that no ensemble-based algorithm can improve without additional structure (as in the proof of Corollary \ref{cor:finite_particle_converges_to_mean_field}). The interaction term scales at the optimal Monte Carlo rate \(N^{-1/2}\), while the i.i.d.\ term scales as \(N^{-1/n}\) in the high-dimensional regime \(n>2p\), reflecting the intrinsic curse of dimensionality for  \(W_p\) approximation.  
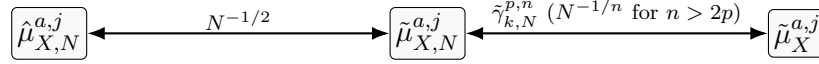
\begin{figure}[t]
\centering
\begin{tikzpicture}[
  >=LaTeX, node distance=40mm,
  meas/.style={draw, rounded corners=2pt, inner sep=2pt, fill=gray!5, align=center},
  lab/.style={font=\scriptsize, fill=white, inner sep=1pt}
]
\node[meas] (hat) {$\hat\mu_{X,N}^{a,j}$};
\node[meas, right=of hat] (iid) {$\tilde\mu_{X,N}^{a,j}$};
\node[meas, right=of iid] (mf) {$\tilde\mu_{X}^{a,j}$};
\draw[<->, thick] (hat) -- node[lab, above]{\(N^{-1/2}\) } (iid);
\draw[<->, thick] (iid) -- node[lab, above]{\(\tilde\gamma_{k,N}^{p,n} \ (N^{-1/n} \text{ for } n > 2p)\)} (mf);
\end{tikzpicture}
\caption{In $W_p$-distance, the interacting ensemble is very close to an i.i.d.\ ensemble from the mean-field law (Monte Carlo rate), but the i.i.d.\ sample can still be far from the mean-field law in the high-dimensional regime due to the curse of dimensionality.}
\label{fig:decomposition}
\end{figure}

\begin{remark}
While these theorems provide finite-horizon convergence guarantees for a broad class of interacting particle systems, natural directions for future work concern stronger concentration bounds and long-time behavior. The main challenge in the expectation analysis is to ensure that the
contraction induced by the analysis step quantitatively dominates the
expansive effect of the forecast, so that the net one-step map remains
contractive on average. Natural structural assumptions on the underlying dynamics that can allow for this dominance include
dissipation, observability, or ergodicity---and possibly stabilization mechanisms embedded in the analysis map~$\map$ (e.g., covariance inflation/localization in the linear case for $\map$~\cite{kelly2014well,tong2015nonlinear,tong2023localized, sanz2024long}).   Obtaining  time-uniform concentration bounds likely requires {simultaneously} controlling randomness across time, particles, and observations---rather than our current inductive approach that bounds each step in isolation. 
\end{remark}

\subsection{Controlling the Dynamics}
\label{subsection:dynamics_control}
To handle the randomness introduced by the observation path $y^{1:j}$, we develop control of the moments of the conditional forecast laws $\mu^{f,j}_{XY}$. These estimates will be a key ingredient for bounding the random Lipschitz  constants arising in the coupling argument. We will derive these bounds under the following assumptions on the dynamics and observation model in Equation \eqref{eq:dynamics}. 

\begin{assumption}[Dynamics]
\label{assumption:dyn_assumptions}
The dynamics satisfy the following properties:
\begin{enumerate} 
 \item  \emph{Regularity of Noise:} There are parameters $\sigma_{\max}, \lambda_{\max} \geq 0$ such that
 $$
 \|\eta^j\|_{2,\psi_2}  \leq \lambda_{\max}\quad (j\geq 1),
 \qquad
 \| \xi^j \|_{2,\psi_2} \leq \sigma_{\max}\quad (j\geq 0).
 $$
 \item \emph{HMM:} $X^0$, $\{\xi^j\}$, and $\{\eta^j\}$ are jointly independent.
 \item \emph{Regularity of Initialization: }$\|\mu^0\|_{2,\psi_2} \leq  \sigma_X$ for some $\sigma_X > 0$. 
    \item \emph{Lipschitzness of Dynamics: } $\Psi: \mathbb{R}^n \rightarrow \mathbb{R}^n$ is Lipschitz with constant $L_\Psi \geq 0$. 
    \item  \emph{Lipschitzness of Observation Map: } $h: \mathbb{R}^n \rightarrow \mathbb{R}^m$ is Lipschitz with constant $L_h \geq 0$. 
\end{enumerate}
\end{assumption}

Our main goal in this subsection will be to prove the following result.

\begin{proposition}
 \label{prop:process_covariance} 
Suppose Assumption \ref{assumption:dyn_assumptions} holds.  Fix some $J \in\N_{\ge1}$, $q\geq 1$,  and consider the random observation $Y^{1:J} \sim \rho^J$. Let $\{\mu_{XY}^{f,j} \}_{j= 1}^J$ be the random forecast distributions obtained from the observation $y^{1:J} = Y^{1:J}$ as defined in Equation \eqref{eq:joint_forecast}.  Then, there is a constant
\[
C_M\coloneqq  C_M(\sigma_X,\sigma_{\max}, \lambda_{\max}, L_\Psi, L_h, J, q)
\]
such that for all $k \geq 2$ and $1\leq q^\prime \leq q$,
\begin{align*}
\P\Bigl(&\overline M_{q^\prime}\left(\mu_{XY}^{f,j}\right) \leq C_M \log k
\text{ and }
\|Y^j - \E(Y^j|Y^{1:j-1})\|_2^2  \leq C_M{\log k},\\
&\hspace{7cm}\forall j \in \{1, \ldots, J\}\Bigr)
\geq 1 - \frac{1}{k}.
\end{align*}
\end{proposition}

To prove Proposition \ref{prop:process_covariance}, our strategy is to prove sub-Gaussianity of
\[
\|Y^j - \E(Y^j|Y^{1:j-1})\|_2
\]
and the norm of $\mu_{XY}^{f,j}$ with high probability over the random observation path \(Y^{1:j}\). 
Proving that under Assumption \ref{assumption:dyn_assumptions} the Euclidean norm of  $(X^j, Y^j)$ as defined in Equation \eqref{eq:dynamics} has sub-Gaussian tails is simple.

\begin{lemma}
    \label{lem:process_subG} 
Assume that $(X^j, Y^j)$ is defined as in Equation \eqref{eq:dynamics} and suppose that Assumption \ref{assumption:dyn_assumptions} holds. Let $J \in \N$. Then the norm of the centered process is sub-Gaussian for $1\leq j \leq J$, 
$$
\left\|(X^j, Y^j) - \E(X^j, Y^j)\right\|_{2,\psi_2} \leq C_{\text{subG}}
$$ with $C_{\text{subG}} \coloneqq C_{\text{subG}} (\sigma_X, \sigma_{\max}, \lambda_{\max}, L_\Psi, L_h, J)$. 
\end{lemma}

The proof of this lemma is a straightforward induction using standard properties of sub-Gaussian random variables; see Appendix~\ref{subsection:sup_moments_dyn_process}. 
The main difficulty in Proposition~\ref{prop:process_covariance} stems from conditioning: both the term  
$\|Y^j-\E(Y^j\mid Y^{1:j-1})\|_2$ and the forecast law $\mu_{XY}^{f,j}$ are random and $Y^{1:j}$-dependent. 
The sub-Gaussianity of $\|Y^j-\E(Y^j\mid Y^{1:j-1})\|_2$ is immediate from that of $Y^j$ and follows from the following proposition.

\begin{proposition}
\label{prop:subG_stability_cond_exp}
Let $X$ be a sub-Gaussian random vector. Consider any other random variable $Y$ on the same probability space. Then 
$$
\left\|\E(X|Y)\right\|_{2,\psi_2} \leq  \|X\|_{2,\psi_2}. 
$$
\end{proposition}

A short proof is included in Appendix \ref{subsection:sup_moments_dyn_process}. 
By contrast, controlling $\mu_{XY}^{f,j}$ is more subtle: even after conditioning on $Y^{1:j}$, it remains a   probability measure. To obtain usable bounds, we must understand how sub-Gaussianity behaves under conditioning and how such stability properties transfer to measure-valued objects.  We recall  the following basic fact from \cite{vershynin2018high}. 

\begin{proposition}[Prop.~2.5.2 of \cite{vershynin2018high}]
\label{prop:equiv_subg}
Let $X$ be a random variable. The following statements are equivalent.
\begin{enumerate}
    \item $X$ is sub-Gaussian, namely $K_1 \coloneqq  \|X\|_{\psi_2}< \infty$. 
    \item The tails of $X$ satisfy 
    $$
    \P\bigl(|X|\ge t\bigr)\leq 2\exp \bigl(-t^2/K_2^2\bigr)\quad \text{ for all }  t\geq0 
    $$
    for some constant $K_2 > 0$.
  \item  The moments of $X$ grow as 
  $$
\E(|X|^p)^{\frac{1}{p}}\leq K_3 \sqrt{p} \quad \text{ for all } p \geq 1
  $$
  for some constant $K_3 > 0$.
\end{enumerate}
Moreover, if $\E(X) = 0$, then all the above are also equivalent to the following fourth property:
$$
\E \exp (\lambda X) \leq \exp \left(\lambda^{2}K_4^2\right) \quad \text{ for all }\lambda \in \mathbb{R}.
$$
for some constant $K_4 > 0$. Further, there exists a universal constant $\tilde C_{\psi_2} \geq 1$ such that if Property $j\in\{1,2,3\}$ holds with parameter $K_j> 0$, then Property $i\in\{1,2,3\}$ also holds with parameter $K_i>0$ that is bounded as
$$
K_i \leq  \tilde C_{\psi_2}K_j
$$
for all $i, j \in \{1,2,3\}$. In the centered case, the same statement holds for all $i,j\in\{1,2,3,4\}$, where Property \(4\) is the moment-generating-function bound above.
\end{proposition}

It turns out that an alternative characterization of sub-Gaussian random variables, which extends the moment generating function condition to settings where $\E(X)$ is nonzero, will be useful for demonstrating the stability of sub-Gaussianity under conditioning.

\begin{proposition}
\label{prop:subg_mgf_away_from_0}
A random variable $X$ is sub-Gaussian if and only if there is a constant $K_5 > 0$ such that for any $|t|\geq 1/K_5$ it holds that
\begin{equation}
\label{eq:subG_mgf}
    \E\exp(tX) \leq \exp\left( {t^2K_5^2}\right). 
\end{equation}
This property is monotone: if $\tilde K_5 \ge K_5$ and $|t| \ge 1/\tilde K_5$, then 
\[
\mathbb{E}\exp(tX) \le \exp\left(t^{2}\tilde K_5^{2}\right).
\]
Further, there exists a universal constant $C_{\psi_2} \geq 1$, satisfying \(C_{\psi_2}\ge \tilde C_{\psi_2}\) for the equivalence constant \(\tilde C_{\psi_2}\) in Proposition~\ref{prop:equiv_subg}, such that \eqref{eq:subG_mgf} holds for some \(K_5>0\) with \(K_5\le C_{\psi_2}\|X\|_{\psi_2}\) for nonzero \(X\), and conversely any \(K_5>0\) satisfying \eqref{eq:subG_mgf} obeys \(\|X\|_{\psi_2}\le C_{\psi_2}K_5\).
\end{proposition}

A proof of this result is given in Appendix \ref{subsection:sup_moments_dyn_process}.   
Using this new characterization, we can prove the stability of sub-Gaussianity under conditioning. 

\begin{proposition}
\label{prop:subG_stability_conditioning}
Let $X$ be a sub-Gaussian random variable. Consider any other random variable $Y$ defined on the same probability space and taking values in a Euclidean space; let $\pi_{XY}$ be their joint measure and let $\pi_{X|Y=y}$ be a regular conditional measure. Then there is a universal constant $C_{\mathrm{cond}} > 0$ such that we can bound the sub-Gaussianity of $\pi_{X|Y=y}$ with high probability over realizations of $Y$: 
$$
\P\left(\left\|\pi_{X|Y=Y}\right\|_{\psi_2} \leq C_{\mathrm{cond}} (1 + c)\|X\|_{\psi_2} \right) \geq 1 -\exp({-c}), \;\forall c > 0.
$$
$\|\pi_{X\mid Y=y}\|_{\psi_2}$ is the Orlicz norm of the  regular conditional  distribution evaluated at $y$. Thus the expression above is a high-probability bound on the $Y$-measurable real-valued quantity $\|\pi_{X\mid Y=Y}\|_{\psi_2}$.
\end{proposition}

\begin{proof}
Pick $K_5 =C_{\psi_2}\|X\|_{\psi_2}$ as in Proposition \ref{prop:subg_mgf_away_from_0} such that Equation \eqref{eq:subG_mgf} holds. For any $\lambda \in \R$, the moment generating function of $\pi_{X|Y=y} \sim X^\prime $ is   $M_\lambda(y) \coloneqq \E\exp\left(\lambda X^\prime \right).$

\textit{Bound mgf for a fixed value:} Pick any $|\lambda|\geq 1/K_5$.   As a consequence of  the disintegration theorem and sub-Gaussianity of $X$, the following  is true:
\begin{align*}
   \ \E M_\lambda(Y) &=   \E\exp\left(\lambda X \right)  \; \leq  \;{\exp\left(K_5^2|\lambda|^2\right)}. 
\end{align*}
By Markov's inequality, for all $k\geq 1$,  $\P\big(M_\lambda(Y) \geq  k\E(M_\lambda(Y)) \big) \leq k^{-1}.$ 
Since we have an upper bound on the expectation, we plug in $k = \exp\left(cK_5^2|\lambda|^2\right)$ for some $c \geq 0$ and conclude 
$$
\P\left(M_\lambda(Y) \geq  \exp\left(K_{5,c}^2|\lambda|^2\right) \right) \leq \exp\left( -cK_5^2|\lambda|^2\right)
$$
for any $|\lambda| \geq 1/K_5$ and $K_{5,c}\coloneqq  \sqrt{1 + c}K_5$.

\textit{Bound on dyadic net:} Define the set and the event
$$
\Lambda = \{\pm 2^{j}/K_{5}  |j \in \N \}, \quad
 A_{c} = \left\{\exists \lambda \in \Lambda: M_\lambda(Y) \geq \exp\left(K_{5,c}^2|\lambda|^2\right) \right\}.
$$
Apply a union bound over the net, for any $c>0$
\begin{equation}
\label{eq:prob_ac_bound}
    \P\left(A_{c} \right) \leq  2\sum\limits_{j = 0}^\infty  \exp\left( - {c}2^{2j}\right) \;\leq\; 2\exp({-c})\left(1+\frac{1}{c\log 4}\right). 
\end{equation}
The upper bound in the sum above is derived in Lemma \ref{lem:simple_exponential_sum} (Appendix \ref{subsection:sup_moments_dyn_process}).

\textit{Extending outside the dyadic net:} Define the event
$$
 B_{c} = \left\{ \pi_{X|Y=Y} \text{ is not sub-Gaussian with parameter } \tilde K_{5,c}\right\}
$$
for $\tilde K_{5,c} = 2\sqrt{1+c}K_{5,c}$ and some choice of the regular conditional $\pi_{X|Y}$. The parameter we refer to in the definition of $B_c$ is the one as in Equation \eqref{eq:subG_mgf} of Proposition \ref{prop:subg_mgf_away_from_0}. The randomness in the event $B_{c}$ comes from the randomness of the conditioning variable $Y$. Therefore, equivalently,
\[
B_{c}
= \left\{ \exists |\lambda| \geq 1/ \tilde K_{5, c} \text{ s.t. } M_{\lambda} (Y) > \exp\left({\tilde K_{5,c}^2|\lambda|^2}\right)\right\}.
\]
We will show that $B_{c} \subseteq A_{c}$. Let $|\lambda| \geq 1/ \tilde K_{5, c}$ be such that
\[
M_{\lambda} (Y) > \exp\left({\tilde K_{5,c}^2|\lambda|^2}\right).
\]
In what follows, we split into the cases (1) $|\lambda| \geq 1/  K_{5}$, and (2) $1/K_{5} \geq |\lambda| \geq 1/ \tilde K_{5, c}$.
\begin{itemize}
    \item \textit{Case 1:} Assume that $|\lambda| \geq 1/  K_{5}$.
By the mean value theorem, there are $j \in \N, p \in [0, 1], s \in \{\pm 1\}$ such that we have the convex decomposition
$$
\lambda = p \cdot s2^{j}/K_5  + (1-p)\cdot s2^{j+1}/K_5.
$$
By log-convexity, this shows that 
\begin{align*}
  K_{5, c}^2\left(\frac{2^{j+1}}{K_5}\right)^2& \leq  \tilde K_{5,c}^2 |\lambda|^2  
< \log M_{\lambda}(Y)   \\
& \leq p \log M_{s2^{j}/K_5}(Y) + (1-p)\log M_{s2^{j+1}/K_5}(Y).
\end{align*}
By the pigeonhole principle, this means that either  $\log M_{s2^{j}/K_5}(Y) >  K_{5, c}^2\left(\frac{2^{j + 1}}{K_5}\right)^2$  or $\log M_{s2^{j + 1}/K_5}(Y) >  K_{5, c}^2\left(\frac{2^{j+1}}{K_5}\right)^2$, each of which  implies that we are on $A_{c}$.
\item \textit{Case 2:} Assume that $1/K_{5} \geq |\lambda| \geq 1/ \tilde K_{5, c}$. Pick $s= \text{sgn}(\lambda)$ and decompose convexly as 
$$
\lambda = \left(1-|\lambda| K_5\right)\cdot 0  + \left(|\lambda| K_5\right)\cdot  \frac{s}{K_5}.
$$
Log-convexity implies
\begin{align*}
\tilde K_{5,c}^2 |\lambda|^2  < \log M_{\lambda}(Y) &
 \;\leq\; \left(1-|\lambda| K_5\right)\cdot\log M_0(Y) +   \left(|\lambda| K_5\right) \cdot\log M_{\frac{s}{K_5}}(Y)\\
& \;=\; (|\lambda| K_5)\cdot\log M_{\frac{s}{K_5}}(Y).
\end{align*}
This means, in particular, that 
\begin{align*}
   2(1 + c)=  \frac{\tilde K_{5,c}}{K_5} \;\leq\; \frac{\tilde K_{5,c}^2 |\lambda|^2 }{|\lambda| K_5} \;<\; \log M_{\frac{s}{K_5}}(Y)
\end{align*}
which is equivalent to 
$$
M_{\frac{s}{K_5}}(Y)>  \exp\left(2(1+c)\right) \geq \exp\left(K_{5,c}^2\left|\frac{s}{K_5}\right|^2\right)
$$
and so we are on the event $A_{c}$. 
\end{itemize}
This concludes the proof of $B_{c} \subseteq A_{c}$. To conclude  the entire argument, define any  $\Delta C>0$ sufficiently large, set \(C_{\mathrm{cond}}\coloneqq 2C_{\psi_2}^2(1+\Delta C)\), use equivalence of sub-Gaussian definitions (Proposition \ref{prop:subg_mgf_away_from_0}) and apply $B_{c} \subseteq A_{c}$:
\begin{align*}
    \P\bigl(\|\pi_{X|Y=Y}\|_{\psi_2} > C_{\mathrm{cond}}(1 + c)\|X\|_{\psi_2} \bigr)
    & \leq  \P\bigl(\|\pi_{X|Y=Y}\|_{\psi_2}\\
    &\qquad > 2C_{\psi_2}^2(1 + c +\Delta C)\|X\|_{\psi_2} \bigr) \\
    & \leq \P(B_{c + \Delta C}) \\
    & \leq \P(A_{c + \Delta C}) \\
    &\leq \exp(-c).
\end{align*}
By inequality  \eqref{eq:prob_ac_bound}, the last line holds for sufficiently large $\Delta C>0$, independent of $c$ and $\|X\|_{\psi_2}$, completing the proof.
\end{proof}

With Proposition \ref{prop:subG_stability_conditioning} and Lemma \ref{lem:process_subG} it is straightforward to prove Proposition \ref{prop:process_covariance}. 

\begin{proof}[Proof of Proposition \ref{prop:process_covariance}]
Let $k\geq 2$. For \(j=1\), we interpret \(Y^{1:0}\) as the trivial \(\sigma\)-algebra, so that \(\E(Y^1\mid Y^{1:0})=\E Y^1\). Lemma \ref{lem:process_subG} shows that the Euclidean norm of the centered process satisfies
$$
\left\|(X^j, Y^j) - \E(X^j, Y^j)\right\|_{2,\psi_2} \leq C_{\text{subG}}
$$
for $1\leq j \leq J$.

\textit{Conditional expectation: } 
In particular, $\left\| Y^j - \E(Y^j)\right\|_2$ is sub-Gaussian with the same Orlicz-norm bound. Further,  note that 
$$
\|Y^j - \E(Y^j|Y^{1:j-1})\|_2 \leq \|Y^j - \E(Y^j)\|_2  + \|\E(Y^j|Y^{1:j-1})- \E(Y^j)\|_2.
$$
By Proposition \ref{prop:subG_stability_cond_exp}, $\|\E(Y^j|Y^{1:j-1})- \E(Y^j)\|_2$ is sub-Gaussian with the same Orlicz-norm bound. This shows  $ \left\|Y^j - \E(Y^j|Y^{1:j-1})\right\|_{2,\psi_2} \leq 2 C_{\text{subG}}$ for all  $1 \leq j \leq J$. Appealing to the  equivalent definitions of sub-Gaussianity in Proposition \ref{prop:subg_mgf_away_from_0}, we showed that 
$$
\P\left( \|Y^j - \E(Y^j|Y^{1:j-1})\|_2  \geq c_1\sqrt{\log k}\right) \leq \frac{1}{3kJ}
$$
for a suitable choice of  $c_1 \coloneqq c_1(\sigma_X,\sigma_{\max}, \lambda_{\max}, L_\Psi, L_h, J)$ since $C_{\text{subG}}$ is a constant depending only on $\sigma_X,\sigma_{\max}, \lambda_{\max}, L_\Psi, L_h, J$ according to Lemma \ref{lem:process_subG} .

\textit{Moments: }
Next, consider a version of the regular conditional $\mu_{X Y}^{f, j}$ and put  
$$
Z(y^{1:j-1}) \sim \mu_{X Y}^{f, j}(y^{1:j-1})\;  \forall y^{1:j-1} \in (\R^m)^{j-1}
$$
where we make the dependence on the observations $y^{1:j-1}$ explicit here. In particular, by disintegration $Z(Y^{1:j-1})$ is the same in law as $(X^j,Y^j)$.  
By definition, for every $q \geq 1$,  
\begin{align*}
    \overline M_q\left( Z(y^{1:j-1})\right)
    & = \left(\E\left\|Z(y^{1:j-1})- \E \left(Z(y^{1:j-1}) \right) \right\|_2^q\right)^\frac{1}{q} \\
    & \leq \underbrace{\left(\E\left\|Z(y^{1:j-1})- \E \left(X^j, Y^j \right) \right\|_2^q\right)^\frac{1}{q}}_{\text{I}}\\
    &\quad +\underbrace{ \left\|\E(X^j, Y^j)- \E \left(Z(y^{1:j-1}) \right) \right\|_2}_{\text{II}}.
\end{align*}
We proceed term-by-term.

\textit{Term I:} By the disintegration theorem, we have
\[
\Law\left(\|Z(Y^{1:j-1})- \E (X^j, Y^j)\|_2\right)
= \Law\left( \|(X^j, Y^j)- \E (X^j, Y^j)\|_2\right).
\]
Therefore, by Lemma \ref{lem:process_subG}, the scalar random variable $\|Z(Y^{1:j-1})- \E (X^j, Y^j)\|_2$ is sub-Gaussian with Orlicz norm bounded by $C_{\text{subG}}$. We now apply Proposition \ref{prop:subG_stability_conditioning} to this scalar norm, with conditioning variable $Y^{1:j-1}$. Thus, for all $c \geq 0$ it holds that  
$$
\left\|Z(y^{1:j-1})- \E (X^j, Y^j)\right\|_{2,\psi_2} \leq C_{\mathrm{cond}}(1+c) C_{\text{subG}}
$$  with probability $ 1 -  \exp\left( -c\right)$ in $y^{1:j-1}  = Y^{1:j-1}$.   Using the moment definition of sub-Gaussianity in Proposition \ref{prop:subg_mgf_away_from_0}, we have that with the same probability and for all $q\geq 1$,
$$
\left(\E\|Z(y^{1:j-1})- \E (X^j, Y^j)\|_2^q\right)^{\frac{1}{q}} \leq \sqrt{q} C_{\psi_2}C_{\mathrm{cond}}(1+c) C_{\text{subG}}. 
$$
Reminding ourselves that $k\geq 2$, put differently, there is a constant
\[
c_2 \coloneqq c_2(\sigma_X,\sigma_{\max}, \lambda_{\max}, L_\Psi, L_h, J)
\]
such that with probability $1 - \frac{1}{3kJ}$ in $y^{1:j-1}=Y^{1:j-1}$  
$$
\left(\E\|Z(y^{1:j-1})- \E (X^j, Y^j)\|_2^q\right)^{\frac{1}{q}} \leq c_2\sqrt{q} \log k.
$$
\textit{Term II:} Further, by Proposition \ref{prop:subG_stability_cond_exp},  $\left\|\E(X^j,Y^j)- \E \left(Z(Y^{1:j-1})\right)\right\|_{2}$ is sub-Gaussian with norm $C_{\text{subG}}$. Therefore,  there is a constant $c_3 = c_3(\sigma_X,\sigma_{\max}, \lambda_{\max}, L_\Psi, L_h, J)$  such that with probability $1 - \frac{1}{3kJ}$ in $y^{1:j-1} = Y^{1:j-1}$ it holds that 
$$
\left\|\E(X^j, Y^j)- \E \left(Z(y^{1:j-1}) \right) \right\|_2 \leq c_3 \sqrt{\log k}.
$$
When \(j=1\), \(Y^{1:0}\) is trivial and \((\R^m)^0\) is a singleton; then \(Z(Y^{1:0})\sim\Law(X^1,Y^1)\), so term II is zero and the same estimates apply.

\textit{Union Bound:} Taking a union bound, we showed that there is a constant
\[
C_M\coloneqq C_M(\sigma_X,\sigma_{\max}, \lambda_{\max}, L_\Psi, L_h, J, q)
\]
such that for all $k \geq 2$ and $1\leq q^\prime \leq q$,
\begin{align*}
\P&\left( \overline M_{q^\prime}\left(\mu_{XY}^{f,j}\right) \leq   C_M\log k \right.\\
\quad \text{and} \quad &\left.\|Y^j - \E(Y^j|Y^{1:j-1})\|_2^2  \leq C_M{\log k} \text{ for all }j \in \{1, \ldots, J\}\right) \geq 1 - \frac{1}{k}.
\end{align*}
\end{proof}

\begin{remark}
\label{rmk:f-sub-G}
Our argument in the proof of Proposition~\ref{prop:process_covariance} also shows that the random measure \(\mu_{X Y}^{f, j}\) is sub-Gaussian with high probability. 
\end{remark}

\subsection{Bounding Moments of Mean-Field Measure}
\label{subsection:moments_mean_field}
Next, having established control of the dynamical process (Equation \eqref{eq:dynamics}) and the latent stochasticity due to the observation $y^{1:j}$, we turn to bounding higher moments of the mean-field forecast law $\tilde\mu^{f,j}_{XY}$.  To this end, we impose the following assumption on the transport ensemble filter, as defined through $\map$ and $\kappa$ in Algorithm \ref{alg:TEF}. 

\begin{assumption}[Transport Ensemble Filter]
    \label{assumption:algorithm_finite_particle}
    We assume the following to hold.
\begin{enumerate} 
\item \emph{Lipschitz properties.}
    \begin{enumerate}
        \item \emph{Lipschitz after estimation:} There are constants $C_{\mathrm{L}}, e_{\mathrm{L}} \geq 0$ such that, for every $\mu \in \Pp_2\left(\R^n\times\R^m\right)$ and $y \in \R^m$, the map $\map_y^\mu:\R^n\times\R^m\times\R^a \to \R^n$ is Lipschitz continuous in its non-measure arguments with constant
        \[
        L_\map^\mu = C_{\mathrm{L}} \left(1  + \Trcov(\mu)^{e_{\mathrm{L}}}\right).
        \]
        \item \emph{Estimation stability:}  There are constants  $e_{\mathrm{est},1}, e_{\mathrm{est},2}, C_{\mathrm{est}} \geq 0 $ and a Lipschitz function $g:\R^{n}\times \R^{m} \rightarrow \R^b$ with Lipschitz constant $L_g$ such that for every  \( \mu , \nu \in \mathcal{P}_{2}(\mathbb{R}^n \times \R^m)\) and  all $x\in\R^n$, $y,y^\star\in\R^m$, $\omega\in\R^a$,
\begin{align*}
\left\|\map^{\nu}_{y^\star}(x,y, \omega) -  \map^{\mu}_{y^\star}(x, y, \omega)\right\|_2 & \leq C_{\mathrm{est}}   \left(1 +  \Trcov( \nu)^{e_{\mathrm{est},1}} + \Trcov(\mu)^{e_{\mathrm{est},2}}  \right)\\
&\cdot \left(\|\omega\|_2+ \|y - y^\star\|_2\right)\left\|\Cov(g_\sharp \nu) - \Cov(g_\sharp \mu)\right\|_2.
\end{align*}
\end{enumerate}
\item \emph{Sub-Gaussian auxiliary noise: } There is a finite constant $\sigma_\kappa \geq0$ such that  $\left\|\kappa\right\|_{2,\psi_2} \leq \sigma_\kappa$.
\item \emph{Transport sensitivity:} There are constants $L_y, e_y \geq 0$ such that for all
\[
\begin{gathered}
\nu \in \Pp_2\left(\R^n\times\R^m\right), \quad x\in\R^n, \quad y\in\R^m,\\
y^\star \in\R^m, \quad \omega \in \R^a,
\end{gathered}
\]
\begin{equation}
    \label{eq:y_lipschitz} \left\|\map^\nu_{y^\star}(x,y, \omega ) - x \right\|_2 \leq  L_y\left(1 +  \Trcov\left( \nu\right)^{e_y}\right) \left( \left\|y - y^\star \right\|_2+ \left\|\omega\right\|_2\right).
\end{equation}
\end{enumerate}
\end{assumption}

Under this assumption, the required moment bounds for $\tilde\mu^{f,j}_{XY}$ follow from a straightforward induction on $j$, which we prove in Appendix  \ref{subsection:aux_results_mf_moments}. 

\begin{proposition}
\label{prop:mf_dynamics_higher_moments}
Fix $J\in\N_{\ge1}$ and an observation path $y^{1:J}$. Suppose  Assumptions \ref{assumption:dyn_assumptions} and \ref{assumption:algorithm_finite_particle} hold. 
Then, there exists a constant
\[
\tilde C_{\mathrm{subG}}\coloneqq
\tilde C_{\mathrm{subG}}\big(L_\Psi,L_h,J,\sigma_X,\sigma_{\max},
\lambda_{\max},\sigma_\kappa,C_{\mathrm{L}},e_{\mathrm{L}}\big)
\]
such that
\[
\left\|\tilde\mu_{XY}^{f,j} -\E\tilde\mu_{XY}^{f,j}\right\|_{2,\psi_2} \le\ \tilde C_{\mathrm{subG}}\qquad\text{for all }1\le j\le J.
\]
\end{proposition}

As noted in Subsection~\ref{subsection:idea_proof}, another key step towards showing Theorem \ref{thm:ips_converges_to_iid_particles} is to control the gap between the update maps  $\map^{\hat \mu^{f,j}_{XY,N}}_{Y^j}$ and $\map^{\tilde \mu^{f,j}_{XY}}_{Y^j}$. Unavoidably, such bounds hinge on the ``innovation'' term $(\tilde Y_f^j - Y^j)$, which must be controlled with high probability.  Here $Y^j$ is generated by the true dynamics (Equation \eqref{eq:dynamics}) while  $\tilde Y_f^j$ denotes the mean-field forecast observation given the observation $y^{1:j-1}=Y^{1:j-1}$, i.e.,\ the $Y$-component of $(\tilde X_f^j,\tilde Y_f^j)\sim \tilde\mu^{f,j}_{XY}$ from Equation \eqref{eq:mf_alg_measure_perspective}. A priori, these two objects are unrelated, and there is no direct reason to expect the approximate and exact conditioning operators  $\tilde B_y$ and  $B_y$  to be close, so $(\tilde Y_f^j - Y^j)$ need not be small. The crucial leverage comes from the transport stability in  Assumption~\ref{assumption:algorithm_finite_particle} (3): the map 
$y^\star\mapsto \map^\nu_{y^\star}(x,y,\omega)$ varies smoothly in the observation argument. Applied with $y^\star=Y^j$ and $y=\tilde Y_f^j$, this provides exactly the mechanism we need to control $\left\|\E\left( \tilde Y^j_f - Y^j\left|Y^{1:j-1}\right.\right) \right\|_2$ with high     probability. 
By Proposition~\ref{prop:mf_dynamics_higher_moments}, after increasing constants if necessary, there is a universal constant \(C_{\mathrm{tr}}\) such that
\[
\Trcov(\tilde\mu_{XY}^{f,j}) \le C_{\mathrm{tr}}\tilde C_{\mathrm{subG}}^2
\qquad \text{for all } 1\le j\le J,
\]
uniformly over observation paths. Define the deterministic sequence by
\[
\bar t^1 = 2L_\Psi\sigma_X + 2\sigma_{\max},
\]
and, for \(j\ge 2\),
\[
\bar t^j
=
L_\Psi\left[
L_y\left(1+\left(C_{\mathrm{tr}}\tilde C_{\mathrm{subG}}^2\right)^{e_y}\right)
\left(L_h\bar t^{j-1}+2\lambda_{\max}+\sigma_\kappa\right)
+\bar t^{j-1}
\right]
+2\sigma_{\max}.
\]

\begin{proposition}
\label{prop:mf_y_stability}
Suppose Assumptions \ref{assumption:dyn_assumptions} and   \ref{assumption:algorithm_finite_particle} hold. Fix $J\in\N_{\geq 1}$, consider the dynamical process $\{(X^j, Y^{j+1})\}_{j \geq 0}$ as defined in Equation \eqref{eq:dynamics}, and let  $(\tilde X^j_f, \tilde Y^j_f) \sim \tilde \mu_{XY}^{f,j}$ for the mean-field algorithm $\{\tilde \mu_{XY}^{f, j}\}_{j=1}^J$ defined in Equation \eqref{eq:mf_alg_measure_perspective},  using the observation $y^{1:J}=Y^{1:J}$. Then, for any $k\geq 1$ with probability $1-\frac{1}{k}$ in $Y^{1:J}$
\begin{equation}
\label{eq:stability_mf_true}
\left\|\E\left( \tilde Y^j_f - Y^j\left|Y^{1:j-1}\right.\right) \right\|_2\leq L_hC_{\psi_2}\bar t^j\sqrt{\log(2kJ)} +2 \lambda_{\max} \text{ for all }j =1,\ldots, J. 
\end{equation}
\end{proposition}

\begin{proof}
Inequality \eqref{eq:stability_mf_true} does not depend on the particular coupling of $(X^j,Y^j)|Y^{1:j-1}$ and $(\tilde X_f^j,\tilde Y_f^j)|Y^{1:j-1}$, so we are free to choose whichever coupling we want, as long as the marginals are correct. The joint distribution of $\{(X^j, Y^j)\}_{j\geq 1}$ is described through the dynamics in Equation \eqref{eq:dynamics}, namely
\begin{align*}
    X^{j}&=\Psi\left(X^{j-1}\right)+\xi^{j-1}, \qquad \xi^{j-1} \sim \nu_\xi^{j-1} \text{ independently}  \\ 
    Y^{j}& =h\left(X^{j}\right)+\eta^{j}, \qquad \eta^{j} \sim \nu_\eta^{j} \text{ independently}
\end{align*}
for all $j\geq 1$ and $X^0 \sim \mu^0$. We couple $\left\{(\tilde X^j_f, \tilde Y^j_f)\right\}_{j\geq 1}$ through
\begin{align*}
    \tilde X^j_f &= \Psi\left(\tilde X^{j-1}\right)+\tilde \xi^{j-1},\qquad \tilde \xi^{j-1} \sim \nu_\xi^{j-1} \text{ independently} \\
     \tilde Y^{j}_f& =h\left(\tilde X^{j}_f\right)+\tilde\eta^{j}, \qquad \tilde\eta^{j} \sim \nu_\eta^{j} \text{ independently} \\
    \tilde X^{j} &= \map^{\tilde \mu^{f, j}_{XY}}_{Y^j} ( \tilde X^j_f,  \tilde Y^j_f, \omega^j), \quad \omega^{j} \sim \kappa \text{ independently}
\end{align*}
for all $j\geq 1$ and $\tilde X^0 \sim\mu^0$ independently.\\ \textit{Step I.} We begin the proof by showing, inductively, that
\[
\Delta^j \coloneqq \|X^j - \tilde X_f^j\|_2
\]
is sub-Gaussian with deterministic parameter \(\bar t^j\). We will use this fact later to control
\[
\left\|\mathbb{E}\left(\tilde Y_f^j - Y^j \mid Y^{1:j-1}\right)\right\|_2 .
\]
For $j = 1$, we have that $\left\|\Delta^1\right\|_{\psi_2} \leq 2L_\Psi\sigma_X + 2 \sigma_{\max}=\bar t^1$.   Let $j\geq 2$ and assume $\left\|\Delta^{j-1}\right\|_{\psi_2} \leq \bar t^{j-1}.$ Then,
\begin{align*}
    \Delta^j &= \left\|X^j -\tilde X^j_f\right\|_2  \;\leq\; L_\Psi\left\|\tilde X^{j-1}- X^{j-1}\right\|_2 +\left\|\xi^{j-1} \right\|_2 + \left\|\tilde \xi^{j-1}\right\|_2 . 
\end{align*}
Apply Equation \eqref{eq:y_lipschitz} to bound
\begin{align*}
    \left\|\tilde X^{j-1}- X^{j-1}\right\|_2 &= \left\|\map^{\tilde \mu^{f, j-1}_{XY}}_{Y^{j-1}} ( \tilde X^{j-1}_f,  \tilde Y^{j-1}_f, \omega^{j-1}) -   X^{j-1}\right\|_2 \\
     &\leq \left\|\map^{\tilde \mu^{f, j-1}_{XY}}_{Y^{j-1}} ( \tilde X^{j-1}_f,  \tilde Y^{j-1}_f, \omega^{j-1}) -   \tilde X^{j-1}_f\right\|_2 +   \left\|X^{j-1} - \tilde X^{j-1}_f\right\|_2\\
    & \leq L_y\left(1 +  \Trcov\left( \tilde \mu^{f, j-1}_{XY}\right)^{e_y}\right) \left( \left\|\tilde Y^{j-1}_f - Y^{j-1} \right\|_2+ \left\|\omega^{j-1}\right\|_2\right)+   \Delta^{j-1}.
\end{align*}
\textit{Step II. }Further, we have $\left\|\tilde Y^{j-1}_f - Y^{j-1} \right\|_2 \leq L_h\Delta^{j-1} + \left\|\tilde \eta^{j-1}\right\|_2 +\left\| \eta^{j-1} \right\|_2.$ Combining these estimates yields
\begin{align*}
   \left\|\Delta^j \right\|_{\psi_2}
   & \leq L_\Psi\Bigl(
   L_y\left(1 +  \Trcov\left( \tilde \mu^{f, j-1}_{XY}\right)^{e_y}\right)
   \\
   &\qquad\qquad\cdot
   \left(  L_h\left\|\Delta^{j-1}\right\|_{\psi_2}+2\lambda_{\max}+\sigma_\kappa\right)\\
   &\qquad\qquad
   +   \left\|\Delta^{j-1}\right\|_{\psi_2}\Bigr)+2\sigma_{\max} \\
     & \leq L_\Psi\Bigl(
     L_y\left(1 +  \left(C_{\mathrm{tr}}\tilde C_{\mathrm{subG}}^2\right)^{e_y}\right)
     \\
     &\qquad\qquad\cdot
     \left(  L_h\bar t^{j-1}+2\lambda_{\max}+\sigma_\kappa\right)
     + \bar t^{j-1}\Bigr)+2\sigma_{\max}
     =\bar t^j\,,
\end{align*}
and this completes the induction.

We have
\begin{align*}
 \left\|\E\left( \tilde Y^j_f - Y^j\left|Y^{1:j-1}\right.\right) \right\|_2& \leq  \E\left( \left\|\tilde Y^j_f - Y^j\right\|_2\left|Y^{1:j-1}\right.\right)  \\
 &=  \E\left( \left\|h\left(\tilde X^{j}_f\right)+\tilde \eta^{j} - h\left(X^{j}\right)-\eta^{j}\right\|_2\left|Y^{1:j-1}\right.\right) \\
 &\leq  \E\left( \left\|h\left(\tilde X^{j}_f\right) - h\left(X^{j}\right)\right\|_2\left|Y^{1:j-1}\right.\right) + 2\E\left( \left\|\eta^{j}\right\|_2\right) \\
 & \leq  L_h\E\left( \Delta^j\left|Y^{1:j-1}\right.\right)+ 2 \lambda_{\max}.
\end{align*}
Here we used \(\E\|\eta^j\|_2\le \|\eta^j\|_{2,\psi_2}\le\lambda_{\max}\), which follows from Jensen's inequality and the definition of \(\|\cdot\|_{2,\psi_2}\) in Equation~\eqref{eq:Orlicz}.
Step I and Proposition \ref{prop:subG_stability_cond_exp} prove that $\E\left( \Delta^j\left|Y^{1:j-1}\right.\right) $ is sub-Gaussian with parameter $\bar t^j$. Further, Property 2 in Proposition \ref{prop:equiv_subg} proves that 
$$
 \P\bigl(\E\left( \Delta^j\left|Y^{1:j-1}\right.\right)\ge \alpha\bigr)\leq 2\exp\!\bigl(-\alpha^2/K_2^2\bigr)\quad \text{ for all }  \alpha\geq0 
    $$
for $K_2 = C_{\psi_2}\bar t^j$. Therefore, for any $k\geq 1$, with probability $1 - \frac{1}{kJ}$,
    $$
    \E\left( \Delta^j\left|Y^{1:j-1}\right.\right)< C_{\psi_2}\bar t^j\sqrt{\log(2kJ)}
    $$
and with probability $1 - \frac{1}{k}$
    $$
    \left\|\E\left( \tilde Y^j_f - Y^j\left|Y^{1:j-1}\right.\right) \right\|_2  \leq  L_hC_{\psi_2}\bar t^j\sqrt{\log(2kJ)} +{2}\lambda_{\max} \text{ for all }j =1,\ldots, J. 
    $$
\end{proof}


\subsection{High-Probability Moment Bounds for the Interacting Particle System}
\label{subsection:moments_ips}
The final ingredient for controlling the random coefficients $C_1$ and $C_2$ in Equation \eqref{eq:triangle_intuitive_2}---and hence for proving convergence to the mean-field---is a high-probability bound on the second moment of the interacting finite-$N$ forecast law $\hat\mu_{XY,N}^{f,j}$. Although the forecast ensemble $\hat\mu^{f,j}_{XY,N}$ is \emph{not} i.i.d., to control its moments it suffices to understand how the centered second empirical moment of i.i.d.\ sub-Gaussian samples concentrates. In this and the next subsection, we will apply the following concentration bounds, which are simple consequences of standard concentration of measure results~\cite{wainwright2019high}.  

\begin{lemma}
\label{lem:mean-cov-subg}
Let $X_1,\dots,X_N\in\mathbb{R}^d$ be i.i.d.\ with mean $M=\mathbb{E}X_1$ and covariance $\Sigma=\E\left((X_1-M)(X_1-M)^\top\right)$. Assume $\big\|X_1\big\|_{2,\psi_2}\le K$. Write 
\[
\begin{aligned}
\overline X_N&=\frac1N\sum_{i=1}^N X_i,\\
\Sigma_N&=\frac1N\sum_{i=1}^N (X_i-\overline X_N)(X_i-\overline X_N)^\top,\\
\widetilde\Sigma_N&=\frac1N\sum_{i=1}^N (X_i-M)(X_i-M)^\top.
\end{aligned}
\]
There exist universal constants $c_1,c_2, c_3 \geq 0$ such that for all $u\ge \log 2$:
\begin{enumerate}
\item[\emph{(i)}] With probability at least $1-e^{-u}$,
\[
\|\overline X_N-M\|_2 \le c_1K\sqrt{\frac{du}{N}}.
\]
\item[\emph{(ii)}] With probability at least $1-e^{-u}$,
\[
\|\Sigma_N-\Sigma\|_2 \le c_2K^2d{\frac{u}{\sqrt{N}}}.
\]
\item[\emph{(iii)}] With probability at least $1-e^{-u}$,
$$
\left|\Tr(\widetilde\Sigma_N - \Sigma)\right|  \leq c_3K^2 d^2\frac{u}{\sqrt{N}},
\qquad
\left|\Tr(\Sigma_N - \Sigma)\right|  \leq c_3K^2 d^2\frac{u}{\sqrt{N}}.  
$$
\end{enumerate}
\end{lemma}

\begin{proof}
Let $Y_i\coloneqq X_i-M$.

\emph{(i)} By the monotonicity and centering properties of $\psi_2$ norms (see \cite{vershynin2025high}), for each $v\in\mathbb{S}^{d-1}$ the marginal $\langle v,Y_1\rangle$ is centered sub-Gaussian with parameter $2K$. Therefore, $\langle v,\overline X_N-M\rangle$ is  sub-Gaussian with parameter $2K/\sqrt{N}$. In particular, $\|\overline{X}_N - M\|_{\psi_2} \leq \frac{2K}{\sqrt{N}}$. By Exercise~3.39 of \cite{vershynin2025high},  this implies 
$$
\left\|\overline X_N-M\right\|_{2,\psi_2} \lesssim \sqrt{d/N}K.
$$
The claim then follows by Proposition~\ref{prop:equiv_subg} (Property 2), after choosing the universal constant $c_1$ sufficiently large.

\emph{(ii)} Because the $Y_i$ are $2K$-sub-Gaussian and $\widetilde\Sigma_N=\frac1N\sum_{i=1}^N (X_i-M)(X_i-M)^\top$,  Theorem~6.5 of \cite{wainwright2019high} gives, for all $s\ge0$ and universal $\tilde c_1, \tilde c_2, \tilde c_3\geq 0$,
\[
\mathbb{P}\!\left(
\|\widetilde\Sigma_N-\Sigma\|_2 >
\tilde c_14 K^2\left(\sqrt{\frac{d}{N}}+\frac{d}{N}+s\right)
\right)\le \tilde c_2e^{-\tilde c_3N\min\{s,s^2\}}.
\]
Choosing
\[
s= \tilde c_4\left(\sqrt{u/N}+u/N\right)
\]
with a sufficiently large constant $\tilde c_4$ gives
\[
\tilde c_2e^{-\tilde c_3N\min\{s,s^2\}} \leq e^{-u}.
\]
Therefore, with probability at least $1 - e^{-u}$
$$
\|\widetilde\Sigma_N-\Sigma\|_2 \le \tilde c_5K^2 \frac{d + u}{\sqrt{N}}
$$
for another universal constant $\tilde c_5$. Finally, consider the event from (i). Since  $\Sigma_N=\widetilde\Sigma_N-(\overline X_N-M)(\overline X_N-M)^\top$, the decomposition $\|\Sigma_N-\Sigma\|_2\le \|\widetilde\Sigma_N-\Sigma\|_2+\|\overline X_N-M\|_2^2$  concludes the proof.

\emph{(iii)} Finally, 
$$
\left|\Tr(\widetilde\Sigma_N - \Sigma)\right| \leq d \left\|\widetilde\Sigma_N - \Sigma\right\|_2 
$$
shows the first inequality. The other inequality follows from the same argument, utilizing the result from \emph{(ii)}.
\end{proof}

\begin{remark}
More sophisticated dimension-free bounds exist that are applicable to strictly sub-Gaussian random variables~\cite{koltchinskii2017concentration, hsu2012tail}. Such bounds are of interest for extending our argument to infinite-dimensional Hilbert spaces. 
\end{remark}

Using Lemma~\ref{lem:mean-cov-subg} as the i.i.d.\ input at each step and leveraging the Lipschitz stability provided by our assumptions, we obtain the desired high-probability bounds for the second moment of the interacting forecast law $\hat\mu_{XY,N}^{f,j}$.

\begin{proposition}
\label{prop:algorithm_second_moment}
Let $J,N\in\N_{\ge1}$, and let $\{\hat\mu_{XY,N}^{f,j}\}_{j=1}^{J}$ be as defined in Equation \eqref{eq:forecast_ensemble} with observations $y^{1:j}$ fixed.  Suppose Assumptions \ref{assumption:dyn_assumptions} and \ref{assumption:algorithm_finite_particle} hold. Define $r_j \coloneqq (1+2e_\mathrm{L})^{j-1}$. Then there exists a constant
\[
C_{\mathrm{ips}}\coloneqq C_{\mathrm{ips}}\big(L_\Psi,L_h,C_{\mathrm{L}},e_{\mathrm{L}},J,\sigma_X,\sigma_{\max},\lambda_{\max},\sigma_\kappa,n,m,a\big)
\]
such that, for all $k\ge2$,
\[
\overline M_2\big(\hat\mu_{XY,N}^{f,j}\big)\le  C_{\mathrm{ips}} \left(1+ \left(\frac{\log k}{\sqrt{N}}\right)^{r_j/2}\right) \text{ for } 1\leq j\leq J
\]
with probability at least $1-\frac{1}{k}$. 
\end{proposition}

\begin{proof}
We will repeatedly use the following two standard bounds: if $f$ is $L$-Lipschitz and $Z$ has second moments 
\begin{equation}\label{eq:two-sample}
\overline M_2\big(\Law(f(Z))\big)
\le L\overline M_2(\Law(Z))
\end{equation}
Further, for any joint random variable $(U,V) \in \R^n \times \R^m$,
\begin{equation}\label{eq:prod-moment-bound}
\overline M_2(\Law(U,V))\le \overline M_2(\Law(U))+\overline M_2(\Law(V)) .
\end{equation}
Introduce the independent noise-pair and seed samples
\[
\varepsilon_\ell^{j-1}\coloneqq (\xi_\ell^{j-1},\eta_\ell^{j}) \sim \nu_\xi^{j-1} \otimes \nu_\eta^{j},
\qquad
\omega_\ell^{j}\sim\kappa,
\]
and the empirical measures
\begin{gather*}
   \hat \mu^{a,0}_{X,N}=\tfrac1N\sum_{\ell=1}^N\delta_{x_{\ell}^{a,0}},\quad
\hat \mu^{j-1}_{\varepsilon,N}=\tfrac1N\sum_{\ell=1}^N\delta_{\varepsilon_\ell^{j-1}},\\
\hat \mu^{f,j}_{\omega,N}=\tfrac1N\sum_{\ell=1}^N\delta_{\omega_\ell^{j}}, \quad \hat \mu^{j-1}_{X\varepsilon,N} = \tfrac1N\sum_{\ell=1}^N\delta_{(x_{\ell}^{a,j-1}, \varepsilon_\ell^{j-1})}. 
\end{gather*}
By Assumption \ref{assumption:dyn_assumptions}, $x_{\ell}^{a,0}\sim\mu^0$ are sub-Gaussian with parameter $\sigma_X$,
$\xi_\ell^{j-1}$ and $\eta_\ell^{j}$ are sub-Gaussian with parameters $\sigma_{\max}$ and $\lambda_{\max}$, and $\omega_\ell^{j}$ has parameter $\sigma_\kappa$.
Further, note that for any two distributions $\mu, \nu \in \Pp_2\left(\R^u\right)$, $\overline M_2^2(\mu) = \Trcov(\mu) \leq |\Trcov(\mu) -\Trcov(\nu)| +\Trcov(\nu)$. Therefore, by \emph{(iii)} in Lemma~\ref{lem:mean-cov-subg} and a union bound over $j=1,\dots,J$, we have that the event
\[
\mathcal E_k\coloneqq \left\{
\max\left(\overline M_2(\hat \mu^{a,0}_{X,N}), \overline M_2(\hat \mu^{j-1}_{\varepsilon,N}), \overline M_2(\hat \mu^{f,j}_{\omega,N})\right)\le C_1 + C_2\sqrt{\frac{\log k}{\sqrt{N}}}\quad \forall j\le J
\right\}
\]
has probability at least $1-\frac{1}{k}$ for constants $C_1, C_2$ depending only on $\sigma_X$, $\sigma_{\max}$, $\lambda_{\max}$, $\sigma_\kappa$, $n$, $m$, $a$, $J$. While the constant $C_2$ bounds the concentration terms (e.g., $|\Trcov(\hat\mu^{a,0}_{X,N}) - \Trcov(\mu^0)|$), the constant $C_1$ controls the constant terms (e.g., $\Trcov( \mu^{0})$). Next, we condition on $\mathcal{E}_k$. 
Write
\[
d_{XY}^j\coloneqq \overline M_2(\hat\mu_{XY,N}^{f,j}),
\qquad
d_X^{j-1}\coloneqq \overline M_2(\hat\mu_{X,N}^{a,j-1}).
\]
Define the forecast map
$\Phi(x,\xi,\eta)\coloneqq \big(\Psi(x)+\xi, h(\Psi(x)+\xi)+\eta\big),$ 
and its Lipschitz constant inherited from Equation \eqref{eq:dynamics}:
$L_{PQ}\coloneqq (1+L_h)(1+L_\Psi).$ 
By the definitions in Equations \eqref{eq:particles_dynamics} and  \eqref{eq:forecast_ensemble},  \(
\hat\mu_{XY,N}^{f,j}=\Phi_{\sharp}\hat \mu^{j-1}_{X\varepsilon,N},
\)  so by Equations  \eqref{eq:two-sample} and \eqref{eq:prod-moment-bound},
\begin{equation}\label{eq:forecast-moment}
d_{XY}^j
=\overline M_2\big(\Phi_{\sharp}\hat \mu^{j-1}_{X\varepsilon,N}\big)
\;\le\; L_{PQ}\overline M_2(\hat \mu^{j-1}_{X\varepsilon,N})
\;\le\;  L_{PQ}\left(d_X^{j-1}+C_1 +C_2\sqrt{\frac{\log k}{\sqrt{N}}}\right).
\end{equation}
Let 
\(
\hat \mu^{f,j-1}_{XY\omega,N}
=\tfrac1N\sum_{\ell=1}^N\delta_{(x_{\ell}^{f,j-1},y_{\ell}^{f,j-1},\omega_\ell^{j-1})}.
\) 
Then 
\(
\hat\mu_{X,N}^{a,j-1}
=(\map_{y^{j-1}}^{\hat\mu_{XY,N}^{f,j-1}})_{\sharp}\hat \mu^{f,j-1}_{XY\omega,N}.
\) 
By Assumption \ref{assumption:algorithm_finite_particle} (1a),
\[
\mathrm{Lip}\Big(\map_{y^{j-1}}^{\hat\mu_{XY,N}^{f,j-1}}\Big)
\le C_{\mathrm{L}}\Big(1+\Trcov(\hat\mu_{XY,N}^{f,j-1})^{e_{\mathrm{L}}}\Big) = C_{\mathrm{L}}(1+(d_{XY}^{j-1})^{2e_{\mathrm{L}}}).
\]
Another application of Equations  \eqref{eq:two-sample} and \eqref{eq:prod-moment-bound} yields
$$
d_X^{j-1}
\le C_{\mathrm{L}}\big(1+(d_{XY}^{j-1})^{2e_{\mathrm{L}}}\big)
\overline M_2(\hat \mu^{f,j-1}_{XY\omega,N})
\le C_{\mathrm{L}}\big(1+(d_{XY}^{j-1})^{2e_{\mathrm{L}}}\big)
\big(d_{XY}^{j-1}+\overline M_2(\hat \mu^{f,j-1}_{\omega,N})\big).
$$
We can simplify this further to 
\begin{equation}\label{eq:analysis-moment}
d_X^{j-1}
\le C_{\mathrm{L}}\big(1+(d_{XY}^{j-1})^{2e_{\mathrm{L}}}\big)
\left(d_{XY}^{j-1}+C_1 +C_2\sqrt{\frac{\log k}{\sqrt{N}}}\right).
\end{equation}
Set $A\coloneqq  C_1 +C_2\sqrt{\frac{\log k}{\sqrt{N}}}$.  Insert Equation \eqref{eq:analysis-moment} into Equation \eqref{eq:forecast-moment} to get, for all $j\ge2$,
\begin{equation}\label{eq:main-recursion}
d_{XY}^{j}
\ \le\ L_{PQ}\Big(
C_{\mathrm{L}}\big(1+(d_{XY}^{j-1})^{2e_{\mathrm{L}}}\big)\big(d_{XY}^{j-1}+A\big)
+A\Big).
\end{equation}
For $j=1$, Equation \eqref{eq:forecast-moment} gives
\[
d_{XY}^{1}\leq 2 L_{PQ}A.
\]
Choosing sufficiently large  $\alpha\geq 1$ depending only on $L_{PQ}$ and $ C_{\mathrm{L}}$,  Equation \eqref{eq:main-recursion} implies that, for all \(j\ge2\),
\begin{align*}
1 +d_{XY}^{j} &\le  L_{PQ}\Big(
C_{\mathrm{L}}\big(1+(d_{XY}^{j-1})^{2e_{\mathrm{L}}}\big)\big(d_{XY}^{j-1}+A\big) +A\Big) + 1 \\
& \leq  \alpha\left(\big(1 +d_{XY}^{j-1}\big)^{2e_{\mathrm{L}}  + 1} + \big(1 +d_{XY}^{j-1}\big)^{2e_{\mathrm{L}}}  (1+A)\right)
\end{align*}
and $d_{XY}^{1}\le \alpha A$. Define $\beta\coloneqq 2e_{\mathrm{L}}  + 1$, $r_1\coloneqq1$, $r_j\coloneqq\beta r_{j-1}$ for $j\ge 2$, and $D_1\coloneqq\alpha$. Since $\alpha\ge1$, we have $1 +d_{XY}^{1} \le D_1 (1+A)^{r_1}$. Assume $1 +d_{XY}^{j-1} \le D_{j-1} (1+A)^{r_{j-1}}$. Then 
\begin{align*}
    1 +d_{XY}^{j} &\le 
    \alpha\left(\big(1 +d_{XY}^{j-1}\big)^{\beta} + \big(1 +d_{XY}^{j-1}\big)^{\beta-1}  (1+A)\right)\\
    &\le \alpha\left(D_{j-1}^\beta(1+A)^{\beta r_{j-1}} + D_{j-1}^{\beta-1} (1+A)^{(\beta-1)r_{j-1}+1}\right)\,.
\end{align*}
Since $r_{j-1}\ge 1$, we have $(\beta-1)r_{j-1}+1\le \beta r_{j-1}= r_j$. Hence, define
\[
D_j\coloneqq \alpha \left( D_{j-1}^\beta+D_{j-1}^{\beta-1}\right).
\]
Then $1 +d_{XY}^{j} \le D_j (1+A)^{r_j}$, and so this estimate holds for any $j\ge 1$ by induction. Define $C_{\mathrm{ips}}\coloneqq\max_{1\le j \le J} D_j$. Then $C_{\mathrm{ips}}$ depends only on $L_\Psi,L_h,C_{\mathrm{L}},e_{\mathrm{L}},J$, and it holds
$$
1 +d^j_{XY} \leq C_{\mathrm{ips}} (1+ A)^{r_j}
$$
for $ r_j \coloneqq (1 +2 e_{\mathrm{L}})^{j-1}$  and all $j\in\{1,\ldots J\}$.
Since $k\geq 2, N\geq 1$, we can redefine $ C_{\mathrm{ips}}$ such that
$$
d^j_{XY} \leq C_{\mathrm{ips}} \left(1+ \left(\frac{\log k}{\sqrt{N}}\right)^{r_j/2}\right)
$$
and introducing additional dependencies on $\sigma_X,\sigma_{\max},\lambda_{\max},\sigma_\kappa,n,m,a$. 
\end{proof}

\begin{remark}
Crucially, $C_{\mathrm{ips}}$ does \emph{not} depend on the particular choice of $y^{1:j}$ and will therefore hold for random choices of $y^{1:j}$ too.
\end{remark}


\subsection{Proof of Main Mean-Field Limit Result}
\label{subsection:convergence_proof}
Finally, we can combine our results from the previous sections to prove the main result, Theorem~\ref{thm:ips_converges_to_iid_particles}.

\begin{proof}[Proof of Theorem~\ref{thm:ips_converges_to_iid_particles}]
Whenever we say something depends on the allowed quantities, we mean it depends on $\mathcal{Q}$. We write $\lesssim$ if the constant in the bound depends only on these quantities. Define the empirical input measures (including the latent seed) as 
\[
\hat\mu^{j}_{XY\omega,N}\coloneqq \frac1N\sum_{\ell=1}^N \delta_{(x_{\ell}^{f,j},y_{\ell}^{f,j},\omega_\ell^{j})},
\qquad
\tilde\mu^{j}_{XY\omega,N}\coloneqq \frac1N\sum_{\ell=1}^N \delta_{(v_\ell^{f,j},o_\ell^{f,j},\omega_\ell^{j})}.
\]
Let $(\tilde X^j_f, \tilde Y^j_f) \sim \tilde \mu_{XY}^{f,j}$ denote samples from the mean-field algorithm $\{\tilde \mu_{XY}^{f, j}\}_{j=1}^J$ defined in Equation \eqref{eq:mf_alg_measure_perspective}, independent of all other random variables. We also abbreviate 
$$
\delta_{k,N} = \frac{\log k}{\sqrt{N}},
$$
as this term will come up repeatedly throughout our argument. 
Since we use the synchronous coupling as defined in Equation \eqref{eq:sync-coupling}, note that 
\begin{equation}
\label{eq:forecast_coupling}
\|(x_{\ell}^{f,j},y_{\ell}^{f,j})-(v_\ell^{f,j},o_\ell^{f,j})\|_2
\le L_{PQ}\|x_{\ell}^{a,j-1}-v_\ell^{a,j-1}\|_2,
\end{equation}
for $L_{PQ}\coloneqq (1+L_h)(1+L_\Psi)$. Set
\[
\varepsilon^j\coloneqq  \sqrt{\frac{1}{N}\sum\limits_{\ell= 1}^N\left\|x_{\ell}^{a,j} - v_\ell^{a,j}\right\|_2^2}.
\]
Then $\varepsilon^0=0$ by initialization. Our argument will be inductive over $j$. We start by fixing $k\ge2$ and it will be helpful to condition on an event  $\mathcal A_k$ that we define as the intersection of:
\begin{itemize}
\item The dynamical-process moment bounds
$$
\overline M_{2}^2\left(\mu_{XY}^{f,j}\right) \lesssim    (\log k)^2 \text{ and }   \|Y^j - \E(Y^j|Y^{1:j-1})\|_2^2  \lesssim {\log k}
\text{for all } j=1,\ldots,J,
$$
from Proposition \ref{prop:process_covariance}.
\item The mean-field stability bounds
$$\left\|\E\left( \tilde Y^j_f - Y^j\left|Y^{1:j-1}\right.\right) \right\|_2\lesssim  \sqrt{\log k}\text{ for all }j =1,\ldots, J,$$
as in Proposition~\ref{prop:mf_y_stability}, using Proposition~\ref{prop:mf_dynamics_higher_moments} to bound \(\Trcov\left(\tilde{\mu}_{X Y}^{f, j-1}\right)\). 
\item  Defining $\tilde \mu^{f,j}_{Y}$  as the $Y$-marginal of $\tilde \mu^{f,j}_{XY}$, noting that $o^{f,j}_\ell \sim \tilde \mu^{f,j}_{Y}$ i.i.d., and writing \(m_f^j\coloneqq\E(\tilde Y_f^j|Y^{1:j-1})\),
\begin{align*}
\frac{1}{N}\sum\limits_{\ell = 1}^N\left\|o^{f,j}_\ell - m_f^j\right\|_2^2
&= \Tr\left(\frac{1}{N}\sum\limits_{\ell = 1}^N
\bigl(o^{f,j}_\ell - m_f^j\bigr)
\bigl(o^{f,j}_\ell - m_f^j\bigr)^\top\right. \\
&\left.\qquad-\Cov(\tilde Y^{j}_f|Y^{1:j-1}) \right)
+\Tr\bigl(\Cov(\tilde Y^{j}_f|Y^{1:j-1})\bigr)\\
&\lesssim\delta_{k,N}   +1
\qquad \text{for all } j=1,\ldots,J,
\end{align*}
where Lemma \ref{lem:mean-cov-subg} controls the first term and Proposition \ref{prop:mf_dynamics_higher_moments} the second. 
\item The moment bounds $\overline M_2\big(\hat\mu_{XY,N}^{f,j}\big)\lesssim 1+ \delta_{k,N}^{(1+2e_\mathrm{L})^{J-1}/2}$ for all \(j=1,\ldots,J\), following from Proposition~\ref{prop:algorithm_second_moment}.
\item Concentration of  mean-field covariance and mean: for all $j = 1, \ldots, J$
\begin{align*}
\|\E\tilde \mu^{f,j}_{XY,N}-\E\tilde \mu^{f,j}_{XY}\|_2 &\lesssim \sqrt{\frac{\log k}{N }} \,,\\
\|\Cov(\tilde \mu^{f,j}_{XY,N})-\Cov(\tilde \mu^{f,j}_{XY}) \|_2&\lesssim \delta_{k,N}\,,
\end{align*}
using Proposition~\ref{prop:mf_dynamics_higher_moments} and Lemma~\ref{lem:mean-cov-subg}. We also include the corresponding direct concentration bound after applying \(g\) to the i.i.d.\ mean-field samples:
\[
\|\Cov(g_\sharp\tilde\mu^{f,j}_{XY,N})-\Cov(g_\sharp\tilde\mu^{f,j}_{XY}) \|_2\lesssim \delta_{k,N}
\qquad \text{for all }j = 1,\ldots,J,
\]
which follows from the same concentration lemma applied to \(g_\sharp\tilde\mu^{f,j}_{XY}\), using the Lipschitz continuity of \(g\).
\item Concentration of auxiliary noise: for all $j = 1, ..., J$,
$$
\sum\limits_{\ell = 1}^N\frac{1}{N} \|\omega^j_\ell\|_2^2 \lesssim  \delta_{k,N}+  1,
$$
through a similar decomposition as for $\frac{1}{N}\sum\limits_{\ell = 1}^N\left\|o^{f,j}_\ell - m_f^j\right\|_2^2 $, relying on Lemma \ref{lem:mean-cov-subg} and Proposition \ref{prop:subg_mgf_away_from_0}.
\end{itemize} 
By a union bound and our standing assumptions, we can tune the confidence levels in the constituent events so that their intersection satisfies $\P(\mathcal A_k)\;\ge\;1-\tfrac{1}{k}$ for all $k\ge 2.$ In particular, although each marginal statement is stated with probability \(1-\tfrac{1}{k}\), we may assign them probabilities \(1-\tfrac{1}{C k}\) (for a fixed constant \(C\)) so that the total failure probability remains \(\le \tfrac{1}{k}\) and the resulting $C$-factors are absorbed into harmless constants. The lower bound $2$ for $k$ is chosen arbitrarily in all proofs simply to ensure $\log k$  is lower bounded by some fixed positive constant. 
Now, let \(j\geq 1\) and use the triangle inequality to decompose
\[
\begin{aligned}
\varepsilon^j
&= \sqrt{\sum\limits_{\ell = 1}^N\frac{1}{N} \left\|\map_{y^j}^{\hat\mu_{XY,N}^{f,j}}(x_{\ell}^{f,j},y_{\ell}^{f,j},\omega_\ell^{j}) - \map_{y^j}^{\tilde\mu_{XY}^{f,j}}(v_\ell^{f,j},o_\ell^{f,j},\omega_\ell^{j})\right\|_2^2} \\
&\leq \underbrace{\sqrt{\sum\limits_{\ell = 1}^N\frac{1}{N} \left\|\map_{y^j}^{\hat\mu_{XY,N}^{f,j}}(x_{\ell}^{f,j},y_{\ell}^{f,j},\omega_\ell^{j}) -\map_{y^j}^{\hat\mu_{XY,N}^{f,j}}(v_\ell^{f,j},o_\ell^{f,j},\omega_\ell^{j})\right\|_2^2}}_{\text{Term A}}\\
&+ \underbrace{\sqrt{\sum\limits_{\ell = 1}^N\frac{1}{N} \left\|\map_{y^j}^{\hat\mu_{XY,N}^{f,j}}(v_\ell^{f,j},o_\ell^{f,j},\omega_\ell^{j}) - \map_{y^j}^{\tilde\mu_{XY}^{f,j}}(v_\ell^{f,j},o_\ell^{f,j},\omega_\ell^{j})\right\|_2^2}}_{\text{Term B}}
\end{aligned}
\]
\emph{Term A.} Note that for $j=1$, $(x_{\ell}^{f,1},y_{\ell}^{f,1}) =(v_\ell^{f,1},o_\ell^{f,1})$, and so Term A vanishes. For $j\ge 2$, Assumption~\ref{assumption:algorithm_finite_particle} (1a) and Equation \eqref{eq:forecast_coupling} yield
\begin{align*}
\text{Term A} & \le 
L_\map^{\hat\mu_{XY,N}^{f,j}}{\sqrt{\sum\limits_{\ell = 1}^N\frac{1}{N} \left\|(x_{\ell}^{f,j},y_{\ell}^{f,j},\omega_\ell^{j}) -(v_\ell^{f,j},o_\ell^{f,j},\omega_\ell^{j})\right\|_2^2}}\\
& \leq L_\map^{\hat\mu_{XY,N}^{f,j}}L_{PQ}\varepsilon^{j-1}.
\end{align*}
On $\mathcal A_k$, 
\[
L_\map^{\hat\mu_{XY,N}^{f,j}}
=C_{\mathrm{L}}\bigl(1+\Trcov(\hat\mu_{XY,N}^{f,j})^{e_{\mathrm{L}}}\bigr) \lesssim 1 +  \delta_{k,N}^{(1+2e_\mathrm{L})^{J-1}e_\mathrm{L}}.
\]
Combining both inequalities, 
\begin{equation}
\label{eq:termA_final}
\text{Term A} \lesssim \left(1 +  \delta_{k,N}^{(1+2e_\mathrm{L})^{J-1}e_\mathrm{L}}\right)\varepsilon^{j-1}.
\end{equation}
\emph{Term B.} Apply Assumption~\ref{assumption:algorithm_finite_particle} (1b) with $\nu=\hat\mu_{XY,N}^{f,j}$, $\mu=\tilde\mu_{XY}^{f,j}$ 
\begin{align*}
\text{Term B} &=    \sqrt{\sum\limits_{\ell = 1}^N\frac{1}{N} \left\|\map_{y^j}^{\hat\mu_{XY,N}^{f,j}}(v_\ell^{f,j},o_\ell^{f,j},\omega_\ell^{j}) - \map_{y^j}^{\tilde\mu_{XY}^{f,j}}(v_\ell^{f,j},o_\ell^{f,j},\omega_\ell^{j})\right\|_2^2} \\
& \lesssim  \left(1 +  \Trcov\left( \hat\mu_{XY,N}^{f,j}\right)^{e_{\mathrm{est},1}} + \Trcov(\tilde\mu_{XY}^{f,j})^{e_{\mathrm{est},2}}  \right)\\
& \cdot\left\|\Cov(g_\sharp \hat\mu_{XY,N}^{f,j}) - \Cov(g_\sharp \tilde\mu_{XY}^{f,j})\right\|_2\sqrt{\sum\limits_{\ell = 1}^N\frac{1}{N} 
\left(\|\omega^j_\ell\|_2+ \|o^{f,j}_\ell - y^j\|_2\right)^2}\\
& \lesssim \left(1 +  \delta_{k,N}^{(1+2e_\mathrm{L})^{J-1}e_{\mathrm{est},1}}\right)\left\|\Cov(g_\sharp \hat\mu_{XY,N}^{f,j}) - \Cov(g_\sharp \tilde\mu_{XY}^{f,j})\right\|_2\\
& \cdot\left(\sqrt{\sum\limits_{\ell = 1}^N\frac{1}{N} 
\|o^{f,j}_\ell - y^j\|_2^2}+ \sqrt{\sum\limits_{\ell = 1}^N\frac{1}{N} \|\omega^j_\ell\|_2^2}\right)
\end{align*}
We bound the covariance term and the innovation term separately. Note that 
\begin{align*}
\overline M_2^2(\tilde\mu_{XY,N}^{f,j}) &= \Tr\left(\Cov(\tilde\mu_{XY,N}^{f,j}) - \Cov(\tilde\mu_{XY}^{f,j})\right) + \Trcov(\tilde\mu_{XY}^{f,j}) \\
& \lesssim 1 +\delta_{k,N}
\end{align*}
and by inequality \eqref{eq:forecast_coupling}, $W_2\left(\hat\mu_{XY,N}^{f,j},\tilde\mu_{XY,N}^{f,j}\right) \lesssim \varepsilon^{j-1}$. 
Therefore, for the covariance term, we have by Lemma \ref{lem:cov_perturbation_W2} (Appendix \ref{subsection:fp_aux})
\begin{align*}
\left\|\Cov(g_\sharp \hat\mu_{XY,N}^{f,j})
- \Cov(g_\sharp \tilde\mu_{XY}^{f,j})\right\|_2
&\leq \left\|\Cov(g_\sharp \hat\mu_{XY,N}^{f,j})
-\Cov(g_\sharp \tilde\mu_{XY,N}^{f,j}) \right\|_2\\
&\quad +\left\|\Cov(g_\sharp \tilde\mu_{XY,N}^{f,j})
 - \Cov(g_\sharp \tilde\mu_{XY}^{f,j})\right\|_2 \\
&\lesssim \left(\overline M_2(g_\sharp \hat\mu_{XY,N}^{f,j})
+ \overline M_2(g_\sharp \tilde\mu_{XY,N}^{f,j})\right)\\
&\quad\cdot W_2\left(g_\sharp \hat\mu_{XY,N}^{f,j},
g_\sharp \tilde\mu_{XY,N}^{f,j}\right)\\
&\quad+\left\|\Cov(g_\sharp \tilde\mu_{XY,N}^{f,j})
 - \Cov(g_\sharp \tilde\mu_{XY}^{f,j})\right\|_2\\
&\lesssim \left(\overline M_2(\hat\mu_{XY,N}^{f,j})
+\overline M_2(\tilde\mu_{XY,N}^{f,j})\right)\\
&\quad\cdot W_2\left(\hat\mu_{XY,N}^{f,j},\tilde\mu_{XY,N}^{f,j}\right)\\
&\quad+\left\|\Cov(g_\sharp\tilde\mu_{XY,N}^{f,j})
 - \Cov(g_\sharp\tilde\mu_{XY}^{f,j})\right\|_2\\
&\lesssim  \left(1 +  \delta_{k,N}^{(1+2e_\mathrm{L})^{J-1}/2}\right)\varepsilon^{j-1} + \delta_{k,N}\,,
\end{align*}
where the last estimate holds on the event $\mathcal A_k$. 
We can bound the innovation term as:
\begin{align*}
\sqrt{\frac{1}{N}\sum\limits_{\ell = 1}^N\|o^{f,j}_\ell - y^j\|_2^2 }
&\leq \sqrt{\frac{1}{N}\sum\limits_{\ell = 1}^N\|o^{f,j}_\ell - m_f^j\|_2^2} 
 + \|m_f^j- \E(Y^{j}|Y^{1:j-1})\|_2 \\
&\quad +\|Y^j - \E(Y^{j}|Y^{1:j-1})\|_2 \\
&\lesssim   \sqrt{\delta_{k,N}}+ 1+\sqrt{\log k} + \sqrt{\log k} 
\lesssim \sqrt{\log k}.
\end{align*}
Finally, noting $ \sqrt{\sum\limits_{\ell = 1}^N\frac{1}{N} \|\omega^j_\ell\|_2^2}\lesssim \sqrt{\log k}$ on $\mathcal A_k$,  we can combine these estimates to 
\begin{align*}
\text{Term B} &\lesssim  \left(1 +  \delta_{k,N}^{(1+2e_\mathrm{L})^{J-1}e_{\mathrm{est},1}}\right)\left(\left(1 +  \delta_{k,N}^{(1+2e_\mathrm{L})^{J-1}/2}\right)\varepsilon^{j-1} + \delta_{k,N}\right) \sqrt{\log k} \\
& \lesssim \left(1 +  \delta_{k,N}^{(1+2e_\mathrm{L})^{J-1}(e_{\mathrm{est},1} +\frac{1}{2})}\right)\sqrt{\log k}\varepsilon^{j-1}  + \left(1 +  \delta_{k,N}^{(1+2e_\mathrm{L})^{J-1}e_{\mathrm{est},1}}\right)\frac{(\log k)^{3/2}}{\sqrt{N}}.
\end{align*}
\emph{Closing the recursion.} Combining Term A and B yields
\[
\varepsilon^j \lesssim \left(1 +  \delta_{k,N}^{e_{\mathrm{rec},1}}\right)\sqrt{\log k}\varepsilon^{j-1} + \left(1 +  \delta_{k,N}^{e_{\mathrm{rec}, 2}}\right)\frac{(\log k)^{3/2}}{\sqrt{N}}\text{ for } j=1,\ldots,J,\qquad \varepsilon^0=0
\]
for $e_{\mathrm{rec}, 1} \coloneqq  (1+2e_\mathrm{L})^{J-1}\max\left(e_\mathrm{L}, e_{\mathrm{est},1} +\frac{1}{2}\right)$ and $e_{\mathrm{rec}, 2} \coloneqq  (1+2e_\mathrm{L})^{J-1}e_{\mathrm{est},1}$.
Inducting this inequality over $1\leq j\le J$ and defining   $(q_1, p_1)= (3/2,e_{\mathrm{rec},2})$, $(q_{j+1}, p_{j+1}) =  (q_j +\frac{1}{2}, p_j +e_{\mathrm{rec},1})$ shows we can bound this  by 
\[
\varepsilon^j\ \lesssim \left(1 + \delta_{k,N}^{p_j}\right)\frac{(\log k)^{q_j}}{\sqrt{N}}\,.
\]
For the induction step \(j\ge2\), if $\varepsilon^{j-1}\lesssim \left(1 + \delta_{k,N}^{p_{j-1}}\right)\frac{(\log k)^{q_{j-1}}}{\sqrt{N}}$ holds, then we have
\begin{align*}
\varepsilon^j & \lesssim \left(1 +  \delta_{k,N}^{e_{\mathrm{rec},1}}\right)\sqrt{\log k}\varepsilon^{j-1} + \left(1 +  \delta_{k,N}^{e_{\mathrm{rec}, 2}}\right)\frac{(\log k)^{3/2}}{\sqrt{N}}\\
& \lesssim \left(1 +  \delta_{k,N}^{e_{\mathrm{rec},1}}\right)\sqrt{\log k}\left(1 + \delta_{k,N}^{p_{j-1}}\right)\frac{(\log k)^{q_{j-1}}}{\sqrt{N}} + \left(1 +  \delta_{k,N}^{e_{\mathrm{rec}, 2}}\right)\frac{(\log k)^{3/2}}{\sqrt{N}}\\
& \lesssim\left(1 + \delta_{k,N}^{p_{j-1} + e_{\mathrm{rec},1}}\right)\frac{(\log k)^{q_{j-1} +\frac{1}{2}}}{\sqrt{N}} + \left(1 +  \delta_{k,N}^{e_{\mathrm{rec}, 2}}\right)\frac{(\log k)^{3/2}}{\sqrt{N}}\\
& \lesssim\left(1 + \delta_{k,N}^{p_{j}}\right)\frac{(\log k)^{q_{j} }}{\sqrt{N}}\,.
\end{align*}
Finally, the coupling \(\frac1N\sum_{\ell=1}^N\delta_{(x_{\ell}^{a,j},v_\ell^{a,j})}\) is an admissible coupling between \(\hat\mu_{X,N}^{a,j}\) and \(\tilde\mu_{X,N}^{a,j}\), and the parameters $p_j, q_j$ can be explicitly expressed as
\begin{align*}
p_j \coloneqq(1+2e_\mathrm{L})^{J-1}\left(e_{\mathrm{est},1} + (j-1)\max\left(e_\mathrm{L}, e_{\mathrm{est},1} +\frac{1}{2}\right)\right)\,,\qquad q_j=1+j/2\,.
\end{align*}
Therefore, for all $1\le j \le J$,
\[
W_2\bigl(\hat\mu_{X,N}^{a,j},\tilde\mu_{X,N}^{a,j}\bigr)\le \varepsilon^j \lesssim \left(1 + \left( \frac{\log k}{\sqrt{N}}\right)^{p_j}\right)\frac{(\log k)^{1 + j/2}}{\sqrt{N}},
\]
which gives the stated bound.
\end{proof}

\section{Applications to Filtering Algorithms} \label{sec:applications}
In this section we apply the general mean-field convergence framework of
Section~\ref{section:proof_formal} to concrete ensemble filtering algorithms.
In Subsection~\ref{subsection:application_enkf} we derive mean-field convergence bounds for the ensemble Kalman filter (EnKF) that improve on previous work as described in the introduction. In Subsection~\ref{subsection:application_ensmf} we derive bounds for the ensemble stochastic map filter (EnSMF), a nonlinear extension of the EnKF.

\subsection{Mean-Field Convergence of the EnKF}
\label{subsection:application_enkf}
To develop the theory for the EnKF, we recall that it is a special case of Algorithm \ref{alg:TEF} for maps $\map$ that are affine functions of $x$ and $y$; see the summary in Algorithm~\ref{alg:EnKF}. 
For simplicity of exposition in this subsection, we
assume the distribution for the observation noise is independent of time, i.e.,
$\nu_\eta^{j} \equiv \nu_\eta$ for all $j\in\mathbb{N}_{\ge 1},$
and additionally assume the corresponding covariance matrix is positive definite:
\begin{equation}
\label{eq:enkf_observation_covariance}
\Lambda\coloneqq \Cov(\eta^j) \succeq \lambda_{\min} I_m
\quad \text{for some }\lambda_{\min}>0.
\end{equation}
Then, the EnKF update that is applied to samples from a joint forecast distribution $\pi$ is given by the affine map 
\begin{equation}
\label{eq:enkf_beta_map_gen}
\map_{y^j}^{\mathrm{K},\pi}(x,y) \coloneqq x + \Cov(\pi)_{XY}
\bigl(\Cov(h_{\sharp}\pi_X)+\Lambda\bigr)^{-1}(y^{j}-y),
\end{equation}
where the superscript \(\mathrm K\) is used to denote the Kalman map.
\begin{algorithm}[t]
\caption{Ensemble Kalman Filter (EnKF)}\label{alg:EnKF}
\begin{algorithmic}[1]
\State \textbf{Input:} ensemble size $N$ and sequentially acquired data $\{y^{j}\}_{j \geq 1}$.
\State \textbf{Initialization:} $x_{\ell}^{a,0} \overset{\text{i.i.d.}}{\sim} \mu^0,\qquad 1 \leq \ell \leq N.$
\For{$j = 1, 2, \ldots$}
    \State \textbf{Forecast: }
    \[
    x_{\ell}^{f,j} = \Psi(x_{\ell}^{a,j-1}) + \xi_\ell^{j-1}, \qquad \xi_\ell^{j-1} \overset{\text{i.i.d.}}{\sim} \nu_\xi^{j-1}, \qquad 1 \leq \ell \leq N.
    \]      \[
    y_{\ell}^{f,j} = h(x_{\ell}^{f,j}) + \eta_\ell^{j}, \qquad \eta_\ell^{j} \overset{\text{i.i.d.}}{\sim} \nu_\eta,\qquad 1 \leq \ell \leq N.
    \]
    \State \textbf{Analysis:}
    \[\hat \mu^{f, j}_{XY,N} =  \frac{1}{N}\sum\limits_{\ell = 1}^N\delta_{ (x_{\ell}^{f,j},  y_{\ell}^{f,j})} \]
    \[
x_{\ell}^{a,j}=  \map^{\mathrm{K},\hat \mu^{f, j}_{XY,N}}_{y^j} \left(x_{\ell}^{f,j},  y_{\ell}^{f,j}\right), \qquad 1 \leq \ell \leq N.
\]
\EndFor
\State \textbf{Output:} Approximate filtering distribution $\hat\mu_{X,N}^{a,j} = \frac{1}{N}\sum\limits_{\ell = 1}^N\delta_{x_{\ell}^{a,j}} $ for $j = 1, 2, \ldots$
\end{algorithmic}
\end{algorithm}
As outlined in Subsection~\ref{subsection:deriving_mean_field}, the mean-field dynamics associated with Algorithm~\ref{alg:EnKF} are given by the distributions
\begin{subequations}
\label{eq:mf_alg_measure_perspective_enkf}
\begin{align}
\tilde\mu_X^{a,0} &= \mu^0,\\
\tilde\mu_X^{a,j} &= \tilde B_{y^j}^{\mathrm{K}} Q^j P^{j-1}\tilde\mu_X^{a,j-1},
\end{align}
\end{subequations}
where \(\tilde B ^ {\mathrm{K}}_y\) is the approximate conditioning operator as defined 
in Equation~\eqref{eq:mean_field_operator}. To establish the convergence of \(\hat\mu_{X,N}^{a,j}\) to \(\tilde\mu_{X}^{a,j}\) rigorously, we invoke the general results of Section~\ref{section:proof_formal}. In order to apply these theorems in the EnKF setting, it remains to verify that the transport map \(\map_y^{\mathrm K,\pi}\) in Equation~\eqref{eq:enkf_beta_map_gen} satisfies the required Lipschitz and stability properties in Assumption~\ref{assumption:algorithm_finite_particle}. 

\begin{lemma}
\label{lem:enkf_satisfies_assumptions}
Assume Equation~\eqref{eq:enkf_observation_covariance} holds. The map $\map^{\mathrm{K},\mu}_{y}$ as defined in Equation~\eqref{eq:enkf_beta_map_gen} for all $\mu \in \Pp_2(\R^n \times \R^m)$ and $y \in \R^m$ satisfies Assumption \ref{assumption:algorithm_finite_particle} with the constants in Table \ref{tab:enkf_constants_values}.
\end{lemma}

\begin{proof}
We will show the three conditions required by Assumption~\ref{assumption:algorithm_finite_particle}.

\textit{(1a)} By definition of $\map^{\text{K},\mu}_{y^\star}$, for all $x,x^\prime \in \R^n$  and  $y,y^\prime \in \R^m$ it holds that 
\begin{align*}
   &\left\| \map^{\text{K},\mu}_{y^\star}(x,y)
   - \map^{\text{K},\mu}_{y^\star}(x^\prime, y^\prime)\right\|_2\\
   &\quad = \left\|(x - x^\prime)
   +  \Cov\left(\mu\right)_{XY}
   \left(\Cov\left(h_\sharp\mu_X\right) + \Lambda\right)^{-1}(y^\prime - y)\right\|_2 \\
   &\leq \left( 1 + \opnorm{\Lambda^{-1}}\opnorm{\Cov\left(\mu\right)_{XY}}\right)\left\|(x,y) - (x^\prime,y^\prime)\right\|_2 \\
    &\leq \left( 1 + \opnorm{\Lambda^{-1}}{\Trcov\left(\mu\right)}\right)\left\|(x,y) - (x^\prime,y^\prime)\right\|_2 \\
    &\leq \left( 1 + \lambda_{\min}^{-1}\right)\left( 1 + {\Trcov\left(\mu\right)}\right)\left\|(x,y) - (x^\prime,y^\prime)\right\|_2,
\end{align*}
where we use Equation~\eqref{eq:enkf_observation_covariance}.

\noindent \textit{(1b)}  We abbreviate $\Delta y = y^\star - y$ and set
\[
H_\rho\coloneqq \Cov\left(h_\sharp\rho_X\right),
\qquad
A_\rho\coloneqq H_\rho+\Lambda .
\]
Let $x \in \R^n$, $y, y^\star \in \R^m$, $\nu, \mu \in \Pp_{2}(\R^n\times\R^m)$. Then,
\begin{align*}
   \left\| \map^{\text{K},\nu}_{y^\star}(x,y)  -  \map^{\text{K},\mu}_{y^\star}(x,y)\right\|_2
   &= \left\|\Cov\left(\nu\right)_{XY}A_\nu^{-1}\Delta y
   - \Cov\left(\mu\right)_{XY}A_\mu^{-1}\Delta y \right\|_2 \\
   &\leq \left\|\Cov\left(\nu\right)_{XY}A_\nu^{-1}
   - \Cov\left(\mu\right)_{XY}A_\mu^{-1}\right\|
   \left\|\Delta y \right\|_2 \\
      &\leq \opnorm{\Lambda^{-1}}\opnorm{\Cov\left(\nu\right)_{XY}  - \Cov\left(\mu\right)_{XY}}\left\|\Delta y \right\|_2 \\ 
      & + \opnorm{\Cov(\nu)_{XY}}
      \opnorm{A_\nu^{-1}-A_\mu^{-1}}
      \left\|\Delta y \right\|_2 \\
      &\leq \opnorm{\Lambda^{-1}}\opnorm{\Cov\left(\nu\right)  - \Cov\left(\mu\right)}\left\|\Delta y \right\|_2 \\ 
      & + \opnorm{\Cov(\nu)}
      \opnorm{A_\nu^{-1}-A_\mu^{-1}}
      \left\|\Delta y \right\|_2\\
       &\leq \opnorm{\Lambda^{-1}}\opnorm{\Cov\left(\nu\right)  - \Cov\left(\mu\right)}
       \left\|\Delta y \right\|_2 \\ 
      & + \Trcov(\nu)\opnorm{\Lambda^{-1}}^2
      \opnorm{ H_\nu - H_\mu}
      \left\|\Delta y \right\|_2.
\end{align*}
Stacking the covariance matrices together, i.e., inserting $g(x,y) = ((x,y), h(x))$, verifies \textit{(1b)}.

\textit{(3)} Finally, let $\nu \in \Pp_2\left(\R^n\times\R^m\right), x\in\R^n,y\in\R^m, y^\star \in\R^m$. Then, 
\begin{align*}
        \left\|\map^{\text{K},\nu}_{y^\star}(x,y) - x \right\|_2
        &= \left\| \Cov(\nu)_{XY}A_\nu^{-1}(y^\star - y)\right\|_2 \\
    &\leq \lambda_{\min}^{-1}\left(1 + \Trcov(\nu)\right)\left\|y^\star - y\right\|_2.
\end{align*}
\end{proof}

\ifpreprint
\begin{table}[H]
\else
\begin{table}[t]
\fi
\centering
\setlength{\tabcolsep}{4pt} 
\renewcommand{\arraystretch}{1.2} 
\resizebox{\textwidth}{!}{%
\begin{tabular}{cccccccccccc}
\toprule
$C_{\mathrm{L}}$ & $e_{\mathrm{L}}$ & $g(x,y)$ & $b$ & $L_g$ & $C_{\mathrm{est}}$ & $e_{\mathrm{est},1}$ & $e_{\mathrm{est},2}$ & $a$ & $\sigma_{\kappa}$ & $L_y$ & $e_y$ \\
\midrule
$1+\lambda_{\min}^{-1}$ & $1$ & $((x,y),h(x))$ & $n+2m$& $\sqrt{1+L_h^2}$ &
$\lambda_{\min}^{-1}+\lambda_{\min}^{-2}$ & $1$ & $0$ & $0$ & $0$ &
$\lambda_{\min}^{-1}$ & $1$ \\
\bottomrule
\end{tabular}
}
\caption{Constants for the ensemble Kalman filter in Assumption~\ref{assumption:algorithm_finite_particle}.}
\label{tab:enkf_constants_values}
\end{table}

Now, we apply the results of Section \ref{section:proof_formal} to obtain the mean-field convergence of the EnKF as an immediate consequence of Theorem~\ref{thm:ips_converges_to_iid_particles} and Lemma~\ref{lem:enkf_satisfies_assumptions}.


\begin{theorem}
\label{thm:enkf_finite_sample}
Fix $J\in\N_{\ge1}$, draw observations $Y^{1:J}$ as in Equation~\eqref{eq:dynamics} (i.e., $Y^{1:J} \sim \rho^J$ as defined in Equation~\eqref{eq:rho_def}) 
and suppose Assumption~\ref{assumption:dyn_assumptions} holds with time-independent observation noise \(\nu_\eta^j\equiv\nu_\eta\). Assume also that Equation~\eqref{eq:enkf_observation_covariance} holds. Let
$\{\hat\mu_{X,N}^{a,j}\}_{j=1}^{J}$ be the analysis empirical measures for the interacting-particle system produced by Algorithm~\ref{alg:EnKF}, and let \(\tilde\mu^{a,j}_X\) be the mean-field sequence of distributions defined by Equation~\eqref{eq:mf_alg_measure_perspective_enkf}, both from observations \(y^{1:J}=Y^{1:J}\). Then, for all \(k\ge 2\) and $N\geq (\log k)^2$, there exists a coupling of i.i.d.\ particles \(v_\ell^{a,j}\sim\tilde\mu_X^{a,j}\) with the interacting ensemble such that the corresponding empirical measure \(\tilde\mu^{a,j}_{X,N}=\frac1N\sum_{\ell=1}^N\delta_{v_\ell^{a,j}}\) satisfies
\begin{equation} \label{eq:EnKF_W2_convergence_meanfield}
W_2\bigl(\hat\mu_{X,N}^{a,j},\tilde\mu_{X,N}^{a,j}\bigr)
\ \le\ C\frac{(\log k)^{1 +j/2}}{\sqrt{N}} \text{ for all }1 \leq j \leq J,
\end{equation}
with probability at least $1 - \frac{1}{k}$ over the joint randomness of the particle draws in Algorithm~\ref{alg:EnKF} and the observation path $Y^{1:J}$, where \(C \coloneqq C(\mathcal{Q})\) depends only on
\[
\mathcal{Q} \coloneqq (n,m,\sigma_{\max},\lambda_{\max},\lambda_{\min},\sigma_X,L_\Psi,L_h,J).
\]
Taking expectations over this randomness yields the non-asymptotic bound
\[
\E \left(\sup\limits_{1\le j\le J} W_2\bigl(\hat\mu_{X,N}^{a,j},\tilde\mu_{X,N}^{a,j}\bigr)\right)
 \le \frac{C'(\mathcal{Q})}{\sqrt{N}},
\]
for a constant \(C'(\mathcal{Q})\) depending only on \(\mathcal{Q}\).
\end{theorem}

Theorem~\ref{thm:enkf_finite_sample} establishes the first non-asymptotic, high-probability convergence guarantee in Wasserstein distance for the discrete-time EnKF under general Lipschitz dynamics and observation maps with sub-Gaussian process and observation noise and random observations. The result shows that the empirical EnKF ensemble behaves approximately as i.i.d.\ samples from its mean-field limit, with the optimal Monte Carlo rate \(N^{-1/2}\) in \(W_2\). Combined with standard i.i.d.\ concentration inequalities, this yields the dimension-dependent \(W_p\) convergence rate \(N^{-1/n}\) in the high-dimensional regime \(n>2p\). For completeness, we also quantify the $W_p$-distance to the mean-field distribution $\tilde\mu_{X}^{a,j}$, which
follows directly from Corollary \ref{cor:finite_particle_converges_to_mean_field}.

\begin{corollary}
Pick any $p\in[1,2]$. Suppose the conditions in Theorem \ref{thm:enkf_finite_sample} apply. Then,  there exists a constant $C\geq0$ depending only on $\mathcal{Q}$ and $p$ such that for all $k\ge 2$ and \(N\ge(\log k)^2\), with probability at least $1-\frac1k$, we have
\[
W_p\bigl(\hat\mu_{X,N}^{a,j},\tilde\mu_X^{a,j}\bigr)
\ \le C\left( \gamma_{k,N}^{p,n}  +\frac{(\log k)^{1 +j/2}}{\sqrt{N}}\right) \text{ for all }1 \leq j \leq J,
\]
with the rate \(\gamma_{k,N}^{p,n}\) as in Equation~\eqref{eq:standard_iid_wasserstein_rate} and
\[
\E \left(\sup\limits_{1\le j\le J} W_p\bigl(\hat\mu_{X,N}^{a,j},\tilde\mu_X^{a,j}\bigr)\right)
 \le  C \bar\gamma_N^{p,n}.
\]
\end{corollary}

We emphasize that Theorem~\ref{thm:enkf_finite_sample} is, however, the more informative statement: it shows that the EnKF ensemble is (with high probability) as close as possible to an i.i.d.\ sample from \(\tilde\mu_{X}^{a,j}\) at the Monte Carlo scale. Once this i.i.d.\ behavior is in place, the residual discrepancy to the mean-field is precisely the standard empirical-process error \(\bar\gamma_N^{p,n}\), which reduces to \(N^{-1/n}\) when \(n>2p\) and is known to be sharp in the absence of further assumptions~\cite{fournier2015rate}.

\subsection{Mean-Field Convergence of the EnSMF}
\label{subsection:application_ensmf} 
In this subsection, we illustrate how our theory also applies to nonlinear transport ensemble filters by analyzing the ensemble stochastic map filter (EnSMF)~\cite{spantini2022coupling} as a representative example. The EnSMF generalizes the 
 linear ansatz of the EnKF 
by learning a (possibly) nonlinear map $T$ using a variational approach based on samples from the joint forecast distribution. 
This approach overcomes a limitation of the EnKF's linear update, which fails to accurately characterize analysis distributions whose dependence on the observations is nonlinear~\cite{jorgensen2026exactaffineconditioninggaussians}, as is the case for general non-Gaussian distributions. 
We also refer to \cite{spantini2022coupling,ramgraber2023Bensemble} for evidence of the EnSMF's lower bias as compared to the EnKF for capturing posterior moments when using nonlinear, rather than linear, transport maps. We introduce stochastic maps and the EnSMF precisely in Subsection~\ref{subsection:smf_intro} and 
establish our mean-field convergence result for the EnSMF in Subsection~\ref{subsubsection:convergence_ensmf}.

\subsubsection{Introduction to Stochastic Maps and the EnSMF} 
\label{subsection:smf_intro}
The EnSMF differs from the EnKF (Algorithm~\ref{alg:EnKF}) only in the analysis step where it replaces the EnKF's affine map $\map^{\mathrm K}_{y}$ with a (possibly nonlinear) transport map that approximates the exact conditioning operator $B_{y^\star}$ in Equation~\eqref{eq:cond_operator} for 
an arbitrary observation \(y^\star\). We begin with an informal sketch of the EnSMF analysis step, whose precise, rigorous formulation will follow. 

To define the nonlinear transport, one first seeks a map $\tilde S\colon\R^n\times\R^m\longrightarrow\R^n\times\R^m$  
 that pushes the joint forecast \(\mu^{f,j}_{XY}\) onto the product distribution of a simple reference measure $\gamma_\text{G} \in \Pp_2\left(\R^n \right)$, such as a standard Gaussian, and the marginal measure for the observations $\mu_Y^{f,j} \in \Pp_2\left(\R^m \right)$. That is,
 \begin{equation} \label{eq:push_forward_condition}
 \tilde S_\sharp\mu_{XY}^{f,j} = \gamma_{\mathrm G}\otimes\mu_Y^{f,j}.
 \end{equation}
A convenient structure for conditional simulation is to use invertible transport maps with the block-triangular form 
$\tilde S(x,y) :=\bigl(S(x,y),y\bigr),$ where both components depend on the observation vector \(y\), but only the first component \(S:\R^n\times\R^m\longrightarrow\R^n\) depends on $x$. 
As shown in \cite[Thm.~2.4]{baptista2024conditional}, if the map satisfies Equation~\eqref{eq:push_forward_condition}, then the inverse of $x \mapsto S(x,y^j)$, denoted by $S(\cdot,y^j)^{-1}$, enables posterior sampling via the pushforward
$$\mu_{X}^{a,j}
= \bigl(S(\cdot,y^j)^{-1}\bigr)_\sharp \gamma_{\mathrm G}.$$
This is the main idea behind conditional sampling with triangular transport~\cite{parno2018transport, marzouk2017sampling}. In the ensemble-filtering setting, where one uses a particle ensemble to approximate \(\mu_{XY}^{f,j}\), \cite{spantini2022coupling} proposes a ``prior-to-posterior'' scheme that replaces the reference measure $\gamma_{\mathrm G}$ with samples from the pushforward distribution $ S_\sharp\mu_{XY}^{f,j}$. This yields the conditioning operator
\begin{equation}
\label{eq:sm_idea_map}
 \mu_{X}^{a,j} = \bigl(S(\cdot,y^j)^{-1}\circ  S\bigr)_\sharp  \mu_{XY}^{f,j}.
\end{equation}
This procedure is called the ``stochastic map'' and is precisely the analysis step in the EnSMF. It preserves the ensemble size when $\mu_{XY}^{f,j}$ is replaced by an empirical measure, and it is exact in the case $\tilde S_\sharp\mu_{XY}^{f,j}=  \gamma_{\mathrm G}\otimes\mu_Y^{f,j}$. In the following paragraphs we outline the parametric form of the transport map, the learning problem for the map, and regularization techniques for estimating the map.

\textit{Map Parameterization.}
In practice, it is common to
construct the map $S\colon \R^n \times\R^m \rightarrow \R^n$ by considering a parametric family of transport maps $\mathcal{S}$. Following \cite{spantini2022coupling}, we 
define $S$ to be a strictly triangular map of the form
$$S(x,y) = \begin{bmatrix*}[l] S_1(x_1,y) \\ S_2(x_1,x_2,y) \\ \vdots \\ S_n(x_1,\dots,x_n,y)  \end{bmatrix*},$$
where each component $S_k$ for $k=1,\dots,n$ is given by
\begin{equation} \label{eq:S_components}
S_k(x,y) = \alpha_k x_k + c_k + \sum_{i=1}^{\iota_k} \theta_k^i f_k^i(x_{1:k-1},y),
\end{equation} 
where \(\iota_k \in \mathbb{N}\) and \(\{f_k^i\}_{k=1,\ldots,n; \, i=1,\ldots,\iota_k}\) is a given collection of functions 
\(f_k^i: \mathbb{R}^{k-1} \times \mathbb{R}^m \to \mathbb{R}\), with the convention that \(x_{1:0}\) is the unique element of \(\mathbb{R}^0\). The parameters of all components are then defined by the coefficients \(\alpha \in \mathbb{R}_{\geq \alpha_{\min}}^n\) for some fixed \(\alpha_{\text{min}} > 0\), \(c \in \mathbb{R}^n\) and scalars \(\{\theta_k^i\}_{k=1,\ldots,n; \,  i=1,\ldots,\iota_k}\) with \(\theta_k^i \in \mathbb{R}\). In what follows, we rewrite the coefficients \(\{\theta_k^i\}_{k,i}\) and functions \(\{f_k^i\}_{k,i}\) using vectors:
\begin{align*}
 \theta_k &\coloneqq \left( \theta_k^1, \ldots,  \theta_k^{\iota_k}\right)^\top \\
    f_k &\coloneqq \left( f_k^1, \ldots,  f_k^{\iota_k}\right)^\top.
\end{align*}
When a full state vector is supplied as the first argument, \(f_k^i(x,y)\) means \(f_k^i(x_{1:k-1},y)\), and \(f_k(x,y)\) means the corresponding vector of feature evaluations.
We write \(\iota\coloneqq(\iota_1,\ldots,\iota_n)\) for the vector of feature dimensions.
Moreover, we will write $S = \mathcal{S}(\alpha, c, \theta) \in \mathcal{S}$ 
to denote instances of such maps in the parametric family $\mathcal{S}$. For our results, the functions \(f_k^i\) are assumed to satisfy the following conditions.

\begin{assumption}
\label{assumption:sm_assumption}
For each \(k \in \{1, \ldots, n\}\), we have
\begin{enumerate}
    \item \textit{Lipschitz continuity:} Each function \( f_k\)
    is \(L_f\)-Lipschitz for a fixed constant \(L_f \geq 0\).
    \item \textit{Differentiability:}   Each function \( f_k^i \) is continuously differentiable in $x$. 
    \item \textit{Triangular variable dependence:}  For each $i \in \{1,\ldots, \iota_k\}$ and $j \geq k$, $\partial_{x_j}f_k^i =0$. 
\end{enumerate}
\end{assumption}

Concrete choices of the basis functions \(f_{k}^i\), as described in \cite{spantini2022coupling}, include polynomial (monomial) bases---typically regularized at infinity to ensure global Lipschitz continuity---and Gaussian radial basis functions. Both of these examples satisfy Assumption~\ref{assumption:sm_assumption}.

\textit{Learning the Map.} 
As explained at the beginning of this section, a natural approach is to choose \(S\) so that \(\tilde S_\sharp\mu\) has small Kullback--Leibler (KL)  divergence from \(\gamma_{\mathrm G}\otimes\mu_Y\). When this quantity is well defined, this approach leads to the variational problem
\begin{equation}
\label{eq:KL_min_sm}
\min_{S \in \mathcal{S}} \KL\bigl(\tilde S_\sharp\mu\,\|\,\gamma_{\mathrm G}\otimes\mu_Y\bigr),
\end{equation}
with \(\tilde S(x,y)\coloneqq(S(x,y),y)\), where \(\KL(\pi\,\|\,\pi')\) denotes the KL divergence from \(\pi\) to \(\pi'\). As is typical in stochastic-map constructions~\cite{parno2018transport,spantini2022coupling}, we take \(\gamma_{\mathrm G}=\mathcal N(0,I_n)\) as the reference measure. The triangular structure of \(S\), together with the constraint \(\alpha_k\ge \alpha_{\min}>0\), implies---by a straightforward induction over the state dimension---that, for every \(S\in\mathcal{S}\) and \(y\in\mathbb{R}^m\), the map \(S(\cdot,y):x\mapsto S(x,y)\) is a diffeomorphism; see Lemma~\ref{lem:inverting_sm} in the Appendix. Thus the pullback measure \(\left(S(\cdot, y)\right)^\sharp\gamma_{\mathrm G}\coloneqq \left(S(\cdot, y)^{-1}\right)_\sharp\gamma_{\mathrm G}\) has a density.

To avoid potential well-posedness issues in Equation~\eqref{eq:KL_min_sm}, while preserving its minimizers whenever the KL formulation is valid, we choose \(S\) directly as the maximum-likelihood estimator, i.e., as a minimizer of the cross-entropy objective~\cite{parno2018transport,spantini2022coupling}
\begin{equation}
\label{eq:cross_entropy_min_sm}
\min_{S \in \mathcal{S}} -\int \dd \mu(x,y)\log \left(\left(S(\cdot, y)\right)^\sharp\gamma_{\mathrm G}\right)(x).
\end{equation}
This objective depends on the joint distribution \(\mu\) only through the expectation and is therefore amenable to empirical approximation, e.g., using forecast samples. The following result shows that, under suitable conditions, the Gaussian reference choice \(\gamma_{\mathrm G}\) leads to a well-posed stochastic-map optimization in Equation~\eqref{eq:cross_entropy_min_sm}. A proof is included in the Appendix.
\begin{restatable}{proposition}{propexistenceminimizer}
\label{prop:existence_minimizer_sm_KL}
Let $\mu \in\Pp_2\left(\R^n \times \R^m \right)$ satisfy $\Cov\left(\left( f_k\right)_\sharp \mu \right)  \succ 0$ and 
\begin{equation*}
 R_k(\mu) \coloneqq \Cov\left(X_k\right)  -  \Cov\left(  f_k(X,Y), X_k \right) ^\top\Cov\left( f_k(X,Y) \right)^{-1} \Cov\left(  f_k(X,Y), X_k \right) > 0,
\end{equation*}
for every $k = 1, \ldots, n$ and $(X, Y) \sim \mu$. Then, Equation~\eqref{eq:cross_entropy_min_sm} has a unique minimizer $S^\mu \in \mathcal{S}$
that is defined by the parameters
\begin{subequations}
\label{eq:params_sm}
\begin{align}
    \alpha_k &=  \max\left(\frac{1}{\sqrt{R_k(\mu)}}, \alpha_{\min}\right) \\
\theta_k &=  -  \alpha_k \Cov\left( f_k(X,Y) \right)^{-1} \Cov\left(  f_k(X,Y), X_k \right) \\
c_k  &=- \E\left(\alpha_kX_k  + \sum_{i=1}^{\iota_k}   \theta_k^i f_k^i(X_{1:k-1},Y)\right).
\end{align}
\end{subequations}
\end{restatable}
For suitable $\mu \in\Pp_2\left(\R^n\times\R^m\right)$ satisfying the condition in Proposition \ref{prop:existence_minimizer_sm_KL} and $y^\star \in \R^m$, we now construct the prior-to-posterior transport map in~\eqref{eq:sm_idea_map} that defines the approximate conditioning operator $\tilde B_{y^\star}$ using the stochastic map
\[
\map_{y^\star}^{\mu}(x,y) := \left(S^\mu(\cdot,y^\star)^{-1}\circ S^\mu\right)(x,y).
\]
For the map components in Equation~\eqref{eq:S_components}, we obtain a closed form for the stochastic map by direct calculation.
Defining the map's output as $w = \map_{y^\star}^\mu(x,y) \in \R^n$, the following recursive equations hold for all $k\geq 1$ with $\alpha, \theta$ defined as in Proposition \ref{prop:existence_minimizer_sm_KL}: 
\begin{equation*}
w_k  = x_k + \frac{\theta_k^\top}{\alpha_k} \left(f_k(x,y) -f_k(w,y^\star)\right).
\end{equation*}
Since $w_k$
depends only on the ratio  
$\frac{\theta_k}{\alpha_k}
= -\Cov\!\left(f_k(X,Y)\right)^{-1}\Cov\!\left(f_k(X,Y), X_k\right),$
the definition can be naturally extended to cases where $R_k(\mu)=0$ by defining the stochastic map directly in terms of the moments of $\mu$.
Defining the map's output as $\chi_{y^\star}(x,y;\theta) \in \R^n$ with $\theta_k \in \R^{\iota_k}$ for each $k$, the following recursive equations hold:
\begin{equation}
\label{eq:sm_approximate_conditioning_recursion}
(\chi_{y^\star}(x,y;\theta))_k  = x_k - \theta_k^\top\left(f_k(x,y) -f_k(\chi_{y^\star}(x,y),y^\star)\right).
\end{equation}
Indeed, under the assumptions of Proposition \ref{prop:existence_minimizer_sm_KL}, the prior-to-posterior map \(\map_{y^\star}^{\mu}\) obtained from the minimizer \(S^\mu\) in Equation \eqref{eq:cross_entropy_min_sm} coincides with \(\chi_{y^\star}(\cdot,\cdot;\theta)\) when
\[
\theta_k = \Cov\left( f_k(X,Y) \right)^{-1} \Cov\left(  f_k(X,Y), X_k \right).
\]

\textit{Map Regularization.}
To extend this expression stably when $\Cov\!\left(f_k(X,Y)\right)$ is singular or ill-conditioned, it is common to incorporate regularization of the covariance matrix~\cite{spantini2022coupling}. Given a symmetric matrix $C \in \R^{u \times u}$ and a regularization constant $\upsilon> 0$, we define $C_{\upsilon}$ as the closest positive semidefinite approximation to $C$ that is bounded below by $\upsilon I_{u}$. That is
\begin{equation}
\label{eq:frob_proj}
C_{\upsilon} = \argmin_{\tilde C \succeq \upsilon I_{u}} \left\|\tilde C - C\right\|_F^2. 
\end{equation}
With a fixed regularization level \(\hat\sigma_f>0\), we define the stochastic map for $y^\star \in \R^m$ and any $\mu \in   \Pp_2\left(\R^n\times\R^m\right)$ as 
\begin{equation}
\label{eq:alg_conditioning_transport_SM}
   \map^{\mathrm{SM},\mu}_{y^\star}(x,y) = \chi_{y^\star}(x,y;\theta),
\end{equation}
for the parameter $\theta_k \coloneqq \Cov_{\hat \sigma_f}\left( f_k( X, Y) \right)^{-1} \Cov\left(  f_k( X, Y),  X_k \right)$ and $(X,  Y) \sim  \mu$, where \(\Cov_{\hat\sigma_f}(Z)\coloneqq(\Cov(Z))_{\hat\sigma_f}\).

The update step in~\eqref{eq:alg_conditioning_transport_SM} results in the transport ensemble filter. Taking $\mu$ to be the empirical measure from the forecast step yields the ensemble stochastic map filter, which is summarized in Algorithm~\ref{alg:SMF}.

\begin{remark}
Alternative regularizations (e.g., Tikhonov) lead to an essentially identical analysis and comparable bounds. Here, we adopt the projection scheme in~\eqref{eq:frob_proj}, which is equivalent to eigenvalue thresholding, because it yields a cleaner and more informative mean-field limit. A brief note on the computational implementation of this scheme is provided in Appendix~\ref{subsection:implementation_regularization}.
\end{remark}

\begin{algorithm}[t]
\caption{Ensemble Stochastic Map Filter (EnSMF)}\label{alg:SMF}
\begin{algorithmic}[1]
\State \textbf{Input:} ensemble size $N$ and sequentially acquired data $\{y^{j}\}_{j \geq 1}$.
\State \textbf{Initialization:} $x_{\ell}^{a,0} \overset{\text{i.i.d.}}{\sim} \mu^0,\qquad 1 \leq \ell \leq N.$
\For{$j = 1, 2, \ldots$}
    \State \textbf{Forecast: }
    \begin{align*}
    x_{\ell}^{f,j} &= \Psi(x_{\ell}^{a,j-1}) + \xi_\ell^{j-1}, \qquad \xi_\ell^{j-1} \overset{\text{i.i.d.}}{\sim} \nu_\xi^{j-1}, \qquad 1 \leq \ell \leq N, \\
    y_{\ell}^{f,j} &= h(x_{\ell}^{f,j}) + \eta_\ell^{j}, \qquad \eta_\ell^{j} \overset{\text{i.i.d.}}{\sim} \nu_\eta^j,\qquad 1 \leq \ell \leq N.
    \end{align*}
    \State \textbf{Analysis:}
    \begin{align*}
    &\hat \mu^{f, j}_{XY,N} = \frac{1}{N}\sum\limits_{\ell = 1}^N\delta_{ (x_{\ell}^{f,j},  y_{\ell}^{f,j})} \\
    &x_{\ell}^{a,j} =  \map^{\mathrm{SM},\hat \mu^{f, j}_{XY,N}}_{y^j} \left(x_{\ell}^{f,j},  y_{\ell}^{f,j}\right), \qquad 1 \leq \ell \leq N.
    \end{align*}
\EndFor
\State \textbf{Output:} Approximate filtering distribution $\hat\mu_{X,N}^{a,j} = \frac{1}{N}\sum\limits_{\ell = 1}^N\delta_{x_{\ell}^{a,j}} $ for $j = 1, 2, \ldots$
\end{algorithmic}
\end{algorithm}
\subsubsection{Convergence Result}
\label{subsubsection:convergence_ensmf}
As in Subsection~\ref{subsection:deriving_mean_field}, the mean-field dynamics to which the ensemble $\hat\mu_{X,N}^{a,j}$ produced by Algorithm~\ref{alg:SMF} converges are given by
\begin{subequations}
\label{eq:ensmf_mf_alg_measure_perspective}
\begin{align}
\tilde \mu_{X}^{a, 0} &= \mu^0 \\ 
\tilde \mu_{XY}^{f, j} &=Q^jP^{j-1}\tilde \mu_X^{a,j-1}  \\ 
\tilde \mu_{X}^{a,j} & = \tilde B^{\text{SM}}_{y^j} \tilde \mu_{XY}^{f, j}  
\end{align}
\end{subequations} 
where $\tilde B^{\text{SM}}_y$ is defined in Equation \eqref{eq:mean_field_operator}. To prove this statement rigorously, we rely on the results from Section~\ref{section:proof_formal} and establish that Assumption~\ref{assumption:algorithm_finite_particle} is satisfied for the map $\map_{y}^{\mathrm{SM},\mu}$ at any analysis step of the algorithm in the following lemma.
\begin{restatable}{lemma}{lemsmfassumptions}
\label{lem:smf_satisfies_assumptions}
$\map^{\mathrm{SM},\mu}_y$ satisfies Assumption \ref{assumption:algorithm_finite_particle} with constants depending only on the quantities introduced in Assumptions \ref{assumption:dyn_assumptions} and \ref{assumption:sm_assumption}, on the feature dimensions \(\iota\), and on the regularization level \(\hat\sigma_f\).
\end{restatable}
A proof of this lemma can be found in Appendix \ref{subsection:convergence_ensmf_aux} and we can now invoke the general results of Section~\ref{section:proof_formal} to obtain mean-field convergence of the EnSMF as a direct consequence of Theorem~\ref{thm:ips_converges_to_iid_particles} together with Lemma~\ref{lem:smf_satisfies_assumptions}.

\begin{theorem}
\label{thm:ensmf_finite_sample}
Fix $J\in\N_{\ge1}$ and draw $Y^{1:J}$ as in Equation \eqref{eq:dynamics}, i.e., $Y^{1:J} \sim \rho^J$ as defined in Equation~\eqref{eq:rho_def}. Suppose Assumptions~\ref{assumption:dyn_assumptions} and~\ref{assumption:sm_assumption} hold. Let \(\{\hat\mu_{X,N}^{a,j}\}_{j=1}^{J}\) be the analysis empirical measures produced by Algorithm~\ref{alg:SMF}, and let \(\tilde\mu^{a,j}_X\) be the mean-field sequence defined by Equation \eqref{eq:ensmf_mf_alg_measure_perspective} with observations \(y^{1:J}=Y^{1:J}\).  
Then there exists a coupling of i.i.d.\ particles \(v_\ell^{a,j}\sim\tilde\mu_X^{a,j}\) with the interacting ensemble such that the corresponding empirical measure \(\tilde\mu^{a,j}_{X,N}=\frac1N\sum_{\ell=1}^N\delta_{v_\ell^{a,j}}\) satisfies, for all \(k\ge 2\) and $N\geq (\log k)^2$,  with probability at least $1 - \tfrac1k$
\[
W_2\bigl(\hat\mu_{X,N}^{a,j},\tilde\mu_{X,N}^{a,j}\bigr)
\ \le\ C\frac{(\log k)^{1 +j/2}}{\sqrt{N}} \text{ for all }1 \leq j \leq J,
\]
where \(C\coloneqq C(\mathcal{Q})\) depends only on
\[
\mathcal{Q} \coloneqq (n,m,\sigma_{\max},\lambda_{\max},\sigma_X,L_\Psi,L_h,L_f, \hat\sigma_f,\iota,J).
\]
The randomness is over both the particle draws in Algorithm~\ref{alg:SMF} and the observation path \(Y^{1:J}\). Moreover, we have the non-asymptotic bound
\[
\E \left(\sup\limits_{1\le j\le J}  W_2\bigl(\hat\mu_{X,N}^{a,j},\tilde\mu_{X,N}^{a,j}\bigr)\right)
 \le \frac{C'(\mathcal{Q})}{\sqrt{N}},
\]
for a constant \(C'(\mathcal{Q})\) depending only on \(\mathcal{Q}\).
\end{theorem}

As for the EnKF, we include a result for the convergence of $\hat\mu_{X,N}^{a,j}$ to the mean-field \(\tilde\mu_{X}^{a,j}\) in the \(W_p\) distance. The proof of this result follows the same steps as Corollary~\ref{cor:finite_particle_converges_to_mean_field}.

\begin{corollary}
Consider the same setting as in Theorem \ref{thm:ensmf_finite_sample} and choose any $p\in[1,2]$. Then, there exists a constant $C \geq 0$ depending only on $\mathcal{Q}$ and $p$ such that for all $k\ge 2$ and $N\geq (\log k)^2$, with probability at least $1-\frac1k$,
\[
W_p\bigl(\hat\mu_{X,N}^{a,j},\tilde\mu_X^{a,j}\bigr)
\ \le C\left(\gamma_{k,N}^{p,n}+\frac{(\log k)^{1 +j/2}}{\sqrt{N}}\right) \text{ for all }1 \leq j \leq J,
\]
with the rate \(\gamma_{k,N}^{p,n}\) as in Equation~\eqref{eq:standard_iid_wasserstein_rate} and
\[
\E \left(\sup\limits_{1\le j\le J}W_p\bigl(\hat\mu_{X,N}^{a,j},\tilde\mu_X^{a,j}\bigr)\right)\le  C \bar\gamma_N^{p,n}.
\]
\end{corollary}


\appendix
\section{Auxiliary Results}

\subsection[Auxiliary Results for Mean-Field Limit Arguments]{Auxiliary Results for Section~\ref{subsection:convergence_intro}}
\label{subsection:sup_MFL}

For the empirical-measure concentration result below, define
\[
\gamma_{k,N}^{p,d} =
\begin{cases}
 \Bigl(\frac{\log k}{N}\Bigr)^{1/d}
 &\text{if } 1\le p < \frac d2\\
\Bigl(\frac{\log k}{N}\Bigr)^{1/d}\bigl(\log(2+N)\bigr)^{2/d}
&\text{if } p = \frac{d}{2}\\
\Bigl(\frac{\log k}{N}\Bigr)^{1/(2p)}
&\text{if } \frac d2 < p \le 2\,,
\end{cases}
\]
and
\[
\tilde\gamma_{k,N}^{p,d}
\coloneqq \max\left\{\gamma_{k,N}^{p,d},\sqrt{\frac{\log k}{N}}\right\}.
\]
For convenience, we also recall the corresponding \(k\)-free empirical-measure rate
\[
\bar\gamma_N^{p,d} =
\begin{cases}
N^{-1/d}
&\text{if } 1\le p < \frac d2,\\
N^{-1/d}\bigl(\log(2+N)\bigr)^{2/d}
&\text{if } p = \frac{d}{2},\\
N^{-1/(2p)}
&\text{if } \frac d2 < p \le 2.
\end{cases}
\]

\begin{proposition}
\label{cor:wp_concentration}
Let $d\geq1$ and $p\in[1,2]$. Let $(X_i)_{i=1}^N$ be i.i.d.\ $\mathbb R^d$-valued with law $\mu$, and let
$\mu_N$ be the empirical measure. Assume the centered sub-Gaussian condition
\[
\big\|X_i-\mathbb E X_i\big\|_{2,\psi_2}\le K .
\]
Then there exists a constant $A_1\geq 0$, depending only on $d,p,K$, such that for all
$N\ge1$ and  $k\ge2$ 
\[
\mathbb P\left(W_p(\mu_N,\mu)\ \ge A_1 \tilde\gamma_{k,N}^{p,d}\right)\ \le\ \frac1k\,.
\]
Moreover, there exists a constant $A_2> 0$, depending only on $d,p,K$, such that for all
$N\ge1$,
\[\E W_p(\mu_N,\mu)\le A_2 \bar\gamma_N^{p,d}\,.
\]
\end{proposition}

\begin{proof}
By the sub-Gaussian assumption 
\[
\mathbb E\exp\bigl(\| X-\mathbb EX\|_2^2/K^2\bigr) \leq 2< \infty.
\]
Since a translation does not change Wasserstein distances, we may assume $\mathbb EX=0$ and thus
\(
\int e^{\|y\|_2^2/K^2}\mu(dy) < \infty.
\)  This is the exponential-moment assumption of Example~3.4 in \cite{larsson2024concentration}, with \(\beta=2\ge p\) and constants depending only on \(d,p,K\). Let $\mathcal{T}_p(\mu,\nu):=W_p(\mu,\nu)^p$. Example~3.4(a) with $\beta=2$ gives, for all $N\in\mathbb N$, $x>0$,
\[
\mathbb P\bigl(\mathcal{T}_p(\mu,\mu_N)>x\bigr)
\le
C_1\exp\bigl(-c_1 N\phi(x)\bigr)\mathrm 1_{\{x\le1\}}
+
C_1\exp\bigl(-c_1 N x^{2/p}\bigr)\mathrm 1_{\{x\ge1\}},
\]
where the rate function $\phi$ (depending on $d,p$) is
\[
\phi(x)
=
\begin{cases}
x^2 & p>\frac d2\\
\bigl(x/\log(2+1/x)\bigr)^2 & p=\frac d2\\
x^{d/p} & 0<p<\frac d2.
\end{cases}
\]
Plugging in $x = t^p$ gives
\[
\mathbb P\bigl(W_p(\mu_N,\mu)>t\bigr)
= \mathbb P\bigl(\mathcal{T}_p(\mu,\mu_N)>t^p\bigr)
\le C_1\exp\bigl(-c_1 N\phi(t^p)\bigr)
+
C_1\exp\bigl(-c_1 N t^{2}\bigr).
\]
For $t =A_1 \tilde\gamma_{k,N}^{p,d} $ the second term is upper bounded by  $(2k)^{-1}$ for sufficiently large $A_1 \geq 0$ since $\tilde\gamma_{k,N}^{p,d}  \geq \left(\log k/N \right)^{1/2}$ and $\log(2k)\le 2\log k$ for $k\ge2$. Similarly, plugging in $t =A_1 \tilde\gamma_{k,N}^{p,d}  \geq A_1\gamma_{k,N}^{p,d} $ for the first term and some elementary algebra in each of the three cases shows that it is smaller than $(2k)^{-1}$ for sufficiently large $A_1 \geq 0$. The two bounds sum to the claimed failure probability \(k^{-1}\).

The expectation bound follows by the layer-cake representation.
\end{proof}

\subsection[Auxiliary Results for Dynamics Control]{Auxiliary Results for Section \ref{subsection:dynamics_control}}
\label{subsection:sup_moments_dyn_process}
We derive the auxiliary results used in  Subsection \ref{subsection:dynamics_control}. We start by giving a routine argument for the stability of $\|\cdot\|_{2,\psi_2}$  under Lipschitz maps as we will apply this property repeatedly.

\begin{proposition}
\label{prop:subg_lip}
Let $X\in \R^{u_1}$ be a random vector such that $\|X - \E(X)\|_2$ is sub-Gaussian. Assume that \(f:\R^{u_1}\rightarrow\R^{u_2}\) is $L$-Lipschitz. Then
\[
\left\|f(X) - \E f(X)\right\|_{2,\psi_2}
\leq 2L \left\| X - \E(X)\right\|_{2,\psi_2}.
\]
\end{proposition}

\begin{proof}[Proof of Proposition \ref{prop:subg_lip}]
Consider $X^\prime$, an independent copy of $X$. 
By repeatedly applying Jensen's inequality and independence, we find
\begin{align*}
 \left\|f(X) - \E f(X)\right\|_{2,\psi_2}&= \inf \left\{t>0: \mathbb{E}\left(e^{ \left\|f(X) - \E f(X^\prime)\right\|_2^2  / t^{2}}\right) \leq 2\right\}\\
& \leq \inf \left\{t>0: \mathbb{E}\exp\left(\frac{\|f(X)-f(X^\prime)\|_2^2}{t^2}\right) \leq 2\right\}\\
& \leq \inf \left\{t>0: \mathbb{E}\left(e^{ \left\|X - X^\prime\right\|_2^2  / (t/L)^{2}}\right) \leq 2\right\}\\
& = \inf \left\{t>0: \mathbb{E}\left(e^{ \left\|(X- \E(X)) - (X^\prime - \E(X))\right\|_2^2  / (t/L)^{2}}\right) \leq 2\right\}\\
& \leq \inf \left\{t>0: \mathbb{E}\left(e^{ 2\left(\left\|X- \E(X)\right\|_2^2+ \left\|X^\prime - \E(X)\right\|_2^2\right)  / (t/L)^{2}}\right) \leq 2\right\}\\
& = \inf \left\{t>0: \mathbb{E}\left(e^{ 2\left(\left\|(X- \E(X))\right\|_2^2\right)  / (t/L)^{2}}\right)^2 \leq 2\right\}\\ 
& \leq \inf \left\{t>0: \mathbb{E}\left(e^{ 4\left\|X- \E(X)\right\|_2^2  / (t/L)^{2}}\right)\leq 2\right\}\\
 &= \inf \left\{t>0: \mathbb{E}\left(e^{ \left\|X- \E(X)\right\|_2^2  / (t/2L)^{2}}\right)\leq 2\right\}\\
 &= 2L \inf \left\{t>0: \mathbb{E}\left(e^{ \left\|X- \E(X)\right\|_2^2  / t^{2}}\right)\leq 2\right\}\\
 &= 2L  \left\|X - \E X\right\|_{2,\psi_2}.
\end{align*}
\end{proof}

Now, we can prove Lemma \ref{lem:process_subG}. 

\begin{proof}[Proof of Lemma \ref{lem:process_subG}]
We prove by induction that $\|X^j-\E(X^j)\|_2$ is sub-Gaussian.

For the base case, $\left\|X^0\right\|_{2,\psi_2} \leq \sigma_X < \infty$ by Assumption \ref{assumption:dyn_assumptions}. By $\|X^0 - \E(X^0)\|_2 \leq \|X^0\|_2 + \E(\|X^0\|_2)$ we have $\|X^0 - \E(X^0)\|_{2,\psi_2}\leq 2\sigma_X$ (see also Exercise~2.40e of \cite{vershynin2025high} for tighter centering bounds). Let $j\geq 0$ and $\|X^{j} - \E(X^{j})\|_2$ be sub-Gaussian with norm bounded by a constant depending on the parameters in the statement. We have $X^{j+1} = \Psi(X^j) + \xi^j$ and $\|\Psi(X^j) - \E(\Psi(X^j))\|_2$ is  sub-Gaussian by Proposition \ref{prop:subg_lip}.  Further, 
$$
\|\Psi(X^j) + \xi^j - \E(\Psi(X^j) + \xi^j)\|_2 \leq \|\Psi(X^j) - \E (\Psi(X^j))\|_2 +  \|\xi^j - \E(\xi^j)\|_2
$$
and by triangle inequality of the Orlicz norm the right-hand-side is  sub-Gaussian. By sub-Gaussian summation properties, this also shows that the left-hand side $\left\|X^{j+1} - \E(X^{j+1})\right\|_2$ is sub-Gaussian. This shows that all $X^j, 1\leq j \leq J$ are sub-Gaussian with sub-Gaussian norm upper bounded by a parameter depending only on $\sigma_X, \sigma_{\max}, L_\Psi, J$. 
$\|Y^j - \E(Y^j)\|_2$
is sub-Gaussian by the same argument since  $Y^j = h(X^j) + \eta^j,$ $h$ is Lipschitz, and $\eta^j$ is sub-Gaussian. The final result follows with $$
\left\|(X^j, Y^j) - \E(X^j, Y^j)\right\|_2 \leq \left\|X^j- \E(X^j)\right\|_2+ \left\|Y^j- \E(Y^j)\right\|_2.
$$
We emphasize that each of the cited results demonstrates the stability of sub-Gaussian behavior under operations such as addition, upper bounding, and Lipschitz transformations, with constants that depend solely on the parameters characterizing the sub-Gaussian and Lipschitz parameters. 
\end{proof}

Further, we used Proposition \ref{prop:subG_stability_cond_exp} in the proof of Proposition \ref{prop:process_covariance} and here, we give a short proof. 

\begin{proof}[Proof of Proposition \ref{prop:subG_stability_cond_exp}]
By Jensen's inequality,
\begin{align*}
    e^{\left\|\E(X|Y) / t\right\|_2^{2}}
    \le  \mathbb{E}\left(e^{\|X / t\|_2^{2}}\,|\,Y\right) 
\end{align*}
conditional on $Y$, and hence
\begin{align*}
  \left\|\E(X|Y)\right\|_{2,\psi_2} &=   \inf \left\{t>0: \mathbb{E}\left(e^{\left\|\E(X|Y) / t\right\|_2^{2}}\right) \leq 2\right\} \\
  &\leq  \inf \left\{t>0: \mathbb{E}\left(e^{\|X / t\|_2^{2}}\right) \leq 2\right\} \\
  &= \|X\|_{2,\psi_2} .
\end{align*}
\end{proof}

To prove Proposition \ref{prop:subG_stability_conditioning}, we rely on the following lemma. 

\begin{lemma}
\label{lem:simple_exponential_sum}
Let $c >0$. Then,
$$
\sum\limits_{n = 0}^\infty \exp(-c2^{2n}) \leq \exp({-c})\left(1+\frac{1}{c\log 4}\right). 
$$
\end{lemma}

\begin{proof}
For $c>0$, let $f(x)=\exp({-c4^x})$. Since $f$ is positive and decreasing on $[0,\infty)$,
\[
\sum_{j=0}^\infty f(j)
\le f(0)+\int_0^\infty f(x)\,dx
= \exp({-c})+\int_0^\infty \exp(-c4^x)\,dx.
\]
With $t=4^x$ so $dx=\frac{dt}{t\log 4}$,
\[
\int_0^\infty \exp(-c4^x)dx
= \int_{1}^{\infty}\frac{\exp(-ct)}{t\log 4}dt
= \frac{1}{\log 4}\int_{c}^{\infty}\frac{\exp(-u)}{u}du
\le \frac{\exp(-c)}{c\log 4}.
\]
Hence
\[
\sum_{j=0}^\infty \exp(-c4^j)
\le \exp({-c})\left(1+\frac{1}{c\log 4}\right).
\]

\end{proof}

Now, we prove our new characterization of sub-Gaussianity.

\begin{proof}[Proof of Proposition \ref{prop:subg_mgf_away_from_0}]
Assume there exists a constant $K_5$ such that
for any $|t|\geq 1/K_5$ it holds that 
$\E(\exp(tX)) \leq \exp\left( t^2K_5^2\right).$ 
Fix $x\geq 2K_5$. For any $|\lambda|  \geq1/K_5$,  a Chernoff bound gives
\[
  \P\bigl(X \geq x\bigr)
  \le
  \exp(-\lambda x)
  \E\left(\exp(\lambda X)\right) 
  {\le}
  \exp\!\Bigl(
      -\lambda x + \lambda^{2}K_5^{2}
  \Bigr).
\]
Minimizing the exponent in \(\lambda\) gives \(\lambda^{\star}=x/(2K_5^{2})\), which is valid for the Chernoff bound since $|\lambda^{\star}|=|x|/(2K_5^{2})  = x/(2K_5^{2})  \geq1/K_5$. Hence, $\P\bigl(X \geq x\bigr)
  \le
  \exp\left(-\frac{x^{2}}{4K_5^{2}}\right).$
It follows by applying the same argument to $-X$ and a union bound that for all $x \geq  2K_5$, $\P\bigl(|X | \geq x\bigr)
  \le
  2\exp\!\left(-\frac{x^{2}}{4K_5^{2}}\right).$
For $K_2 = \frac{2K_5}{\sqrt{\log 2}}$, we have that 
$$
\P\bigl(|X | \geq x\bigr)
  \le
  2\exp\!\left(-\frac{x^{2}}{K_2^2}\right) 
$$
for all $x\geq 0$: For $x \geq 2K_5$ this follows from what we just showed, i.e., 
$\P\bigl(|X | \geq x\bigr)
  \le
  2\exp\left(-\frac{x^{2}}{4K_5^2}\right) \leq 2\exp\left(-\log 2\frac{x^{2}}{4K_5^2}\right) = 2\exp\left(-\frac{x^{2}}{K_2^2}\right)$, as $\log 2 < 1$. For $0 \leq x\leq 2K_5$, we have 
  $2\exp\left(-\frac{x^{2}}{K_2^2}\right) \geq 2\exp\left(-\log 2\frac{4K_5^2}{4K_5^2}\right) = 1$, so it is trivially true. 
This shows sub-Gaussianity by Proposition \ref{prop:equiv_subg} with $K_2 = \frac{2}{\sqrt{\log 2}}K_5$. Conversely, let $X$ be sub-Gaussian and write \(K_1=\|X\|_{\psi_2}\). By Proposition~\ref{prop:equiv_subg}, the tail bound in Property 2 holds for some \(K_2\le \tilde C_{\psi_2}K_1\). Then, the moment generating function exists for any $t \in \R$. The case \(t=0\) is immediate, so let $t > 0$ and apply the layer-cake representation as well as Property 2 in Proposition \ref{prop:equiv_subg}:
\begin{align*}
    \E\exp(tX) & = \int\limits_0^\infty \P\left(\exp(tX)\geq s\right)\dd s \\
&\leq 1 + \int\limits_1^\infty \P\left(\exp(tX)\geq s\right)\dd s\;=\; 1 + \int\limits_1^\infty \P\left(X\geq \frac{\log  s }{t}\right)\dd s \\
&= 1 + t\int\limits_0^\infty \exp(tu)\P\left(X\geq u\right)\dd u \;\leq\;  1 + 2t\int\limits_{0}^\infty \exp(tu)\exp\left(-\frac{u^2}{K_2^2}\right)\dd u\\
&\leq  1 + 2t\int\limits_{-\infty}^\infty \exp(tu)\exp\left(-\frac{u^2}{K_2^2}\right)\dd u \;=\;  1 + 2\sqrt{\pi t^2K_{2}^2} \cdot \exp \left(\frac{t^{2} K_{2}^2}{4}\right) \\
\end{align*}
Now, note that the function $\alpha x\exp(-\beta x^2)$ restricted to $x\geq 0$  has its maximum value at $\frac{\alpha}{\sqrt{2\beta}}\exp(-1/2)$. Choosing $\alpha = 2 \sqrt{\pi}$ and $\beta = 2\pi$ yields $2 \sqrt{\pi} x \exp(-2\pi x^2) \leq 1$ for all $x \geq 0$. Therefore, we can further bound, for $x = K_2 t$,
\begin{align*}
    \E\exp(tX) &\leq  1 +  2\sqrt{\pi} x \exp\left(- 2\pi x^2\right) \cdot \exp \left( \frac{x^2}{4} + 2\pi x^2\right) \\
    & \leq 1 + 1 \cdot \exp \left(\left(\frac{1}{4} + 2\pi\right){t^{2} K_{2}^2}\right) \\
    &\leq   \exp \left(\underbrace{\left(\frac{1}{4} + 2\pi\right)K_2^2}_{=:\tilde K_2^2}\cdot {t^{2}} + \log 2\right).
\end{align*}
We can generalize this to all $t\in \R$ by applying the same argument to $-X$ which is sub-Gaussian with the same parameter to obtain 
$$
\E\exp(tX)\leq \exp \left(\tilde K_2^2 {t^{2}} + \log 2\right) \forall t\in\R.
$$
Now, choose $K_5 = 4  \tilde K_2$. Let  $|t|\geq \frac{1}{K_5}$. We have
\begin{align*}
t^2 K_5^2 -     \left({t^{2} \tilde K_{2}^2}+ \log2\right)& = t^2\tilde K_2^2\left( 16-    1\right)- \log2 \;\geq\; \frac{15}{16} - \log2 \;\geq\; 0.
\end{align*}
In particular,
$$
\E\exp(tX) \leq \exp(K_5^2t^2), \quad  \forall |t| \geq \frac{1}{K_5}
$$
and we are done with the reverse bound since $K_5 = 4\sqrt{1/4+2\pi}K_2$. 
Thus, with $$C_{\psi_2}\coloneqq \max\left(1,\tilde C_{\psi_2},\frac{2\tilde C_{\psi_2}}{\sqrt{\log 2}},4\sqrt{1/4+2\pi}\tilde C_{\psi_2}\right)\,,$$ the first direction gives \(\|X\|_{\psi_2}\le C_{\psi_2}K_5\) whenever \eqref{eq:subG_mgf} holds with parameter \(K_5\), while the reverse direction gives an admissible \(K_5\le C_{\psi_2}\|X\|_{\psi_2}\).
We conclude by showing the monotonicity property.  Assume Equation \eqref{eq:subG_mgf} holds and  let $\tilde K_5 > K_5$. For $|t| \geq 1/K_5$ there is nothing to show, so let  $1/\tilde K_5 \leq |t| \leq 1/K_5$. Without loss of generality, $t\geq 0$ since we can apply the same argument to $-X$.   Decompose convexly as $t = (1-p) \cdot 0 +  p\cdot 1/K_5$, so \(p=K_5t\). By log-convexity of the moment generating function, we have that 
$$
\log \E(\exp(tX)) \leq p  \log \E(\exp(X/K_5)) \leq p  \log e = p.
$$
We are done since \(p=K_5t\le \tilde K_5t\le \tilde K_5^2t^2\), where the last inequality uses \(t\ge 1/\tilde K_5\). 
\end{proof}

\begin{remark}
Alternatively, the reverse direction follows directly from the Orlicz characterization: if $\E\exp(X^{2}/K_{1}^{2})\le 2$, then the inequality $ab\le \tfrac{a^{2}}{2}+\tfrac{b^{2}}{2}$ (applied to $a=\sqrt{2}X/K_{1}$ and $b=tK_{1}/\sqrt{2}$) gives
\[
e^{tX}\le \exp\left(\frac{X^{2}}{K_{1}^{2}}\right)\exp\left(\frac{K_{1}^{2}t^{2}}{4}\right),
\]
and taking expectations immediately yields the desired mgf bound for large $|t|$ - the proof can be completed for small $|t|$ after some additional algebra.
\end{remark}

\subsection[Auxiliary Result for Mean-Field Moments]{Auxiliary Result for Subsection \ref{subsection:moments_mean_field}}
\label{subsection:aux_results_mf_moments}
We present a proof of Proposition \ref{prop:mf_dynamics_higher_moments}.

\begin{proof}[Proof of Proposition \ref{prop:mf_dynamics_higher_moments}]
Fix the observation path $y^{1:J}$ throughout this proof. Construct the mean-field dynamics $\tilde Z^j_f\coloneqq (\tilde X^j_f, \tilde Y^j_f) \sim \tilde \mu^{f,j}_{XY}$  at different times on a common probability space by drawing independently $X^0\sim \mu^0$, $\{\xi^{t}\}_{t \ge 0}$ with $\xi^{t} \sim \nu_\xi^t$, $\{\eta^{t}\}_{t \ge 1}$ with $\eta^{t} \sim \nu_\eta^{t}$, $\{\omega^{t}\}_{t \ge 1}$ with $\omega^{t} \sim \kappa$, setting $\tilde X^{a,0}\coloneqq X^0$, 
and defining recursively, for $j \ge 1$, by
\[
\tilde X_f^{j} = \Psi(\tilde X^{a,j-1}) + \xi^{j-1},
\qquad
\tilde Y^{f,j} = h(\tilde X_f^{j}) + \eta^{j},
\qquad
\tilde X^{a,j} = \map_{y^{j}}^{\tilde\mu_{XY}^{f,j}}\big(\tilde Z_f^{j}, \omega^{j}\big). 
\]
Write $s^j  \coloneqq  \big\|\tilde Z_f^{j} - \mathbb{E}\tilde Z_f^{j}\big\|_{2,\psi_2}$ and $a^j  \coloneqq  \big\|\tilde X^{a,j} - \mathbb{E}\tilde X^{a,j}\big\|_{2,\psi_2}$. Let
\[
\Phi(x,\xi,\eta) \coloneqq  \big(\Psi(x) + \xi, h(\Psi(x) + \xi) + \eta\big),
\qquad
L_{PQ} \coloneqq  (1 + L_h)(1 + L_\Psi).
\]
We induct an upper bound $d^j$ on $s^j$ over $j\geq 1$. By Assumption \ref{assumption:dyn_assumptions}, $\Phi$ is $L_{PQ}$-Lipschitz, and
\[
a^0 = \|\tilde X^{a,0}- \mathbb{E}\tilde X^{a,0}\|_{2,\psi_2} \le 2\sigma_X, \quad
\|\xi^{j-1}\|_{2,\psi_2} \le \sigma_{\max}, \quad
\|\eta^{j}\|_{2,\psi_2} \le \lambda_{\max}
\]
for all $j \in \N_{\geq 1}$. Further, by Proposition \ref{prop:subg_lip} with $f=\Phi$,
\begin{equation}
\label{eq:s_and_a_equation}
s^j  \le 2 L_{PQ}  \big(a^{j -1} +2 \sigma_{\max} + 2\lambda_{\max}\big).
\end{equation}
In particular, this implies that 
$$
s^1 \le 4 L_{PQ}  \big(\sigma_X + \sigma_{\max} + \lambda_{\max}\big) =: d^1.
$$
For the inductive step, we use that $\mathrm{Tr}\mathrm{Cov}(\tilde\mu_{XY}^{f,j-1}) = \mathbb{E}\|\tilde Z_f^{j-1} - \mathbb{E}\tilde Z_f^{j-1}\|_2^2 \le 2( s^{j-1})^2$ so that Assumption \ref{assumption:algorithm_finite_particle} (1a) yields
\[
\mathrm{Lip}\big(\map_{y^{j-1}}^{\tilde\mu_{XY}^{f,j-1}}\big)
\le C_{\mathrm{L}} \big(1 + 2^{e_{\mathrm{L}}}( s^{j-1})^{2 e_{\mathrm{L}}}\big).
\]
Thus, by Proposition \ref{prop:subg_lip} 
\begin{align*}
 a^{j -1}
&= \left\|\map_{y^{j-1}}^{\tilde\mu_{XY}^{f,j-1}}(\tilde Z_{j-1}^f, \omega^{j-1}) - \mathbb{E}\left(\map_{y^{j-1}}^{\tilde\mu_{XY}^{f,j-1}}(\tilde Z_{j-1}^f, \omega^{j-1}) \right)\right\|_{2,\psi_2}\\
&\le 2  C_{\mathrm{L}} \big(1 + 2^{e_{\mathrm{L}}}( s^{j-1})^{2 e_{\mathrm{L}}}\big)\cdot \big(s^{j-1} + \sigma_\kappa\big).   
\end{align*}
Combining this with Equation \eqref{eq:s_and_a_equation} gives  the  recursion
\[
s^j  \le A \big(1 + (s^{j -1})^{2 e_{\mathrm{L}}}\big)\big(s^{j -1} + \sigma_\kappa\big) + B
\]
for fixed constants $A, B\geq 0$ depending only on $L_{PQ} $,  $\sigma_{\max},  \lambda_{\max}, C_{\mathrm{L}},e_{\mathrm{L}}, $ and $\sigma_\kappa$   
with initial bound $s^1 \leq d^1$. Define for $j \ge 2$,
\[
d^j \coloneqq  A \big(1 + (d^{j-1})^{2 e_{\mathrm{L}}}\big)\big(d^{j-1} + \sigma_\kappa\big) + B.
\]
By monotonicity in $d^{j-1}$, we have $s^j  \le d^j$ for all $j \le J$. Since the recursion is at most polynomial, $d^j < \infty$ for each finite $J$. Setting
\[
\tilde C_{\mathrm{subG}} \coloneqq  \max_{1 \le j \le J} d^j
\]
gives $s^j  \le \tilde C_{\mathrm{subG}}$ for all $j \le J$, proving the claimed $\psi_2$ control.
\end{proof}

\subsection[Auxiliary Result for Particle Convergence]{Auxiliary Result for Subsection~\ref{subsection:convergence_proof}}
\label{subsection:fp_aux}
We applied the following result in the proof of Theorem \ref{thm:ips_converges_to_iid_particles}. 

\begin{lemma}
\label{lem:cov_perturbation_W2}
Let $g:\R^{n+m}\to\R^{b}$ be $L_g$-Lipschitz. Then for any $\nu,\mu\in \Pp_2(\R^{n+m})$,
\[
\bigl\|\Cov(g_\sharp\nu)-\Cov(g_\sharp\mu)\bigr\|_2
\ \le\ 2L_g\big(\overline M_2(g_\sharp\nu)+\overline M_2(g_\sharp\mu)\big)W_2(\nu,\mu).
\]
\end{lemma}

\begin{proof}
Define $\Delta\coloneqq \Cov(g_\sharp\nu)-\Cov(g_\sharp\mu)$. Write $m_\nu\coloneqq \E g_\sharp \nu$ and $m_\mu\coloneqq \E g_\sharp\mu$. For any coupling $\pi\in\Pi(\nu,\mu)$,
\[
\Delta
=\int\Big((g(z)-m_\nu)(g(z)-m_\nu)^\top-(g(z')-m_\mu)(g(z')-m_\mu)^\top\Big)d\pi(z,z').
\]
Using $\|aa^\top-bb^\top\|_2\le(\|a\|_2+\|b\|_2)\|a-b\|_2$, we get
\[
\bigl\|\Delta\bigr\|_2
\le \int \big(\|g(z)-m_\nu\|_2+\|g(z')-m_\mu\|_2\big)\|g(z)-g(z')+m_\mu-m_\nu\|_2d\pi(z,z').
\]
Split the last norm and apply Cauchy--Schwarz:
\[
\begin{aligned}
&\int \big(\|g(z)-m_\nu\|_2+\|g(z')-m_\mu\|_2\big)\|g(z)-g(z')\|_2d\pi(z,z') \\
\le& \Big(\E_\nu\|g-m_\nu\|_2^2+\E_\mu\|g-m_\mu\|_2^2\Big)^{\!1/2}
      \Big(\int\|g(z)-g(z')\|_2^2d\pi(z,z')\Big)^{1/2} \\
\le &\big(\overline M_2(g_\sharp\nu)+\overline M_2(g_\sharp\mu)\big)L_g\Big(\int\|z-z'\|_2^2d\pi(z,z')\Big)^{\!1/2}.
\end{aligned}
\]
Noting that \[
\|m_\mu-m_\nu\|_2\le L_g\Big(\int\|z-z'\|_2^2d\pi(z,z')\Big)^{1/2}.
\]
and combining the previous inequalities, 
\[
\bigl\|\Delta\bigr\|_2
 \le 2L_g\big(\overline M_2(g_\sharp\nu)+\overline M_2(g_\sharp\mu)\big)
      \Big(\int\|z-z'\|_2^2d\pi(z,z')\Big)^{1/2}.
\]
Optimizing over $\pi$ yields the bound with $W_2(\nu,\mu)$.
\end{proof}

\subsection[Auxiliary Results for the EnSMF Application]{Auxiliary Results for Section \ref{subsection:application_ensmf}}

In this section, we provide proofs for the stated results about the inverse of the stochastic map and well-posedness of the cross-entropy optimization problem in Equation~\eqref{eq:cross_entropy_min_sm}.

\begin{lemma}
\label{lem:inverting_sm}
For every $S=\mathcal S(\alpha,c,\theta)\in\mathcal{S}$ with \(\alpha\in\R_{\ge\alpha_{\min}}^n\), and every $y\in \R^m$, the function $S(\cdot, y): x \mapsto S(x,y)$ defined in Subsection~\ref{subsection:application_ensmf} with the conditions in Assumption~\ref{assumption:sm_assumption} is a diffeomorphism.
\end{lemma}

\begin{proof}
For any $z \in \R^n, y \in\R^m$, we will show that the equation $S(x, y) = z$ has a unique solution. Since $\alpha_k > 0$, we can rewrite component $k$ of the inverse map as 
\begin{align*}
x_k  &= \frac{1}{\alpha_k} \left(z_k - c_k - \sum_{i=1}^{\iota_k} \theta_k^i f_k^i(x_{1:k-1},y)\right).
\end{align*}
The existence result for the inverse then follows by induction. 

We now show the smoothness of the inverse map. For $S = \mathcal{S}(\alpha, c, \theta)$ with $\alpha \in \R_{\geq \alpha_{\min}}^n$, the Jacobian determinant is lower bounded by construction:
\[
|\det\nabla_x S(x,y)|\geq \alpha_{\min}^n.
\]
Then, the inverse function theorem guarantees that the inverse is continuously differentiable.
\end{proof}

We now prove Proposition \ref{prop:existence_minimizer_sm_KL}, which we restate here first. 
\propexistenceminimizer*

\begin{remark}
$R_k(\mu)= 0$ if and only if $X_k$ is almost surely an affine linear function of  $f_k(X,Y)$. In this case, the state variable $X_k$ contains no more information than $(X_{1:k-1}, Y)$, i.e., we are over-parameterizing our state space. In this regard, Proposition~\ref{prop:existence_minimizer_sm_KL} guarantees that, provided the state space is parametrized appropriately, the regression task is well-posed. 
\end{remark}

\begin{proof}[Proof of Proposition \ref{prop:existence_minimizer_sm_KL}]
For any $y$,
the density of $\left(S(\cdot, y)\right)^\sharp\gamma_{\text{G}}$ is given by
$$
 \left(\left(S(\cdot, y)\right)^\sharp\gamma_{\text{G}}\right)(x) =  {(2\pi)^{-n/2}} \exp\left(-\left\|S(x,y)\right\|_2^2/2\right) |\det\nabla_x S(x,y)|,
$$
where $|\det\nabla_x S(x,y)| = |\prod_{k = 1}^n \alpha_k|$ is the Jacobian determinant of $S(\cdot, y)$. Thus, the cross-entropy objective in Equation~\eqref{eq:cross_entropy_min_sm} is given by 
\begin{align*}
 J(S) :=&-\int \dd\mu(x,y) \log\left(\left(S(\cdot, y)\right)^\sharp\gamma_{\text{G}}\right)(x) \\
 =& \int \dd\mu(x,y) \left(\left\|S(x, y) \right\|_2^2/2 - \log|\det\nabla_x S(x,y)| \right) + \frac{n}{2}\log(2\pi) \\
 =& \sum\limits_{k = 1}^n\int \dd\mu(x,y) \left(|S_k(x, y)|^2/2 - \log|\partial_{x_k}S_k| \right) + \frac{n}{2}\log(2\pi).
\end{align*}
For $S = \mathcal{S}(\alpha,  c, \theta)$, we can rewrite the $k$-th summand in terms of only the $k$-lower indexed parameters. Thus, the cross-entropy objective is separable into objectives $J_k$ for each map component $k=1,\dots,n$ given by:
\begin{align*}
J_k\left(\alpha_k, \tilde c_k, \{\tilde \theta_k^i\}_{i = 1,\ldots, \iota_k}\right) &= \int \dd\mu(x,y) \left(|S_k(x, y)|^2/2 - \log|\partial_{x_k}S_k| \right) \\
&= \frac{\alpha_k^2}{2}\int \dd\mu(x,y)\left|x_k  + \tilde c_k + \sum_{i=1}^{\iota_k} \tilde  \theta_k^i f_k^i(x_{1:k-1},y)\right|^2 - \log\alpha_k,
\end{align*}
where $\tilde \theta^i_k = \frac{\theta^i_k}{\alpha_k},  \tilde \theta_k \coloneqq ( \tilde \theta_k^1, \ldots,  \tilde \theta_k^{\iota_k})^\top$; see \cite{parno2018transport} and \cite{baptista2024representation} for similar derivations of this fact.

It follows that for the map to have a unique minimizer, it is sufficient to show that for every $k = 1,\ldots, n$, the function $J_k$
constrained to $\R_{\geq \alpha_{\min}} \times \R \times \R^{\iota_k}$ attains a minimum. $J_k$ is of the form
$$
 J_k\left(\alpha_k, \tilde c_k, \{\tilde \theta_k^i\}_{i = 1,\ldots, \iota_k}\right) 
 = \frac{\alpha_k^2}{2}\tilde J_k(\tilde c_k, \{\tilde \theta_k^i\}_{i = 1,\ldots, \iota_k})- \log \alpha_k
$$
for
$$
\tilde J_k(\tilde c_k, \{\tilde \theta_k^i\}_{i = 1,\ldots, \iota_k}) =\int \dd\mu(x,y) \left|x_k  + \tilde c_k + \sum_{i=1}^{\iota_k} \tilde  \theta_k^i f_k^i(x_{1:k-1},y)\right|^2.
$$
As $J_k$ is strictly monotone in $\tilde J_k$ and depends on $\tilde c_k$ and $\{\tilde \theta_k^i\}_{i=1,\ldots,\iota_k}$ only through $\tilde J_k$, continuity allows us to take the minimum over $\tilde J_k$, provided it exists, before taking the infimum over $\alpha_k$. By the $L_2$-characterization of the expectation, for every fixed $\{\tilde \theta_k^i\}_{i=1,\ldots,\iota_k}$, there is a unique minimizer in $\tilde c_k$, given by
$$
\tilde c_k =- \int \dd\mu(x,y) \left(x_k  + \sum_{i=1}^{\iota_k} \tilde  \theta_k^i f_k^i(x_{1:k-1},y)\right).
$$
Therefore, we must choose $\{\tilde \theta_k^i\}_{i = 1,\ldots, \iota_k}$ to minimize 
$$
\int \dd\mu(x,y) \left| \overline x_k  +  \sum_{i=1}^{\iota_k} \tilde  \theta_k^i  \overline f_k^i(x_{1:k-1},y)\right|^2
$$
where 
\begin{align*}
     \overline x_k   &\coloneqq x_k -  \int \dd\mu(x^\prime,y^\prime)x^\prime_k \\
  \overline f_k^i(x_{1:k-1},y) & \coloneqq  f_k^i(x_{1:k-1},y) - \int \dd\mu(x^\prime,y^\prime)  f_k^i(x^\prime_{1:k-1},y^\prime).
\end{align*}
We can rewrite this as the unconstrained quadratic program
$$
\inf_{\tilde \theta_k \in \R^{\iota_k}} \left(\Cov\left(X_k\right)  + \tilde \theta_k^\top \Cov\left( f_k(X,Y) \right) \tilde \theta_k + 2\Cov\left( f_k(X,Y), X_k \right)^\top \tilde \theta_k\right).
$$
This quadratic program has a unique minimizer if and only if 
$\Cov\left(f_k(X,Y) \right)  \succ 0,$ which is given by 
$$
\tilde \theta_k =  -  \Cov\left( f_k(X,Y) \right)^{-1} \Cov\left(  f_k(X,Y), X_k \right).
$$
This shows that $\tilde J_k$ has a unique minimizer that achieves
$$
\tilde J_k(\tilde c_k, \{\tilde \theta_k^i\}_{i = 1,\ldots, \iota_k}) = R_k(\mu) . 
$$
Note that $ R_k(\mu)  > 0$ by assumption.  So far, we have shown that 
$$
 \inf\limits_{\alpha_k \in \R_{\geq \alpha_{\min}}, \tilde c_k\in \R, \tilde \theta_k \in \R^{\iota_k}}J_k\left(\alpha_k, \tilde c_k, \tilde \theta_k\right)
 =  \inf\limits_{\alpha_k \in \R_{\geq \alpha_{\min}}}\frac{\alpha_k^2}{2} R_k(\mu)- \log \alpha_k
$$
and that this infimum is uniquely achieved in $\tilde c_k, \tilde \theta_k$.  The function $Cx^2 - \log x$ is convex for every $C > 0$ and therefore the infimum is also achieved in $\alpha_k$, uniquely given by
\begin{equation}
\label{eq:alpha_first_order}
    \alpha_k =  \max\left(\frac{1}{\sqrt{R_k(\mu)}}, \alpha_{\min}\right).
\end{equation}
\end{proof}

\subsection[Implementation of the Regularized Estimator]{Implementation of the Regularized Estimator in Subsection \ref{subsection:application_ensmf}}
\label{subsection:implementation_regularization}
In this section, we demonstrate that inverting a regularized covariance estimator of size $\iota \times \iota$ for a small ensemble of size $N$ can be done at a cost of $\mathcal{O}(\iota N^2)$ without forming the full sample covariance matrix. 
For a fixed component $k$, let $A \in \R^{\iota_k \times N}$ be the ensemble matrix containing $\iota_k$ feature evaluations, i.e., $A_{\cdot,\ell} = f_k(x_{\ell}^{f,j}, y_{\ell}^{f,j}) \in \R^{\iota_k}$ for $\ell = 1, \ldots, N$. We let $\bar A  \in \R^{\iota_k \times N}$ denote the centered matrix computed by subtracting the  feature mean from each row. The centered sample covariance is then
\[
C\coloneqq\frac{1}{N}\bar{A}\bar A^\top \in \R^{\iota_k \times \iota_k}.
\]
Algorithm~\ref{alg:stable_theta} presents a numerically stable and computationally efficient method for solving the regularized linear system $\theta =C_{\hat\sigma_f}^{-1}b,$ for any \(b\in\R^\iota\), without ever forming the full sample covariance matrix. For simplicity, we assume \(N\le\iota\). For \(N>\iota\), direct computation is less costly, so this shortcut is unnecessary. Computing the vector $\theta$ has a total complexity based on the $\text{cost of SVD}({\bar A}) + \mathcal{O}(\iota N),$ which is $\mathcal{O}(\iota N^2)$ for $N \leq \iota$.
\begin{algorithm}[t]
\caption{Stable computation of $\theta = C_{\hat{\sigma}_f}^{-1}b$}
\label{alg:stable_theta}
\begin{algorithmic}[1]
\Require Centered matrix $A \in \mathbb{R}^{\iota \times N}$ such that $C\coloneqq\frac{1}{N}{A} A^\top $, vector $b \in \mathbb{R}^\iota$, threshold $\hat{\sigma}_f$, assuming $N \leq \iota$
\State $U, D, V \leftarrow \text{thin SVD}\left(\frac{1}{\sqrt{N}}{A}\right)$, with diagonal $D \in \mathbb{R}^{N \times N}$
\State $\left(D_{\hat{\sigma}_f}\right)_{ii} \leftarrow \max(D_{ii}^2, \hat{\sigma}_f)$ 
\State $y \leftarrow U^\top b$
\State $z \leftarrow D_{\hat{\sigma}_f}^{-1} y$
\State $\theta \leftarrow U z + \hat{\sigma}_f^{-1} \left(b - U y\right)$
\end{algorithmic}
\end{algorithm}

\subsection[Auxiliary Results for EnSMF Convergence]{Auxiliary Results for Section \ref{subsubsection:convergence_ensmf}}
\label{subsection:convergence_ensmf_aux}

In this section, we give a proof of Lemma~\ref{lem:smf_satisfies_assumptions}, which we restate here first.
\lemsmfassumptions*
Some of the constants for the guarantees on the map $\map_y^{\text{SM},\mu}$ are given in Table~\ref{tab:smf_constants_values}. 
We split the proof of this lemma into three forthcoming results:
\begin{enumerate}
    \item Assumption \textit{(1a)} is established in Lemma \ref{lem:sm_lipschitz_alg}.
\item Assumption \textit{(1b)} is shown in Proposition \ref{prop:sm_estimation_stability_map}.
\item Assumption \textit{(3)} is proved in Lemma \ref{lem:distance_x_to_beta_tilde}. 
\end{enumerate}

\begin{lemma}
\label{lem:sm_lipschitz_alg}
Suppose Assumption \ref{assumption:sm_assumption}  holds.  Let $\mu \in \Pp_2\left(\R^n\times\R^m\right)$. For any $y^\star \in \R^m$, the map $\map^{\text{SM},\mu}_{y^\star}: \R^n \times \R^m \rightarrow \R^n$ is Lipschitz-continuous with constant
$$
 \tilde L^\mu = \tilde C_\text{II}^\text{SM}\left(1 + \Trcov(\mu)^{\tilde e_{\text{II}}^\text{SM}}\right),
$$
where
\begin{equation}
   \tilde C_{\text{II}}^\text{SM} = \sqrt{n}2^{n-1} \left({2}^{n/2} +  4^n\hat\sigma_f^{-n}L_f^{2n}\right),\quad  \tilde  e_{\text{II}}^\text{SM} = n.
\end{equation}
\end{lemma}

\afterpage{%
\begin{table}[t]
\centering
\setlength{\tabcolsep}{4pt} 
\renewcommand{\arraystretch}{1.2} 
\begin{tabular}{ccccccc}
\toprule
 $e_{\mathrm{L}}$ & $g(x,y)$& $b$ & $L_g$ &  $e_{\mathrm{est},1}$  & $a$ & $\sigma_\kappa$ \\
\midrule
$n$ &  $(x, f_1(x, y), \ldots, f_n(x,y))$ &$n + \sum\limits_{k=1}^n\iota_k$ & $\sqrt{1+nL_f^2}$ & $2(n+1)$ & $0$ & $0$ \\
\bottomrule
\end{tabular}
\caption{Constants for the ensemble stochastic map filter in Assumption~\ref{assumption:algorithm_finite_particle}.}
\label{tab:smf_constants_values}
\end{table}
}

\begin{proof}
Fix $x, x^\prime\in\R^n, y, y', y^\star \in \R^m$ and let $w =  \map^{\text{SM},\mu}_{y^\star}(x,y), w '=  \map^{\text{SM},\mu}_{y^\star}(x',y')$.  By definition of Equation \eqref{eq:sm_approximate_conditioning_recursion},  we have the recursion 
$$
w_k  = x_k    - \theta_k^\top\left(f_k(x,y) -f_k(w,y^\star)\right)
$$
with $\theta_k \coloneqq \Cov_{\hat \sigma_f}\left( f_k( X, Y) \right)^{-1} \Cov\left(  f_k( X, Y),  X_k \right)$ and $(X,Y) \sim \mu$. Therefore, 
\begin{align*}
    |w_k - w_k'|^2
    &= \left|x_k  - x_k'
    - \theta_k^\top\left(f_k(x,y) - f_k(x',y')
    -(f_k(w,y^\star) - f_k(w',y^\star)) \right) \right|^2\\
    &\leq 2 \left|x_k  - x_k'\right|^2
    +4 \|\theta_k\|_2^2\Bigl(
    \left\|f_k(x,y) - f_k(x',y')\right\|^2_2\\
    &\qquad\qquad
    + \left\|f_k(w,y^\star) - f_k(w',y^\star) \right\|^2_2\Bigr)\\
    &\leq 2 \left|x_k  - x_k'\right|^2
    +4 \|\theta_k\|_2^2L_f^2\Bigl(
    \left\|(x_{1:k-1},y) - (x_{1:k-1}',y')\right\|^2_2\\
    &\qquad\qquad
    + \left\|w_{1:k-1} - w_{1:k-1}' \right\|^2_2\Bigr),
\end{align*}
where the last inequality follows from $f_k$ being $L_f$-Lipschitz. 
Let $\Delta_{1:k}^2 := \|w_{1:k} - w_{1:k}'\|_2^2$ and $\tilde \Delta^2 := \|(x,y) - (x',y')\|_2^2$. Then, another upper bound for the result above implies the recursion:
\begin{align*}
    \Delta_{1:k}^2 \leq \left(2 + 4 \|\theta_k\|_2^2L_f^2\right)\tilde \Delta^2  +  \left(2 +  4 \|\theta_k\|_2^2L_f^2 \right)\Delta_{1:k-1}^2
\end{align*}
The solution to this recursion is 
\begin{align*}
\Delta_{1:k}^2 & \leq \tilde \Delta^2\sum\limits_{\ell = 1}^k  \left(\prod\limits_{g =\ell}^{k}\left(2 + {4\left\| \theta_g\right\|_2^2 L_f^2} \right)\right) \\
& \leq \tilde \Delta^2 k \prod\limits_{g =1}^{k}\left(2 + {4\left\| \theta_g\right\|_2^2 L_f^2} \right).
\end{align*}
Finally, we need to bound $\|\theta_k\|_2^2$. By applying Cauchy-Schwarz and bounding the change in moments under the pushforward of Lipschitz transformations, we have
\begin{align}
    \left\|\theta_k \right\|_2^2& = \left\|\Cov_{\hat \sigma_f}\left( f_k( X, Y) \right)^{-1} \Cov\left(  f_k( X, Y),  X_k \right)\right\|_2^2 \nonumber \\
    & \leq \hat\sigma_f^{-2}\left\|\Cov\left(  f_k( X, Y),  X_k \right) \right\|_2^2 \nonumber \\
    & \leq 4\hat\sigma_f^{-2}L_f^2 \Trcov(\mu)^2. \label{eq:theta_k_norm_bound}
\end{align}
Therefore, applying Jensen's inequality repeatedly results in the following upper bound for the  Lipschitz constant:
\begin{align*}
\frac{\Delta_{1:n}}{\tilde \Delta} &\leq \sqrt{n} \sqrt{\prod\limits_{g = 1}^n\left(2 +  16\hat\sigma_f^{-2}L_f^4 \Trcov(\mu)^2\right)} \\
& \leq  \sqrt{n} {\left(\sqrt{2} +  4\hat\sigma_f^{-1}L_f^{2} {\Trcov(\mu)}\right)^n} \\
& \leq \sqrt{n}2^{n-1} \left({2}^{n/2} +  4^n\hat\sigma_f^{-n}L_f^{2n} \Trcov(\mu)^{n}\right)\\
& \leq \sqrt{n}2^{n-1} \left({2}^{n/2} +  4^n\hat\sigma_f^{-n}L_f^{2n}\right)\left(1+  \Trcov(\mu)^{n}\right).
\end{align*}
\end{proof}

\begin{lemma}
\label{lem:distance_x_to_beta_tilde} Suppose Assumption \ref{assumption:sm_assumption}  holds. 
Let $\mu \in \Pp_2\left(\R^n\times\R^m\right)$. Then, for any $x\in\R^n, y,y^\star \in \R^m$ we have
$$
\left\|x -  \map^{\text{SM},\mu}_{y^\star}(x,y)\right\|_2 ^2
\leq  n2^{n-1}\bigl(1 + 4^nL_f^{4n}\hat\sigma_f^{-2n} \Trcov(\mu)^{2n}\bigr)
\left\|y - y^\star\right\|_2^2. 
$$
\end{lemma}

\begin{proof}
Let
\[
\theta_k \coloneqq \Cov_{\hat \sigma_f}\left( f_k( X, Y) \right)^{-1} \Cov\left(  f_k( X, Y),  X_k \right)
\]
for $(X, Y) \sim \mu$. Define $w =  \map^{\text{SM},\mu}_{y^\star}(x,y)$.  By definition from Equations \eqref{eq:sm_approximate_conditioning_recursion} and~\eqref{eq:alg_conditioning_transport_SM}, for all $k \geq 1$:
\begin{align*}
\left|x_k - w_k\right|^2 & = \left| \theta_k^\top (f_k(x,y) - f_k(w, y^\star))\right|^2 \\
&\leq L_f^2\left\|\theta_k\right\|_2^2\left(  \left\|x_{1:k-1} - w_{1:k-1}\right\|_2^2 +\|y - y^\star\|_2^2\right). 
\end{align*}
This shows inductively that 
$$
\left\|x_{1:k} - w_{1:k}\right\|_2^2  \leq \left\|y - y^\star\right\|_2^2\sum\limits_{\ell = 1}^k \prod\limits_{g = \ell}^k (1 + L_f^2\left\|\theta_g\right\|_2^2).
$$
We can further bound this using elementary inequalities and Equation \eqref{eq:theta_k_norm_bound} as 
\begin{align*}
\left\|x -  \map^{\text{SM},\mu}_{y^\star}(x,y)\right\|_2^2 & \leq \left\|x_{1:n} - w_{1:n}\right\|_2^2 \\
&\leq n\left\|y - y^\star\right\|_2^2 \prod\limits_{g = 1}^n (1 + L_f^2\left\|\theta_g\right\|_2^2) \\
&\leq n\left\|y - y^\star\right\|_2^2 (1 + 4L_f^4\hat\sigma_f^{-2} \Trcov(\mu)^2)^n \\
&\leq n2^{n-1}\bigl(1 + 4^nL_f^{4n}\hat\sigma_f^{-2n} \Trcov(\mu)^{2n}\bigr)
\left\|y - y^\star\right\|_2^2 
\end{align*}
\end{proof}

\begin{lemma}
\label{lem:frob_proj_lip}
Consider the Frobenius-projection defined in Equation~\eqref{eq:frob_proj} for a fixed $\upsilon \geq0$. Then for symmetric matrices $A, B \in \R^{u \times u}$, the projection $A \mapsto A_\upsilon$ is a Lipschitz function in the Frobenius norm. That is,
$$
    \|A_{\upsilon} -   B_{\upsilon}\|_F \leq \|A -   B\|_F.
$$
\end{lemma}

\begin{proof}
By definition, 
$$
A_\upsilon =\argmin_{\tilde A \succeq \upsilon I_{u}} f_{A}(\tilde A), \qquad f_A(\tilde A) \coloneqq\frac{1}{2}\left\|\tilde A - A\right\|_F^2.
 $$
Let $\langle\cdot,\cdot\rangle$ be the inner product inducing the Frobenius norm, i.e.,
\[
\langle A, B\rangle \coloneqq \sum_{i,j=1}^u A_{ij}B_{ij}.
\]
 This is a convex problem with $f_A$ convex and differentiable. Moreover, the set of matrices $\{\tilde A \in \R^{u \times u}: \tilde A \succeq \upsilon I_{u}\}$ where the optimization is performed is convex. Thus, by Equation~4.21 of \cite{boyd2004convex}, the KKT conditions give the following optimality criterion:
 $$
\langle\nabla_{\tilde A}f_{A} (A_\upsilon), \tilde A-A_\upsilon\rangle = \langle A_\upsilon - A, \tilde A-A_\upsilon\rangle \geq 0 \text{ for all }\tilde A\succeq \upsilon I_u.
 $$
Similarly,
$$
\langle B_\upsilon - B, \tilde B-B_\upsilon\rangle \geq 0 \text{ for all }\tilde B\succeq \upsilon I_u.
$$
In particular,
$$
\langle A_\upsilon - A, B_\upsilon-A_\upsilon\rangle \geq 0 ,\qquad \langle B_\upsilon - B, A_\upsilon-B_\upsilon\rangle \geq 0.
$$
Adding these together gives
$$
\langle A_\upsilon - B_\upsilon, B_\upsilon-A_\upsilon\rangle + \langle B - A, B_\upsilon-A_\upsilon\rangle \geq 0 \Leftrightarrow \langle B - A, B_\upsilon-A_\upsilon\rangle \geq \|A_\upsilon - B_\upsilon\|_F^2.
$$
The result then follows from Cauchy-Schwarz. 
\end{proof}

\begin{proposition}
\label{prop:sm_estimation_stability_map} 
Suppose Assumption \ref{assumption:sm_assumption}  holds.  Let  \( \mu, \nu \in \Pp_{2}\left(\R^n\times\R^m\right) \). Then, for any $y^\star,y \in \R^m, x \in \R^n$ the  mean-field stochastic map approximate conditioning  map satisfies
\begin{align*}
        &\left\| \map^{\text{SM},\mu}_{y^\star}(x,y)
        -   \map^{\text{SM},\nu}_{y^\star}(x,y)\right\|_2\\
        &\quad \leq  \tilde C_{\text{est}, \text{II}}
        \left(1 +  \Trcov( \mu)^{2n+1}  + \Trcov( \nu)^{2n+1}  \right)
        \left\|y^\star - y\right\|_2 \\
        &\qquad\cdot\left\|\Cov(g_\sharp\mu) - \Cov(g_\sharp\nu)\right\|_2
\end{align*}
for a constant $\tilde C_{\text{est}, \text{II}} =  \tilde C_{\text{est}, \text{II}}(n, b, L_f, \hat\sigma_f)$.
\end{proposition}

\begin{proof}
Let 
\begin{align*}
    \theta_k &\coloneqq \Cov_{\hat \sigma_f}\left( f_k( X, Y) \right)^{-1} \Cov\left(  f_k( X, Y),  X_k \right) \\
   \theta_k' &\coloneqq \Cov_{\hat \sigma_f}\left( f_k(  X',  Y') \right)^{-1} \Cov\left(  f_k(  X',  Y'),   X_k '\right) 
\end{align*}
for $(X, Y) \sim \mu, ( X',  Y') \sim \nu$. Further, let 
$$
w = \map^{\text{SM},\mu}_{y^\star}(x,y), \qquad  w' =  \map^{\text{SM},\nu}_{y^\star}(x,y).
$$
By definition from Equations\eqref{eq:sm_approximate_conditioning_recursion} and  \eqref{eq:alg_conditioning_transport_SM}, we have 
\begin{align*}
        \left|w_k - w_k'\right|^2 & = \left|\theta_k^\top\left(f_k(x,y) -f_k(w,y^\star)\right) - \theta_k'^\top\left(f_k(x,y) -f_k(w',y^\star)\right)  \right|^2\\
& = \left|\theta_k^\top\left(f_k(x,y) -f_k(x, y^\star) + f_k(x,y^\star)- f_k(w,y^\star)\right)\right. \\
&\left.- \theta_k'^\top\left(f_k(x,y) -f_k(x, y^\star) + f_k(x,y^\star) -f_k(w',y^\star)\right)  \right|^2\\
& = \left|(\theta_k-\theta_k')^\top\left(f_k(x,y) -f_k(x, y^\star)\right)\right. \\
&\qquad\left. + (\theta_k -\theta_k')^\top(f_k(x,y^\star)- f_k(w,y^\star))\right. \\
& \left.- \theta_k'^\top\left(f_k(w,y^\star) -f_k(w',y^\star)\right)  \right|^2\\
& \leq 3\Bigl(\|\theta_k -\theta_k'\|_2^2
\left(\|f_k(x,y) - f_k(x,y^\star)\|_2^2
+ \|f_k(x,y^\star) - f_k(w,y^\star)\|_2^2 \right)\\
&\qquad
+  \|\theta_k'\|_2^2\|f_k(w,y^\star) - f_k(w',y^\star)\|_2^2\Bigr)\\
& \leq 3L_f^2\Bigl(\|\theta_k -\theta_k'\|_2^2
\left(\|y-y^\star\|_2^2 + \|x - w\|_2^2\right)\\
&\qquad
+  \|\theta_k'\|_2^2\|w_{1:k-1}-w'_{1:k-1}\|_2^2\Bigr).
    \end{align*}
Plugging in Lemma \ref{lem:distance_x_to_beta_tilde} and defining $\Delta_{1:k}^2 = \|w_{1:k} - w_{1:k}'\|_2^2$, this yields 
\begin{align*}
\Delta_{1:k}^2
&\leq 3L_f^2\|\theta_k -\theta_k'\|_2^2
\left(1 + n2^{n-1}(1 + 4^nL_f^{4n}\hat\sigma_f^{-2n} \Trcov(\mu)^{2n}) \right)
\|y-y^\star\|_2^2  \\
&\quad+\left(1 +  3L_f^2\|\theta_k'\|_2^2\right)\Delta_{1:k-1}^2
\end{align*}
Further, applying Equation \eqref{eq:theta_k_norm_bound} to $\theta_k'$ with $(X',Y')\sim\nu$ shows 
\begin{align*}
\Delta_{1:k}^2
&\leq 3L_f^2\|\theta_k -\theta_k'\|_2^2
\left(1 + n2^{n-1}(1 + 4^nL_f^{4n}\hat\sigma_f^{-2n} \Trcov(\mu)^{2n}) \right)
\|y-y^\star\|_2^2  \\
&\quad+\left(1 + 12L_f^4\hat\sigma_f^{-2} \Trcov(\nu)^2\right)\Delta_{1:k-1}^2.
\end{align*}
Letting 
\begin{align*}
    C_1& =3L_f^2\max_{1\leq k\leq n}\|\theta_k -\theta_k'\|_2^2
    \left(1 + n2^{n-1}(1 + 4^nL_f^{4n}\hat\sigma_f^{-2n} \Trcov(\mu)^{2n}) \right)
    \|y-y^\star\|_2^2 \\
    C_2&  =  \left(1 + 12L_f^4\hat\sigma_f^{-2} \Trcov(\nu)^2\right),
\end{align*}
this shows inductively that 
$$
\Delta_{1:n}^2 \leq \sum\limits_{\ell = 1}^n C_2^{n-\ell} C_1 \leq n C_1 C_2^n.
$$
Now, we bound
\begin{align*}
    \|\theta_k -\theta_k'\|_2^2 &= \left\| \Cov_{\hat \sigma_f}\left( f_k( X, Y) \right)^{-1} \Cov\left(  f_k( X, Y),  X_k \right) \right.\\
    &\left.-  \Cov_{\hat \sigma_f}\left( f_k( X', Y') \right)^{-1} \Cov\left(  f_k( X', Y'),  X_k' \right)\right\|_2^2\\
    &\leq 2\left\| \Cov_{\hat \sigma_f}\left( f_k( X, Y) \right)^{-1} \Cov\left(  f_k( X, Y),  X_k \right) \right.\\
    &-\left.  \Cov_{\hat \sigma_f}\left( f_k( X, Y) \right)^{-1} \Cov\left(  f_k( X', Y'),  X_k' \right)\right\|_2^2 \\
    &+2\left\| \Cov_{\hat \sigma_f}\left( f_k( X, Y) \right)^{-1} \Cov\left(  f_k( X', Y'),  X_k' \right)\right.\\
    &-\left.  \Cov_{\hat \sigma_f}\left( f_k( X', Y') \right)^{-1} \Cov\left(  f_k( X', Y'),  X_k' \right)\right\|_2^2 \\
      &\leq 2\hat\sigma_f^{-2} \left\| \Cov\left(  f_k( X, Y),  X_k \right) -  \Cov\left(  f_k( X', Y'),  X_k' \right)\right\|_2^2 \\
    &+2\left\|\Cov\left(  f_k( X', Y'),  X_k' \right)\right\|_2^2\left\|\Cov_{\hat \sigma_f}\left( f_k( X, Y) \right)^{-1}\right.\\
    &\left. \cdot  \left(\Cov_{\hat \sigma_f}\left( f_k( X, Y) \right) -  \Cov_{\hat \sigma_f}\left( f_k( X', Y') \right)\right) \Cov_{\hat \sigma_f}\left( f_k( X', Y') \right)^{-1} \right\|_2^2 \\
      &\leq 2\hat\sigma_f^{-2}
      \left\| \Cov\left(  f_k( X, Y),  X_k \right)
      -  \Cov\left(  f_k( X', Y'),  X_k' \right)\right\|_2^2 \\
    &\quad+2\hat\sigma_f^{-4}
    \left\|\Cov\left(  f_k( X', Y'),  X_k' \right)\right\|_2^2\\
    &\qquad\cdot
    \left\|\Cov_{\hat \sigma_f}\left( f_k( X, Y) \right)
    -  \Cov_{\hat \sigma_f}\left( f_k( X', Y') \right) \right\|_2^2.
\end{align*}
Set
\[
D_g\coloneqq
\left\| \Cov\left(  g( X, Y)\right) -  \Cov\left(  g( X', Y') \right)\right\|_2.
\]
Invoking Lemma \ref{lem:frob_proj_lip}, we find the following results:
\begin{align*}
    &\left\|\Cov_{\hat \sigma_f}\left( f_k( X, Y) \right)
    -  \Cov_{\hat \sigma_f}\left( f_k( X', Y') \right) \right\|_2^2\\
    &\qquad \leq \left\|\Cov_{\hat \sigma_f}\left( f_k( X, Y) \right)
    -  \Cov_{\hat \sigma_f}\left( f_k( X', Y') \right) \right\|_F^2 \\
    &\qquad \leq \left\|\Cov\left( f_k( X, Y) \right) -  \Cov\left( f_k( X', Y') \right) \right\|_F^2 \\
    &\qquad \leq bD_g^2.\\
     &\left\| \Cov\left(  f_k( X, Y),  X_k \right)
     -  \Cov\left(  f_k( X', Y'),  X_k' \right)\right\|_2^2\\
     &\qquad \leq D_g^2.
\end{align*}
Moreover, 
$$
\left\|\Cov\left(  f_k( X', Y'),  X_k' \right)\right\|_2^2  \leq 2(L_f^2 + 1)^2\Trcov(\nu)^2.
$$
Thus, it follows that 
\begin{align*}
      \|\theta_k -\theta_k'\|_2^2
      &\leq C_3 \left(1 + \Trcov( \mu) + \Trcov(\nu)\right)^2 D_g^2
\end{align*}
for $C_3 \coloneqq 2\hat\sigma_f^{-2}+4b(L_f^2 + 1)^2\hat\sigma_f^{-4}$. Plugging into $\Delta_{1:n}^2 $, we showed
\begin{align*}
    \Delta_{1:n}^2
    &\leq n \left(1 + 12L_f^4\hat\sigma_f^{-2} \Trcov(\nu)^2\right)^n
    3L_f^2\|\theta_k -\theta_k'\|_2^2\\
    & \cdot \left(1 + n2^{n-1}(1 + 4^nL_f^{4n}\hat\sigma_f^{-2n} \Trcov(\mu)^{2n}) \right)\|y-y^\star\|_2^2  \\
    & \leq C_4\left(1 + \Trcov(\mu)^{2n}\right)
    \left(1 + \Trcov(\nu)^{2n}\right)
    \|y-y^\star\|_2^2\|\theta_k -\theta_k'\|_2^2\\
    &\leq  C_3C_4\left(1 + \Trcov(\mu)^{2n}\right)
    \left(1 + \Trcov(\nu)^{2n}\right)\\
    &\quad\cdot \|y-y^\star\|_2^2
    \left(1 + \Trcov( \mu) + \Trcov(\nu)\right)^2 \\
    & \cdot D_g^2 \\
     &\leq  C_5 \|y-y^\star\|_2^2
     \left(1 + \Trcov( \mu)^{2n+1} + \Trcov(\nu)^{2n+1}\right)^2\\
     &\cdot D_g^2,
\end{align*}
where we defined the intermediary constant 
\begin{align*}
    C_4 &\coloneqq 3L_f^2n2^{n-1}
    \left(1 + 12^nL_f^{4n}\hat\sigma_f^{-2n} \right)\\
    &\quad\cdot
    \left(1 + n2^{n-1}(1 + 4^nL_f^{4n}\hat\sigma_f^{-2n}) \right),
\end{align*}
and $C_5 = C_5(n,b,L_f,\hat\sigma_f)$ absorbs $C_3C_4$ and the polynomial bound in $\Trcov(\mu)$ and $\Trcov(\nu)$.
Taking the square root gives the final result.
\end{proof}

\ifpreprint
\renewcommand{\bibliofont}{\footnotesize}
\bibliographystyle{amsplain} 
\else
\bibliographystyle{imsart-number} 
\fi
\bibliography{bibliography} 

@incollection{mckean1967propagation,
  title={Propagation of Chaos for a Class of Non-linear Parabolic Equations},
  author={McKean, Jr., Henry P.},
  booktitle={Stochastic Differential Equations (Lecture Series in Differential Equations, Session 7, Catholic University, 1967)},
  pages={41--57},
  year={1967},
  publisher={Air Force Office of Scientific Research},
  address={Arlington, VA}
}

@article{baptista2024representation,
  title={On the Representation and Learning of Monotone Triangular Transport Maps},
  author={Baptista, Ricardo and Marzouk, Youssef and Zahm, Olivier},
  journal={Foundations of Computational Mathematics},
  volume={24},
  number={6},
  pages={2063--2108},
  year={2024},
  doi={10.1007/s10208-023-09630-x}
}

@misc{jorgensen2026exactaffineconditioninggaussians,
      title={Exact affine conditioning beyond {Gaussians}: a unique characterization of the ensemble {Kalman} update}, 
      author={Frederic J. N. Jorgensen and Youssef M. Marzouk},
      year={2025},
      eprint={2510.00158},
      archivePrefix={arXiv},
      primaryClass={math.ST},
      note={arXiv preprint arXiv:2510.00158},
      url={https://arxiv.org/abs/2510.00158}, 
}

@article{baptista2024conditional,
  title={Conditional Sampling with Monotone {GANs}: From Generative Models to Likelihood-Free Inference},
  author={Baptista, Ricardo and Hosseini, Bamdad and Kovachki, Nikola B. and Marzouk, Youssef},
  journal={SIAM/ASA Journal on Uncertainty Quantification},
  volume={12},
  number={3},
  pages={868--900},
  year={2024},
  doi={10.1137/23M1581546}
}

@article{crisan2002survey,
  title={A Survey of Convergence Results on Particle Filtering Methods for Practitioners},
  author={Crisan, Dan and Doucet, Arnaud},
  journal={IEEE Transactions on Signal Processing},
  volume={50},
  number={3},
  pages={736--746},
  year={2002},
  doi={10.1109/78.984773}
}

@incollection{bengtsson2008curse,
  title={Curse-of-Dimensionality Revisited: Collapse of the Particle Filter in Very Large Scale Systems},
  author={Bengtsson, Thomas and Bickel, Peter and Li, Bo},
  booktitle={Probability and Statistics: Essays in Honor of David A. Freedman},
  editor={Nolan, Deborah and Speed, Terry},
  volume={2},
  pages={316--334},
  year={2008},
  publisher={Institute of Mathematical Statistics},
  doi={10.1214/193940307000000518}
}

@book{law2015data,
  title={Data Assimilation: A Mathematical Introduction},
  author={Law, Kody and Stuart, Andrew and Zygalakis, Konstantinos},
  series={Texts in Applied Mathematics},
  volume={62},
  year={2015},
  publisher={Springer International Publishing},
  address={Cham},
  doi={10.1007/978-3-319-20325-6}
}

@article{bach2025learning,
  title={Learning enhanced ensemble filters},
  author={Bach, Eviatar and Baptista, Ricardo and Calvello, Edoardo and Chen, Bohan and Stuart, Andrew M.},
  journal={Journal of Computational Physics},
  volume={547},
  pages={114550},
  year={2026},
  doi={10.1016/j.jcp.2025.114550}
}

@book{reich2015probabilistic,
  title={Probabilistic Forecasting and {B}ayesian Data Assimilation},
  author={Reich, Sebastian and Cotter, Colin},
  year={2015},
  publisher={Cambridge University Press},
  doi={10.1017/CBO9781107706804}
}

@book{vershynin2018high,
  title={High-Dimensional Probability: An Introduction with Applications in Data Science},
  author={Vershynin, Roman},
  volume={47},
  year={2018},
  publisher={Cambridge University Press},
  doi={10.1017/9781108231596}
}

@article{fournier2015rate,
  title={On the Rate of Convergence in {Wasserstein} Distance of the Empirical Measure},
  author={Fournier, Nicolas and Guillin, Arnaud},
  journal={Probability Theory and Related Fields},
  volume={162},
  number={3--4},
  pages={707--738},
  year={2015},
  doi={10.1007/s00440-014-0583-7}
}

@article{parno2018transport,
  title={Transport Map Accelerated {Markov} chain {Monte Carlo}},
  author={Parno, Matthew and Marzouk, Youssef},
  journal={SIAM/ASA Journal on Uncertainty Quantification},
  volume={6},
  number={2},
  pages={645--682},
  year={2018},
  doi={10.1137/17M1134640}
}

@book{kallenberg1997foundations,
  title={Foundations of Modern Probability},
  author={Kallenberg, Olav},
  edition={2nd},
  year={2002},
  publisher={Springer},
  address={New York},
  doi={10.1007/978-1-4757-4015-8}
}

@article{hsu2012tail,
  title={A Tail Inequality for Quadratic Forms of Subgaussian Random Vectors},
  author={Hsu, Daniel and Kakade, Sham M. and Zhang, Tong},
  journal={Electronic Communications in Probability},
  volume={17},
  number={52},
  pages={1--6},
  year={2012},
  doi={10.1214/ECP.v17-2079}
}

@book{vershynin2025high,
  title={High-Dimensional Probability: An Introduction with Applications in Data Science},
  author={Vershynin, Roman},
  edition={2nd},
  year={2026},
  publisher={Cambridge University Press},
  url={https://www.math.uci.edu/~rvershyn/papers/HDP-book/HDP-2.pdf},
  note={Pre-publication version dated 19~May~2026, accessed 21~May~2026}
}

@book{boyd2004convex,
  title={Convex Optimization},
  author={Boyd, Stephen P. and Vandenberghe, Lieven},
  year={2004},
  publisher={Cambridge University Press},
  doi={10.1017/CBO9780511804441}
}

@article{evensen2003ensemble,
  title={The Ensemble {Kalman} Filter: Theoretical Formulation and Practical Implementation},
  author={Evensen, Geir},
  journal={Ocean Dynamics},
  volume={53},
  number={4},
  pages={343--367},
  year={2003},
  doi={10.1007/s10236-003-0036-9}
}

@article{spantini2022coupling,
  title={Coupling Techniques for Nonlinear Ensemble Filtering},
  author={Spantini, Alessio and Baptista, Ricardo and Marzouk, Youssef},
  journal={SIAM Review},
  volume={64},
  number={4},
  pages={921--953},
  year={2022},
  doi={10.1137/20M1312204}
}

@article{reich2013nonparametric,
  title={A nonparametric ensemble transform method for {B}ayesian inference},
  author={Reich, Sebastian},
  journal={SIAM Journal on Scientific Computing},
  volume={35},
  number={4},
  pages={A2013--A2024},
  year={2013},
  doi={10.1137/130907367}
}

@inproceedings{aljarrah2023optimal,
  title={Optimal Transport Particle Filters},
  author={Al-Jarrah, Mohammad and Hosseini, Bamdad and Taghvaei, Amirhossein},
  booktitle={2023 62nd {IEEE} Conference on Decision and Control ({CDC})},
  pages={6798--6805},
  year={2023},
  doi={10.1109/CDC49753.2023.10383337}
}

@article{zeng2024ensemble,
  title={Ensemble Transport Filter via Optimized Maximum Mean Discrepancy},
  author={Zeng, Dengfei and Jiang, Lijian},
  journal={Journal of Computational Physics},
  volume={548},
  pages={114582},
  year={2026},
  doi={10.1016/j.jcp.2025.114582}
}

@article{bishop2001adaptive,
  title={Adaptive Sampling with the Ensemble Transform {Kalman} Filter. Part {I}: Theoretical Aspects},
  author={Bishop, Craig H. and Etherton, Brian J. and Majumdar, Sharanya J.},
  journal={Monthly Weather Review},
  volume={129},
  number={3},
  pages={420--436},
  year={2001},
  doi={10.1175/1520-0493(2001)129\%3C0420:ASWTET\%3E2.0.CO;2}
}

@article{chaintron2022propagation,
  title={Propagation of chaos: A review of models, methods and applications. {I}. Models and methods},
  author={Chaintron, Louis-Pierre and Diez, Antoine},
  journal={Kinetic and Related Models},
  volume={15},
  number={6},
  pages={895--1015},
  year={2022},
  doi={10.3934/krm.2022017}
}

@article{calvello2024accuracy,
  title={Accuracy of the Ensemble {Kalman} Filter in the Near-Linear Setting},
  author={Calvello, Edoardo and Monmarch{\'e}, Pierre and Stuart, Andrew M. and Vaes, Urbain},
  journal={SIAM Journal on Numerical Analysis},
  volume={64},
  number={2},
  pages={391--429},
  year={2026},
  doi={10.1137/25M1732544}
}

@incollection{le2009large,
  title={Large Sample Asymptotics for the Ensemble {Kalman} Filter},
  author={Le Gland, Fran{\c{c}}ois and Monbet, Val{\'e}rie and Tran, Vu-Duc},
  booktitle={The Oxford Handbook of Nonlinear Filtering},
  editor={Crisan, Dan and Rozovskii, Boris},
  pages={598--631},
  year={2011},
  publisher={Oxford University Press},
  address={Oxford},
  note={Earlier version: INRIA Research Report RR-7014, 2009}
}

@article{mandel2011convergence,
  title={On the Convergence of the Ensemble {Kalman} Filter},
  author={Mandel, Jan and Cobb, Loren and Beezley, Jonathan D.},
  journal={Applications of Mathematics},
  volume={56},
  number={6},
  pages={533--541},
  year={2011},
  doi={10.1007/s10492-011-0031-2}
}

@article{kelly2014well,
  title={Well-Posedness and Accuracy of the Ensemble {Kalman} Filter in Discrete and Continuous Time},
  author={Kelly, David T. B. and Law, Kody J. H. and Stuart, Andrew M.},
  journal={Nonlinearity},
  volume={27},
  number={10},
  pages={2579--2603},
  year={2014},
  doi={10.1088/0951-7715/27/10/2579}
}

@article{tong2015nonlinear,
  title={Nonlinear Stability of the Ensemble {Kalman} Filter with Adaptive Covariance Inflation},
  author={Tong, Xin T. and Majda, Andrew J. and Kelly, David},
  journal={Communications in Mathematical Sciences},
  volume={14},
  number={5},
  pages={1283--1313},
  year={2016},
  doi={10.4310/CMS.2016.v14.n5.a5}
}

@article{tong2023localized,
  title={Localized Ensemble {Kalman} Inversion},
  author={Tong, Xin T. and Morzfeld, Matthias},
  journal={Inverse Problems},
  volume={39},
  number={6},
  pages={064002},
  year={2023},
  doi={10.1088/1361-6420/accb08}
}

@article{sanz2024long,
  title={Long-Time Accuracy of Ensemble {Kalman} Filters for Chaotic Dynamical Systems and Machine-Learned Dynamical Systems},
  author={Sanz-Alonso, Daniel and Waniorek, Nathan},
  journal={SIAM Journal on Applied Dynamical Systems},
  volume={24},
  number={3},
  pages={2246--2286},
  year={2025},
  doi={10.1137/24M1719232}
}

@article{al2024non,
  title={Non-asymptotic Analysis of ensemble {Kalman} Updates: Effective Dimension and Localization},
  author={Al-Ghattas, Omar and Sanz-Alonso, Daniel},
  journal={Information and Inference: A Journal of the IMA},
  volume={13},
  number={1},
  pages={iaad043},
  year={2024},
  doi={10.1093/imaiai/iaad043}
}

@article{takeda2024uniform,
  title={Uniform Error Bounds of the Ensemble Transform {Kalman} Filter for Chaotic Dynamics with Multiplicative Covariance Inflation},
  author={Takeda, Kota and Sakajo, Takashi},
  journal={SIAM/ASA Journal on Uncertainty Quantification},
  volume={12},
  number={4},
  pages={1315--1335},
  year={2024},
  doi={10.1137/24M1637192}
}

@article{de2018long,
  title={Long-Time Stability and Accuracy of the Ensemble {Kalman--Bucy} Filter for Fully Observed Processes and Small Measurement Noise},
  author={{de Wiljes}, Jana and Reich, Sebastian and Stannat, Wilhelm},
  journal={SIAM Journal on Applied Dynamical Systems},
  volume={17},
  number={2},
  pages={1152--1181},
  year={2018},
  doi={10.1137/17M1119056}
}

@article{hoang2021machine,
  title={Machine Learning-Based Conditional Mean Filter: A Generalization of the Ensemble {Kalman} Filter for Nonlinear Data Assimilation},
  author={Hoang, Truong-Vinh and Krumscheid, Sebastian and Matthies, Hermann G. and Tempone, Ra{\'u}l},
  journal={Foundations of Data Science},
  volume={5},
  number={1},
  pages={56--80},
  year={2023},
  doi={10.3934/fods.2022016}
}

@article{bao2025nonlinear,
  title={Nonlinear Ensemble Filtering with Diffusion Models: Application to the Surface Quasigeostrophic Dynamics},
  author={Bao, Feng and Chipilski, Hristo G. and Liang, Siming and Zhang, Guannan and Whitaker, Jeffrey S.},
  journal={Monthly Weather Review},
  volume={153},
  number={7},
  pages={1155--1169},
  year={2025},
  doi={10.1175/MWR-D-24-0069.1}
}

@incollection{marzouk2017sampling,
  title={Sampling via Measure Transport: An Introduction},
  author={Marzouk, Youssef and Moselhy, Tarek and Parno, Matthew and Spantini, Alessio},
  booktitle={Handbook of Uncertainty Quantification},
  editor={Ghanem, Roger and Higdon, David and Owhadi, Houman},
  pages={785--825},
  year={2017},
  publisher={Springer International Publishing},
  address={Cham},
  doi={10.1007/978-3-319-12385-1_23}
}

@article{weed2019sharp,
  title={Sharp Asymptotic and Finite-Sample Rates of Convergence of Empirical Measures in {Wasserstein} Distance},
  author={Weed, Jonathan and Bach, Francis},
  journal={Bernoulli},
  volume={25},
  number={4A},
  pages={2620--2648},
  year={2019},
  doi={10.3150/18-BEJ1065}
}

@article{koltchinskii2017concentration,
  title={Concentration Inequalities and Moment Bounds for Sample Covariance Operators},
  author={Koltchinskii, Vladimir and Lounici, Karim},
  journal={Bernoulli},
  volume={23},
  number={1},
  pages={110--133},
  year={2017},
  doi={10.3150/15-BEJ730}
}

@book{wainwright2019high,
  title={High-Dimensional Statistics: A Non-Asymptotic Viewpoint},
  author={Wainwright, Martin J.},
  volume={48},
  year={2019},
  publisher={Cambridge University Press},
  doi={10.1017/9781108627771}
}

@article{gordon1993novel,
  title={Novel Approach to Nonlinear/Non-{Gaussian} {Bayesian} State Estimation},
  author={Gordon, Neil J. and Salmond, David J. and Smith, Adrian F. M.},
  journal={IEE Proceedings F (Radar and Signal Processing)},
  volume={140},
  number={2},
  pages={107--113},
  year={1993},
  doi={10.1049/ip-f-2.1993.0015}
}

@incollection{doucet2001introduction,
  title={An Introduction to Sequential {Monte Carlo} Methods},
  author={Doucet, Arnaud and de Freitas, Nando and Gordon, Neil},
  booktitle={Sequential {Monte Carlo} Methods in Practice},
  editor={Doucet, Arnaud and de Freitas, Nando and Gordon, Neil},
  pages={3--14},
  year={2001},
  publisher={Springer},
  address={New York},
  doi={10.1007/978-1-4757-3437-9_1}
}

@incollection{sznitman2006topics,
  title={Topics in Propagation of Chaos},
  author={Sznitman, Alain-Sol},
  booktitle={{\'E}cole d'{\'E}t{\'e} de Probabilit{\'e}s de Saint-Flour {XIX}---1989},
  editor={Hennequin, Paul-Louis},
  series={Lecture Notes in Mathematics},
  volume={1464},
  pages={165--251},
  year={1991},
  publisher={Springer},
  address={Berlin},
  doi={10.1007/BFb0085169}
}

@inproceedings{kac1956foundations,
  title={Foundations of Kinetic Theory},
  author={Kac, Mark},
  booktitle={Proceedings of the Third Berkeley Symposium on Mathematical Statistics and Probability, 1954--1955},
  editor={Neyman, Jerzy},
  volume={3},
  pages={171--197},
  year={1956},
  publisher={University of California Press},
  address={Berkeley and Los Angeles}
}

@article{del2018stability,
  title={On the Stability and the Uniform Propagation of Chaos Properties of Ensemble {Kalman--Bucy} Filters},
  author={Del Moral, Pierre and Tugaut, Julian},
  journal={The Annals of Applied Probability},
  volume={28},
  number={2},
  pages={790--850},
  year={2018},
  doi={10.1214/17-AAP1317}
}

@article{law2016deterministic,
  title={Deterministic Mean-Field Ensemble {Kalman} Filtering},
  author={Law, Kody J. H. and Tembine, Hamidou and Tempone, Ra{\'u}l},
  journal={SIAM Journal on Scientific Computing},
  volume={38},
  number={3},
  pages={A1251--A1279},
  year={2016},
  doi={10.1137/140984415}
}

@article{kwiatkowski2015convergence,
  title={Convergence of the Square Root Ensemble {Kalman} Filter in the Large Ensemble Limit},
  author={Kwiatkowski, Evan and Mandel, Jan},
  journal={SIAM/ASA Journal on Uncertainty Quantification},
  volume={3},
  number={1},
  pages={1--17},
  year={2015},
  doi={10.1137/140965363}
}

@article{bishop2023mathematical,
  title={On the Mathematical Theory of Ensemble (Linear-{Gaussian}) {Kalman--Bucy} Filtering},
  author={Bishop, Adrian N. and Del Moral, Pierre},
  journal={Mathematics of Control, Signals, and Systems},
  volume={35},
  number={4},
  pages={835--903},
  year={2023},
  doi={10.1007/s00498-023-00357-2}
}

@article{dudley1969speed,
  title={The Speed of Mean {Glivenko--Cantelli} Convergence},
  author={Dudley, Richard Mansfield},
  journal={The Annals of Mathematical Statistics},
  volume={40},
  number={1},
  pages={40--50},
  year={1969},
  doi={10.1214/aoms/1177697802}
}

@article{dobric1995asymptotics,
  title={Asymptotics for Transportation Cost in High Dimensions},
  author={Dobri{\'c}, Vladimir and Yukich, Joseph E.},
  journal={Journal of Theoretical Probability},
  volume={8},
  number={1},
  pages={97--118},
  year={1995},
  doi={10.1007/BF02213456}
}

@article{del2017stability,
  title={On the Stability and the Uniform Propagation of Chaos of a Class of Extended Ensemble {Kalman--Bucy} Filters},
  author={Del Moral, Pierre and Kurtzmann, Aline and Tugaut, Julian},
  journal={SIAM Journal on Control and Optimization},
  volume={55},
  number={1},
  pages={119--155},
  year={2017},
  doi={10.1137/16M1087497}
}

@article{larsson2024concentration,
  title={On Concentration of the Empirical Measure for Radial Transport Costs},
  author={Larsson, Martin and Park, Jonghwa and Wiesel, Johannes},
  journal={Stochastic Processes and their Applications},
  volume={178},
  pages={104466},
  year={2024},
  doi={10.1016/j.spa.2024.104466}
}

@article{ramgraber2023Bensemble,
  title={Ensemble transport smoothing. {P}art {II}: nonlinear updates},
  author={Ramgraber, Maximilian and Baptista, Ricardo and McLaughlin, Dennis and Marzouk, Youssef},
  journal={Journal of Computational Physics: X},
  volume={17},
  pages={100133},
  year={2023},
  publisher={Elsevier}
}

\end{document}